\documentclass[12pt,dvipsnames]{amsart}
\usepackage[utf8]{inputenc}
\usepackage[T1]{fontenc}
\usepackage[margin=1.0in,marginparwidth=0.8in]{geometry}
\usepackage{amsfonts, amsmath, amsthm, amssymb, amscd, amsrefs}
\usepackage{hyperref}
\usepackage{graphicx}
\usepackage{tikz}
\usepackage[all]{xy}
\usepackage{tikz}
\usepackage{color}
\usepackage{mathtools} 
\usepackage{wrapfig}
\usepackage{enumerate}
\usepackage{mathrsfs}
\usepackage{soul}

\numberwithin{equation}{section}
\theoremstyle{definition}

\newtheorem{dfn}[equation]{Definition}
\newtheorem{exa}[equation]{Example}
\newtheorem{rmk}[equation]{Remark}

\newtheorem{prp}[equation]{Proposition}
\newtheorem{thm}[equation]{Theorem}
\newtheorem{lem}[equation]{Lemma}
\newtheorem{cor}[equation]{Corollary}

\newtheorem{nta}[equation]{Notation}

\newtheorem{cns}[equation]{Construction}
\newcommand{\trm}{\textrm}

\newcommand{\K}{{\mathbb K}} 

\newcommand{\C}{{\mathbb C}}
\newcommand{\R}{{\mathbb R}}
\newcommand{\Q}{{\mathbb Q}}
\newcommand{\Z}{{\mathbb Z}}
\newcommand{\N}{{\mathbb N}}
\newcommand{\D}{{\mathbb D}}

\newcommand{\cA}{\mathcal A}  \newcommand{\cC}{\mathcal C}
\newcommand{\cD}{\mathcal D}

 \newcommand{\cM}{\mathcal M} \newcommand{\cN}{\mathcal N}
 \newcommand{\cP}{\mathcal P}

\newcommand{\f}{\mathfrak}

\newcommand{\al}{\alpha}
\newcommand{\be}{\beta}
     \newcommand{\Ga}{{\Gamma}}
\newcommand{\de}{\delta}

    \newcommand{\La}{{\Lambda}}
\newcommand{\te}{\theta}    \newcommand{\Te}{{\Theta}}
\newcommand{\si}{\sigma}     
\newcommand{\om}{\omega}     \newcommand{\Om}{{\Omega}}
\newcommand{\vph}{\varphi}

  \newcommand{\Hom}{\textrm{Hom}}

\newcommand{\tr}{\textrm{tr}}
\newcommand{\ch}{\mathrm{ch}}
\newcommand{\Ka}{\mathrm{Ka}}

\newcommand{\Ch}{\mathrm{Ch}}

\newcommand{\CS}{\textrm{CS}}

\newcommand{\KCS}{\mathrm{KCS}}
\newcommand{\chk}{\text{ch}^{\text{Ka}}}

\newcommand{\odd}{\textrm{odd}}
\newcommand{\even}{\textrm{even}}

\newcommand{\ev}{\textrm{ev}}
\newcommand\DGA{\mathbf{DGA}}
\newcommand\Alg{\mathbf{Alg}}
\newcommand\Fgp{\mathbf{Fgp}}
\newcommand\FgpD{\mathbf{Fgp}_{\mathbb{D}}}
\newcommand\FgpuD{\mathbf{Fgp}_{\mathbb{D}^{\mathrm{u}}}}
\newcommand\Bun{\mathbf{Bun}}
\newcommand\BunD{\mathbf{Bun}_{\nabla\!\!\!\!\nabla}}
\newcommand\Lfb{\mathbf{Lfb}}
\newcommand\LfbD{\mathbf{Lfb}_{\mathcal{D}\!\!\!\!\mathcal{D}}}
\newcommand\loc{\mathrm{loc}}
\newcommand\con{\mathrm{con}}
\newcommand{\Man}{\mathbf{Man}}
\newcommand{\Ring}{\mathbf{Ring}}
\newcommand{\Mod}{\mathbf{Mod}}

\newcommand{\Ab}{\mathbf{Ab}}
\newcommand{\ab}{\mathrm{ab}}
\newcommand{\dR}{\mathrm{dR}}
\newcommand{\op}{\mathrm{op}}

\newcommand\cocartext[1]{\overset{\multimap}{#1}}
\newcommand{\we}{\wedge}      
\newcommand{\tsr}{\otimes}
\newcommand\ctsr{\operatorname{\widehat{\otimes}}}
\newcommand\dtsr{\operatorname{\widehat{\otimes}}}

\newcommand\dun{d^{\mathrm{u}}}
\newcommand\univf{\Omega^{\mathrm{u}}}
\newcommand\uhdr{H^{\mathrm{udR}}}
\newcommand\uchk{\mathrm{ch}^{\mathrm{uKa}}}
\newcommand{\dsum}{\oplus}    \newcommand{\Dsum}{\bigoplus}
      
\newcommand{\srl}{\stackrel}
 \newcommand{\ra}{\rightarrow} \newcommand{\lra}{\longrightarrow}

\newcommand{\wtl}{\widetilde}

\newcommand{\na}{\nabla}
\newcommand{\Cinf}{C^\infty}
\newcommand{\bl}{\bullet}
\newcommand{\isom}{\cong}     
\newcommand{\bmat}{\begin{pmatrix}}  \newcommand{\emat}{\end{pmatrix}}
\newcommand{\barr}{\begin{array}}  \newcommand{\earr}{\end{array}}
\newcommand{\bcd}{\begin{CD}}  \newcommand{\ecd}{\end{CD}}
\newcommand{\beq}{\begin{equation}\begin{aligned}}  \newcommand{\eeq}{\end{aligned}\end{equation}}
\newcommand{\beqs}{\begin{equation*}\begin{aligned}}  \newcommand{\eeqs}{\end{aligned}\end{equation*}}

\newcommand\define[1]{\emph{\textbf{#1}}}
\makeatletter
\newcommand{\dwtl}[1]{{%
  \mathpalette\double@widetilde{#1}%
}}
\newcommand{\double@widetilde}[2]{%
  \sbox\z@{$\m@th#1\widetilde{#2}$}%
  \ht\z@=.55\ht\z@
  \widetilde{\box\z@}%
}
\makeatother

\title{Noncommutative Differential $K$-theory}
\author{Byungdo Park}
\address{Department of Mathematics Education, Chungbuk National University, Cheongju 28644, Republic of Korea}
\email{byungdo@cbnu.ac.kr}
\thanks{This research has received funding from the European Research Council (ERC) under the European Union's Horizon 2020 research and innovation program (QUASIFT grant agreement 677368) as well as from the National Research Foundation of Korea (NRF) grant funded by the Korea government (MSIT) (No. 2020R1G1A1A01008746).}

\author{Arthur J. Parzygnat}
\address{Institut des Hautes \'Etudes Scientifiques (IHES), Le Bois-Marie 35 rte de Chartres 91440 Bures-sur-Yvette, France}
\email{parzygnat@ihes.fr}

\author{Corbett Redden}
\address{Department of Mathematics, LIU Post, Long Island University, 720 Northern Blvd, Brookville NY 11548, USA}
\email{corbett.redden@liu.edu}

\author{Augusto Stoffel}
\address{University of Greifswald, Walther-Rathenau-Str.~47, 17489 Greifswald, Germany}
\email{astoffel@mpim-bonn.mpg.de}

\makeatletter
\@namedef{subjclassname@2020}{%
  \textup{2020} Mathematics Subject Classification}
\makeatother

\date{\today}


\begin{document}
\sloppy
\maketitle
\begin{abstract}
We introduce a differential extension of algebraic $K$-theory of an algebra using Karoubi's Chern character. In doing so, we develop a necessary theory of secondary transgression forms as well as a differential refinement of the smooth Serre--Swan correspondence. Our construction subsumes the differential $K$-theory of a smooth manifold when the algebra is complex-valued smooth functions. Furthermore, our construction fits into a noncommutative differential cohomology hexagon diagram.
\end{abstract}

\tableofcontents

\section{Introduction}
Differential $K$-theory is a topic in mathematics that combines homotopy theoretic aspects of topological $K$-theory and differential geometric aspects of the de~Rham complex in a homotopy-theoretically consistent way. It is a natural home for studying gauge fields and their secondary characteristic classes. Many interesting areas, such as stable homotopy theory, the theory of characteristic classes, index theory, and superstring theory, thereby coalesce in this topic. It has been actively investigated from the viewpoint of abstract framework through the theory of $\infty$-categories, especially via sheaves of spectra and differential cohesion \cites{BNV,Sch}, as well as from the computational point of view by developing spectral sequences \cites{GS1,GS2}.

An interest in differential $K$-theory was kindled by attempts to topologically classify $D$-brane charges and Ramond--Ramond fields in type II string theories \cites{MM,Wi,F}. It was in Hopkins and Singer \cite{HS} that the first construction of a differential extension of any exotic cohomology theory was given. Following their work, several different constructions appeared, which are index-theoretic \cite{BS}, homotopy-theoretic \cite{BNV}, and geometric \cites{Kl, FL, SS, TWZ1, TWZ2}. While each model elucidates certain aspects of differential $K$-theory better than the other, these models are equivalent upon verifying uniqueness axioms \cite{BS2} or by direct comparisons \cite{Pa2}. However, we do not yet have a codification that is algebraic.

Furthermore, there has been a need for a formulation of the Minasian--Moore formula for $D$-branes in noncommutative spacetimes (see Szabo \cite{Sz} and references therein). To parallel the proposal to classify Ramond--Ramond fields in Freed \cite{F} in the noncommutative setup, one would need a \emph{noncommutative twisted differential $K$-theory}. However, unlike in the case of smooth manifolds, very little is known about noncommutative twisted $K$-theory other than the noncommutative twisted Chern character introduced by Mathai and Stevenson \cite{MS}. Noncommutative differential $K$-theory has not been well-explored either yet. Although there has been a homotopy-theoretic construction of differential algebraic $K$-theory as a sheaf of spectra over the site of smooth manifolds \cite{BG} whose underlying homotopy invariant sheaf evaluated at a point is an algebraic $K$-theory spectrum, the context and setup only provide a partially noncommutative differential $K$-theory, since one is still working with a site of \emph{commutative} spaces (manifolds).

The purpose of this paper is to fulfill the demand for a noncommutative differential $K$-theory. Namely, we construct differential extensions of $K_0$ of an algebra using universal noncommutative differential forms \`a la Karoubi \cite{Ka1} as an analogue of the de~Rham complex of a smooth manifold. Secondly, we set up a cycle map to relate our construction to the differential $K$-theory of a smooth manifold when the given algebra is the algebra of complex-valued smooth functions on the manifold.%
\footnote{Connes' Chern character $K_0(A)\ra HP_{\text{even}}(A)$ (as opposed to Karoubi's Chern character) is an alternative noncommutative counterpart of the usual Chern character $K^0(X)\ra H_{\dR}^{\text{even}}(X;\Q)$ when $A$ is the algebra of smooth or continuous functions on $X$~\cite{Co}. This counterpart is justified by the Serre--Swan correspondence and the Hochschild--Kostant--Rosenberg theorem, which say that both Chern characters are equivalent. Therefore, an alternative perspective, which we do not pursue in this work, is to develop a differential extension of $K$-theory with Connes' Chern character instead.}
Furthermore, the flat subgroup of our construction specializes to the multiplicative $K$-theory or $K$-theory with $\R/\Z$-coefficients, which are of interest in index theory and operator algebras \cites{AAS, Lott, Ba}.

The advantage of our approach is that our construction uses explicit cycle data, in parallel with existing geometric constructions of differential $K$-theory of a manifold~\cites{Kl,FL,SS,Pa2}. We develop the theory of Karoubi--Chern--Simons forms and give a detailed account of differential refinements of the smooth Serre--Swan correspondence. Furthermore, we establish a natural equivalence between our noncommutative differential $K$-theory and the differential $K$-theory of smooth manifolds when the underlying algebra is complex-valued smooth functions (Theorem \ref{THM.cycle.map.1}). To the best of our knowledge, this paper is the first to consider differential extensions of (co)homological invariants to the noncommutative framework.

This paper is organized as follows. Section~\ref{SEC.prelim} consists of background material on differential $K$-theory on manifolds, connections and curvatures on a module, Karoubi's Chern characters, and universal noncommutative differential forms. Section~\ref{SEC.KCS} is the technical core of this paper wherein we develop a noncommutative counterpart of the theory of Chern--Simons forms and obtain the primary and the secondary transgression formulas.
In Section~\ref{SEC.NCDKT.construction}, we construct a noncommutative differential $K$-theory and prove that the construction subsumes the differential $K$-theory of a manifold when the algebra is complex-valued smooth functions. Finally, we verify that our noncommutative differential $K$-group fits into the noncommutative analogue of the differential cohomology hexagon diagram.
Appendix~\ref{SEC.DSS} contains a differential refinement of smooth Serre--Swan correspondence, which was crucial for constructing a cycle map in Section~\ref{SEC.NCDKT.construction} (cf.\ Lemma~\ref{lem:CStoKCS}).
Appendix~\ref{sec:fibrations} provides a review of Grothendieck fibrations, which is used to prove the naturality of noncommutative Chern characters.

\textbf{Acknowledgements.} We thank Thomas Schick,  Ralf Meyer, Alex Atsushi Takeda, and Jae~Min Lee for helpful conversations as well as the Korea Institute for Advanced Study (KIAS) for the hospitality during our visits. BP thanks the Hausdorff Research Institute for Mathematics (HIM) for the support at the early stage of this work. A part of this research was done while AJP was at the University of Connecticut.

\section{Preliminaries}\label{SEC.prelim}
In this section, we collect background materials and setup notations we will be using throughout this paper. We review differential $K$-theory on smooth manifolds (Section \ref{SEC.FLK}), Karoubi's noncommutative Chern characters (Sections \ref{SEC.connection.curvature.kch}--\ref{SEC.kch}), and universal noncommutative differential forms (Section \ref{SEC.NC.differential.forms}).

\subsection{Review of differential $K$-theory} \label{SEC.FLK} In this subsection, we shall review differential $K$-theory by introducing a geometric model of even differential $K$-theory due to Klonoff \cite{Kl} as well as Freed and Lott \cite{FL}. We will use this construction in Section \ref{SEC.cycle.maps.in.karoubi.diff.k}. In this subsection, $X$ is always a compact smooth manifold. In the rest of the paper, $\Omega_{\dR}^{\bullet}(X)$ denotes the differential graded algebra of smooth complex-valued de~Rham forms on $X$ (whether compact or not) with differential $d$.

\begin{dfn}
\begin{enumerate}
\item
A \define{$\widehat{K}^0$-generator} on $X$ is a triple $(E,\na,\om)$ consisting of a smooth vector bundle $E\ra X$ with connection $\na$ and $\om\in \Om^{\odd}_{\dR}(X)/\text{Im}(d)$, where $\text{Im}(d)$ denotes the image of $d$.

\item
Two $\widehat{K}^0$-generators on $X$, $(E_0,\na_0,\om_0)$ and $(E_1,\na_1,\om_1)$ are \define{$\CS$-equivalent} if there is a smooth complex vector bundle $G\to X$ with connection $\na$ and a vector bundle isomorphism $\phi:E_0\dsum G\ra E_1\dsum G$ such that $\CS(\na_0\dsum\na,\na_1\dsum\na)=\om_0-\om_1\mod \text{Im}(d)$. Here $\CS(\na_0\dsum\na,\na_1\dsum\na)$ denotes the Chern--Simons form of any path of connections joining $\na_0\dsum\na$ and $\phi^*(\na_1\dsum\na)$ (cf.\ Simons and Sullivan \cite[Proposition 1.1, p.583]{SS}).

\item
The direct sum $(E_0,\na_0,\om_0) \dsum (E_1,\na_1,\om_1)$ is defined by $(E_0\dsum E_1, \na_0\dsum\na_1, \om_0+\om_1)$.
\end{enumerate}
\end{dfn}

It is a consequence of the secondary transgression formula of Chern forms (see Simons and Sullivan \cite[Proposition 1.1, p.583]{SS}) that the relation defined in (2) is an equivalence relation. We will use the notation $\sim_{\CS}$ to denote the $\CS$-equivalence relation and $[\cdot]_{\CS}$ to denote an equivalence class of a $\widehat{K}^0$-generator on $X$. The operation $\dsum$ is well-defined on the set of $\CS$-equivalence classes. We thus get a commutative monoid $(\f{M}(X),\dsum)$ consisting of $\CS$-equivalence classes $[(E,\na,\om)]_{\CS}$ endowed with the operation $\dsum$ on $\CS$-equivalence classes.

\begin{dfn}[Freed and Lott \cite{FL}, Klonoff \cite{Kl}] \label{DFN.FLK} The \define{even differential $K$-theory} of $X$, denoted by $\widehat{K}^0(X)$, is the Grothendieck group of the commutative monoid $(\f{M}(X),\dsum)$.
\end{dfn}

\begin{rmk} The even differential $K$-group $\widehat{K}^0(X)$ is known to be naturally isomorphic to other known models (see \cites{Kl,Pa2}).
\end{rmk}

\subsection{Connections and curvatures on a module}\label{SEC.connection.curvature.kch}
In this and the following subsection, we shall review some basic material for defining noncommutative Chern characters by Karoubi~\cite{Ka1}.
\begin{nta}\label{NTA.karoubi.notation.dga} Let $A$ be an algebra, unital but not necessarily commutative, over a field $\K$, and let $\Alg$ be the category of $\K$-algebras and $\K$-algebra homomorphisms. Unless specified otherwise, the notation $\Om_\bl(A)$ in this section means a differential graded algebra (henchforth DGA) over $\K$ satisfying $\Om_0(A)=A$. Any such DGA $\Omega_{\bullet}(A)$ satisfying this condition will be called a \define{DGA on top of $A$}.
Let $\DGA$ be the category of DGA's over $\K$ and maps of DGA's (chain maps preserving the algebra structure).
Its differential will always be denoted by $d$ unless specified otherwise.
The abelianization of $\Om_\bl(A)$ is the chain complex ${\Om}_n(A)_{\ab}:=\left({\Om}_n(A)/[{\Om}_\bl(A),{\Om}_\bl(A)]_n\right)$, where $[{\Om}_\bl(A),{\Om}_\bl(A)]_n$ is the $\K$-submodule
generated by elements of the form $[\om_i,\om_j]=\om_i\cdot \om_j-(-1)^{ij}\om_j\cdot \om_i$, with $\om_i\in {\Om}_i(A)$, $\om_i\in {\Om}_j(A)$, and $i+j=n$. Note that ${\Om}_\bl(A)_{\ab}$ is a chain complex with the differential inherited from the differential of $\Om_\bl(A)$ (but ${\Om}_\bl(A)_{\ab}$ is not necessarily a DGA). We will use the notation $\Fgp(A)$ for the category of finitely generated projective right $A$-modules and $A$-module maps. For any category $\cC$, we will write $\cC^{\text{iso}}$ to denote the subcategory of $\cC$ consisting of the same objects as $\cC$ and morphisms in $\cC^{\text{iso}}$ are isomorphisms in  $\cC$. For example, $\Fgp^{\text{iso}}(A)$ is the category of finitely generated projective right $A$-modules and $A$-module isomorphisms.
\end{nta}

\begin{rmk} \label{RMK.karoubi.notation.dga} We caution the reader that the notation $\Om_\bl(A)$ means an \emph{arbitrary} DGA whose zeroth degree is $A$, as in Karoubi \cite[p.~9]{Ka1}. For example when $X$ is a smooth manifold and $A=\Cinf(X;\C)$, the algebra of smooth $\C$-valued functions on $X$, the de~Rham complex of exterior differential forms $\Om^\bl_{\dR}(X)$ is just one example of a DGA on top of $A$. See Section \ref{SEC.NC.differential.forms} for more examples. Although $\Omega_{\bullet}(A)$ may arise as the application of a functor $\Omega_{\bl}:\Alg\to\DGA$ on an algebra $A$ (specifically when discussing naturality of certain constructions---cf.\ Definition~\ref{defn:DGAontopofAlg} for example), we will only use such additional structure on occasion. Nevertheless, we will write $\Omega_{\bullet}(A)$ even if the DGA on top of $A$ does not arise from some functorial construction on algebras.
\end{rmk}

\begin{dfn}[Karoubi~\cite{Ka1}]\label{DFN.NC.deRham.homology.generic}
Let $A$ be an algebra over $\mathbb{K}$ and let $\Omega_{\bullet}(A)$ be a DGA on top of $A$. The \define{noncommutative de Rham homology group} of $\Omega_{\bullet}(A)$  in degree $n$, denoted by ${H}_n^{\dR}(A)$, is the $n^{\text{th}}$ homology group of the complex ${\Om}_\bl(A)_{\ab}$.
\end{dfn}

\begin{rmk} Although not explicit in the notation, the definition of  ${H}_n^{\dR}(A)$ depends on a choice of a DGA ${\Om}_\bl(A)$ on top of $A$.
If there is a possibility for confusion, the notation $H_{n}^{\dR}(\Omega_{\bl}(A))$ will be used instead.
\end{rmk}

\begin{dfn} Given $M\in \Fgp(A)$ and  a DGA $\Omega_{\bullet}(A)$ on top of $A$, a \define{connection} $D$ on $M$ (with respect to $\Omega_{\bl}(A)$) is a $\K$-linear map $D: M\ra M\tsr_A \Om_1(A)$ satisfying the Leibniz rule $D(m a)=D(m) a+ m\otimes d(a)$ for all $m\in M$ and $a\in A$.
\end{dfn}

A connection as defined above always exists (see~\cite[p.12]{Ka1}).

\begin{exa}
If $A=C^{\infty}(X)$ for some smooth manifold $X$, and if $E\to X$ is a vector bundle, let $\Gamma(E)$ denote  the $A$-module of sections. Then a connection $\nabla:\Gamma(E)\to\Gamma(E)\otimes_{A}\Omega^{1}_{\dR}(X)$ reproduces the usual notion of a connection on a vector bundle (cf.\ \cite[pg~185 and pg~192]{Mo01}). To set notation, if $V$ is a vector field, then $\nabla_{V}s$ is the covariant derivative of a section $s$ in the direction $V$ and is obtained by composing $\nabla$ with evaluation of forms on $V$.
\end{exa}

\begin{dfn}\label{DFN.Fgp.conn.}
Let $A$ be a $\K$-algebra and let $\Omega_{\bullet}(A)$ be a DGA on top of $A$. The category $\FgpD(A)$
consists of the following data. Its objects are pairs $(M,D)$, where $M\in\Fgp(A)$ and  $D$ is a connection on $M$. The hom-set $\FgpD(A)\left((M_0,D_0),(M_1,D_1)\right)$ consists of $A$-linear maps $M_0\srl{f}\ra M_1$ such that the following diagram commutes:
\[\xymatrix{ M_0 \ar[d]^f \ar[rr]^(0.4){D_0} && M_0\tsr_{A} \Om_1(A)\ar[d]^{f \tsr_A \text{id}}\\
    M_1 \ar[rr]^(0.4){D_1} && M_1 \tsr_A \Om_1(A)
   } \]
\end{dfn}

\begin{exa}\label{EXA.examples.of.connection}
Let $E_{i}\xrightarrow{\pi_{i}}X$, $i\in\{0,1\}$, be two smooth vector bundles over a smooth manifold $X$, let $A:=C^{\infty}(X)$, let $M_{i}:=\Gamma(E_{i})$ be the modules of sections, let $D_{i}:\Gamma(E_i)\to\Gamma(E_{i})\otimes_{A}\Om^{1}_{\dR}(X)$ be connections, and let $\varphi:E_{0}\to E_{1}$ be a bundle map fixing the base, but not necessarily a fibrewise isomorphism, such that $D_{1}(\varphi s)=(\varphi\otimes\mathrm{id})D_{0}s$ (a special case of this is a connection-preserving bundle isomorphism). Then $\Gamma(E_{0})\to\Gamma(E_{1})$ sending $s$ to $\varphi\circ s$ is a morphism in $\FgpD(X):=\FgpD(A)$.
\end{exa}

To allow for a noncommutative generalization of morphism to bundle maps that are isomorphic on fibres and with a varying base, we will introduce the extension of scalars for modules (also called change of rings---see~\cite[Section 11.52]{Nes} for example) as well as functorial DGAs on top of algebras.

\begin{dfn}
Let $A$ and $B$ be $\K$-algebras. Let $N$ be a $B$-module. Let $\psi:B\to A$ be an algebra map. The \define{extension of scalars} $A$-module (of $N$ along $\psi$) is the right $A$-module defined as
$N\otimes_{B}A:=N\otimes_{\K}A/_{\sim}$
with the equivalence relation generated by  $nb\otimes_{\K} a\sim n\otimes_{\K}\psi(b)a$ and with the right action given by $(n\otimes_{B}a)a':=n\otimes_{B}(aa')$.
Because of the many different tensor products used, tensor products will often have subscripts $\K, A,$ and $B$ to help the reader distinguish them.
\end{dfn}

\begin{rmk}
Given an algebra map $\psi:B\to A$, the algebra $A$ can be viewed as a $(B,A)$-bimodule with the left $B$ action given by first applying $\psi$ and the using the multiplication on $A$, and the right $A$ action given by the multiplication on $A$. In this way, $N\otimes_{B}A$ is just the tensor product of a right $B$-module and a $(B,A)$-bimodule resulting in a right $A$-module. The extension of scalars is functorial (see Remark~\ref{rmk:grothendieckextensionscalars} and Appendix~\ref{sec:fibrations} for more details).
\end{rmk}

\begin{dfn}\label{dfn:cocartesianmorphism}
Let $A$ and $B$ be $\K$-algebras. Let $M$ be an $A$-module and $N$ a $B$-module. A \define{cocartesian morphism} from $(B,N)$ to $(A,M)$ is a pair $(\psi,\Psi)$ consisting of an algebra map $\psi:B\to A$ and a linear map $\Psi:N\to M$ such that $\Psi(nb)=\Psi(n)\psi(b)$ for all $n\in N$ and $b\in B$ and such that the canonical map%
\footnote{It is easy to see this map is well-defined because $\Psi(nb)a=\Psi(n)\psi(b)a$. It is immediate that $\cocartext{\Psi}$ is also right $A$-module morphism.}
\[
\begin{split}
N\otimes_{B}A&\xrightarrow{\cocartext{\Psi}}M\\
n\otimes_{B}a&\mapsto\Psi(n)a
\end{split}
\]
is an isomorphism of $A$-modules.
\end{dfn}

\begin{rmk}\label{rmk:grothendieckextensionscalars}
This remark is a justification for the terminology of ``cocartesian morphism'' in Definition~\ref{dfn:cocartesianmorphism} for readers familiar with the notion of Grothendieck fibrations, which are reviewed in Appendix~\ref{sec:fibrations}.
Using the same notation as in Definition~\ref{dfn:cocartesianmorphism}, the collection of objects being pairs $(A,M)$, with $A$ a $\K$-algebra and $M$ a finitely generated projective%
\footnote{The finitely generated projective assumption is not needed for the comments that follow in this remark, but it is needed for comparison to the smooth case if one wants an equivalence on certain subcategories.}
$A$-module, forms a category $\Fgp$. A morphism in $\Fgp$ from $(B,N)$ to $(A,M)$ is defined to be a pair $(\psi,\Psi)$ such that $\Psi(nb)=\Psi(n)\psi(b)$ for all inputs. This defines a Grothendieck opfibration $\Fgp\to\Alg$, where a cocartesian lift uses the extension of scalars. The cocartesian morphisms in this opfibration correspond to the cocartesian morphisms as we have defined under an appropriate fibred equivalence of opfibrations.
\end{rmk}

\begin{dfn}\label{defn:DGAontopofAlg}
A \define{DGA on top of $\Alg$} (or some subcategory) is a functor $\Omega_{\bl}:\Alg\to\DGA$ such that $\Omega_{0}:=\mathrm{ev}_{0}\circ\Omega_{\bl}=\mathrm{id}_{\Alg}$, where $\mathrm{ev}_{0}:\DGA\to\Alg$ is the restriction to the zeroth degree.
\end{dfn}

Having a DGA on top of $\Alg$ (instead of just a DGA on top of some specific algebra) is what will allow us to discuss naturality with the differential. Namely, if $\psi:B\to A$ is an algebra map and $\Omega_{\bl}$ a DGA on top of $\Alg$, then $\Omega_{\bl}(\psi):\Omega_{\bl}(B)\to\Omega_{\bl}(A)$ (often denoted abusively by $\psi$) is, in particular, a chain map. We would not necessarily get this from simply assigning a DGA on top of $A$ and $B$ if it was not done in a functorial manner.
Moving on, if one restricts to the subcategory whose objects are algebras of the form $C^{\infty}(X)$ for smooth manifolds $X$ and whose class of morphisms are pullbacks along smooth functions, then the DGA of de~Rham forms gives such a functor. Another example will be given in the context of universal noncommutative differential forms later.

\begin{lem}
\label{lem:pullbackconnection}
Let $\psi:B\to A$ be a map of algebras, let $N\in\Fgp(B)$, and let $\Omega_{\bl}:\Alg\to\DGA$ be a DGA on top of $\Alg$.
\begin{enumerate}
\item
Then $N\otimes_{B}\Omega_{k}(A):=N\otimes_{\K}\Omega_{k}(A)/_{\sim}$, where the equivalence relation is generated by $nb\otimes_{\K}\omega\sim n\otimes_{\K}\psi(b)\omega$, is an $A$-module (not necessarily finitely-generated) with the action given by $(n\otimes_{B}\omega)a:=n\otimes_{B}(\omega a)$. Equivalently, $N\otimes_{B}\Omega_{k}(A)$ is canonically isomorphic to $(N\otimes_{B}A)\otimes_{A}\Omega_{k}(A)$ (the tensor product of the right $A$-module $N\otimes_{B}A$ with the $(A,A)$-bimodule $\Omega_{k}(A)$).
\item
Furthermore, the map $\mathrm{id}\otimes_{\K}\psi:N\otimes_{\K}\Omega_{k}(B)\to N\otimes_{\K}\Omega_{k}(A)$ descends to a well-defined map $\mathrm{id}\otimes_{B}\psi:N\otimes_{B}\Omega_{k}(B)\to N\otimes_{B}\Omega_{k}(A)$. Note that the notation is slighly abusive and the map $\psi$ here is technically $\Omega_{k}(\psi)$.%
\footnote{This same abuse of notation occurs in differential geometry when one uses the same notation for the pullback of differential forms as for the pullback of smooth functions.}
\item
If, in addition, $D:N\to N\otimes_{B}\Omega_{1}(B)$ is a connection on $N$, then the map
\[
\begin{split}
N\otimes_{\K}A&\to N\otimes_{B}\Omega_{1}(A)\\
n\otimes a&\mapsto\big((\mathrm{id}\otimes_{B}\psi)(D(n))\big)a+n\otimes da
\end{split}
\]
descends to a connection
$D_{\psi}:N\otimes_{B}A\to N\otimes_{B}\Omega_{1}(A)\cong(N\otimes_{B}A)\otimes_{A}\Omega_{1}(A)$.
\end{enumerate}
\end{lem}

\begin{proof}
{\color{white}{you found me!}}
\begin{enumerate}
\item
The canonical isomorphism is a property of the tensor product.
\item
This follows from $nb\otimes_{\K}\psi(\omega)\sim n\otimes_{\K} \psi(b)\psi(\omega)=n\otimes_{\K}\psi(b\omega)$.
\item
The fact that $D_{\psi}$ is well-defined follows from $\big((\mathrm{id}\otimes_{B}\psi)(D(n))\big)\psi(b)+n\otimes d\psi(b)=(\mathrm{id}\otimes_{B}\psi)\big(D(n)b+n\otimes db\big)=(\mathrm{id}\otimes_{B}\psi)\big(D(nb)\big).$ \qedhere
\end{enumerate}
\end{proof}

\begin{dfn}\label{dfn:inducedconnection}
The connection $D_{\psi}:N\otimes_{B}A\to N\otimes_{B}\Omega_{1}(A)\cong(N\otimes_{B}A)\otimes_{A}\Omega_{1}(A)$ on $N\otimes_{B}A$ from Lemma~\ref{lem:pullbackconnection},
given a connection $D:N\to N\otimes_{B}\Omega_{1}(B)$ on $N$, is called the \define{induced connection}. If $(B,N)\xrightarrow{(\psi,\Psi)}(A,M)$ is a cocartesian morphism, then the composite
\[
D_{\Psi}:=\left(M\xrightarrow{\cocartext{\Psi}^{-1}}N\otimes_{B}A\xrightarrow{D_{\psi}}(N\otimes_{B}A)\otimes_{A}\Omega_{1}(A)\xrightarrow{\cocartext{\Psi}\otimes\mathrm{id}}M\otimes_{A}\Omega_{1}(A)\right)
\]
is also called the \define{induced connection} (the map $\cocartext{\Psi}$ is defined in Definition~\ref{dfn:cocartesianmorphism}).
Furthermore, if $\varphi:M\to L$ is an isomorphism of $A$-modules and $\nabla$ is a connection on $L$, the \define{pullback connection} $\varphi^*\nabla$ on $M$ is the composite $M\srl{\varphi}\ra L \srl{\nabla}\ra L\tsr_{A}\Om_1(A) \xrightarrow{\varphi^{-1}\tsr\mathrm{id}} M\tsr_{A}\Om_1(A)$.
\end{dfn}

\begin{rmk}
Equivalently, the pullback connection $\varphi^*\nabla$ can be defined as the induced connection on $M$ associated to a cocartesian morphism of the form $(A,M)\xrightarrow{(\mathrm{id}_{A},\varphi)}(A,L)$. Indeed, for such a cocartesian morphism, $\varphi:M\to L$ is an isomorphism because $\cocartext{\varphi}$ equals the composite $M\otimes_{A}A\cong M\xrightarrow{\varphi}L$.
This shows that the two notions of induced connection and pullback connection agree when the algebras are fixed.
\end{rmk}

The following example illustrates that the pullback of smooth vector bundles and their associated pullback connections are described by the extension of scalars and the induced connection construction described in Lemma~\ref{lem:pullbackconnection}.
\begin{exa}
Let $\pi:F\to Y$ be a smooth vector bundle and let $\phi:X\to Y$ be a smooth map of manifolds. Let $\rho:E\to X$ denote the pullback bundle and let $\Phi:E\to F$ denote the associated map of total spaces.
Set $A:=C^{\infty}(X)$, $B:=C^{\infty}(Y)$, $M:=\Gamma(E),$ $N:=\Gamma(F)$, and $\psi:B\to A$ to be the pullback map associated with $\phi$.
The map $\Psi:N\to M$ sends a section $s$ of $F$ to the section of $E$ by composing the appropriate morphisms in the diagram
\[
\xy0;/r.25pc/:
(-10,7.5)*+{E}="1";
(10,7.5)*+{F}="2";
(-10,-7.5)*+{X}="3";
(10,-7.5)*+{Y}="4";
{\ar"1";"2"^{\Phi}};
{\ar"3";"4"_{\phi}};
{\ar"1";"3"_{\rho}};
{\ar"2";"4"_{\pi}};
{\ar@/_1.0pc/@{-->}"4";"2"_{s}};
{\ar@/_1.0pc/@{-->}"3";"1"_{\Psi(s)}};
\endxy
\]
using the fact that $\Phi$ is a fibrewise isomorphism. Namely, $\Psi(s)$ is the section that, when evaluated at $x\in X$, is given by $\Phi_{x}^{-1}(s(\phi(x)))$, where $\Phi_{x}:E_{x}\to F_{\phi(x)}$ is the restriction of $\Phi$ to the fibre over $x$. Therefore, the assignment $\cocartext{\Psi}:N\otimes_{B}A\to M$, given by (cf.\ Definition~\ref{dfn:cocartesianmorphism}) $s\otimes_{B}f\mapsto\Psi(s)f$, is well-defined (by fibrewise linearity of $\Phi$) and establishes an isomorphism of $A$-modules (cf.\ Appendix~\ref{sec:fibrations}).
Thus, the extension of scalars module is an algebraic analogue of the pullback bundle.
Now, given a connection $D$ on $\pi$, let $\nabla$ denote the pullback connection on $\rho$. Then the diagram
\[
\xy0;/r.25pc/:
(-20,7.5)*+{N\otimes_{B}A}="NBA";
(20,7.5)*+{(N\otimes_{B}A)\otimes_{A}\Omega^{1}_{\dR}(X)}="NBAY";
(-20,-7.5)*+{M}="M";
(20,-7.5)*+{M\otimes_{A}\Omega^{1}_{\dR}(X)}="MY";
{\ar"NBA";"NBAY"^(0.37){D_{\psi}}};
{\ar"M";"MY"_(0.4){\nabla}};
{\ar"NBA";"M"_{\cocartext{\Psi}}^{\cong}};
{\ar"NBAY";"MY"^{\cocartext{\Psi}\otimes_{A}\mathrm{id}}_{\cong}};
\endxy
\]
commutes.
Thus, the induced connection coincides with the pullback connection (up to the canonical isomorphism relating the extension of scalars to the module of sections of the pullback bundle, as described in Appendix~\ref{sec:fibrations}).
\end{exa}

\begin{dfn}\label{defn:extendingconnection} Let  $(M,D)\in\FgpD(A)$. The connection $D$ on $M$ is extended to a degree $+1$ $\K$-linear map $D:M\otimes_{A}\Omega_{\bullet}(A)\ra M\otimes_{A}\Omega_{\bullet}(A)$
uniquely determined by $D(s\tsr \om)=Ds\cdot \om+s\tsr d\om$ for all $s\in M$ and $\om\in\Om_\bl(A)$. The \define{curvature} of $D$ is the $\Omega_{\bl}(A)$-linear endomorphism $R_D:=D^2:M\tsr_A \Om_{\bl}(A)\ra M\tsr_A \Om_{\bl}(A)$.
\end{dfn}

The following lemma describes in what sense the curvature of the induced connection equals the pushforward of the curvature.

\begin{lem}\label{lem:incucedcurvature}
Let $A$ and $B$ be $\K$-algebras, let $\Omega$ be a DGA on top of $\Alg$, let $N$ be a $B$-module with connection $D:N\to N\otimes_{B}\Omega_{1}(B)$ whose curvature is $R_{D}$. Let $\psi:B\to A$ be an algebra map and let $D_{\psi}:N\otimes_{B}A\to N\otimes_{B}\Omega_{1}(A)$ be the associated connection (from Lemma~\ref{lem:pullbackconnection} and Definition~\ref{dfn:inducedconnection}) with curvature $R_{D_{\psi}}$. Then the curvatures are related via
\[
R_{D_{\psi}}=
\mathfrak{E}_{\psi}(R_{D}),
\]
where
$\mathfrak{E}_{\psi}:\mathrm{End}_{\Omega_{\bl}(B)}\big(N\otimes_{B}\Omega_{\bl}(B)\big)\to\mathrm{End}_{\Omega_{\bl}(A)}\big(N\otimes_{B}\Omega_{\bl}(A)\big)$
is the function sending an arbitrary $\Omega_{\bl}(B)$-linear endomorphism $L$ of $N\otimes_{B}\Omega_{\bl}(B)$ to the unique $\Omega_{\bl}(A)$-linear extension%
\footnote{The map $\mathfrak{E}_{\psi}(L)$ can equivalently be obtained via the extension of scalars functor (cf.\ Example~\ref{exa:extensionscalarsgrothendieck}) taking the algebra map $\psi:\Omega_{\bl}(B)\to\Omega_{\bl}(A)$ to the functor $\mathfrak{E}_{\psi}$ that sends a morphism $L:N\otimes_{B}\Omega_{\bl}(B)\to N\otimes_{B}\Omega_{\bl}(B)$ of $\Omega_{\bl}(B)$-modules to a morphism $\mathfrak{E}_{\psi}(L):N\otimes_{B}\Omega_{\bl}(A)\to N\otimes_{B}\Omega_{\bl}(A)$ of $\Omega_{\bl}(A)$-modules.
This also explains why the notation $\mathfrak{E}_{\psi}$ is used.
}
$\mathfrak{E}_{\psi}(L)$ of the composite $N\hookrightarrow N\otimes_{B}\Omega_{\bl}(B)\xrightarrow{L} N\otimes_{B}\Omega_{\bl}(B)\xrightarrow{\mathrm{id}\tsr_{B}\psi} N\tsr_{B}\Omega_{\bl}(A)$ to an $\Omega_{\bl}(A)$-linear endomorphism of $N\tsr_{B}\Omega_{\bl}(A)$, i.e.\ $n\otimes\omega$ gets sent to $\big((\mathrm{id}_{N}\otimes_{B}\psi)L(n\otimes1_{B})\big)\omega$.
\end{lem}

\begin{proof}
Fix $n\in N$. By the uniqueness of the $\Omega_{\bl}(A)$-linear extension property of $\mathfrak{E}_{\psi}(R_{D})$ mentioned in the statement of the lemma, it suffices to show $D_{\psi}^{2}(n\otimes_{B}1_{A})=(\mathrm{id}\otimes_{B}\psi)D^2(n)$. Setting $D(n)=:\sum_{\alpha}n_{\alpha}\otimes_{B}\omega_{\alpha}\in N\otimes_{B}\Omega_{1}(B)$, this follows from
\[
\begin{split}
D_{\psi}^{2}(n\otimes_{B}1_{A})
&=D_{\psi}\big((\mathrm{id}\otimes_{B}\psi)D(n)\big)
=\sum_{\alpha}D_{\psi}\big(n_{\alpha}\otimes_{B}\psi(\omega_{\alpha})\big)\\
&=\sum_{\alpha}\Big(D_{\psi}(n_{\alpha}\otimes_{B}1_{A})\psi(\omega_{\alpha})+n_{\alpha}\otimes_{B}d\psi(\omega_{\alpha})\Big)\\
&=\sum_{\alpha}\Big((\mathrm{id}\otimes_{B}\psi)\big(D(n_{\alpha})\big)\psi(\omega_{\alpha})+n_{\alpha}\otimes_{B}\psi(d\omega_{\alpha})\Big)\\
&=\sum_{\alpha}(\mathrm{id}\otimes_{B}\psi)\Big(D(n_{\alpha})\omega_{\alpha}+n_{\alpha}\otimes_{B}d\omega_{\alpha}\Big)\\
&=(\mathrm{id}\otimes_{B}\psi)D^2(n)
\end{split}
\]
by two applications of the definition of $D_{\psi}$ from Definition~\ref{dfn:inducedconnection} and Lemma~\ref{lem:pullbackconnection}. The second term in the third line follows from the comments after Definition~\ref{defn:DGAontopofAlg}.
\end{proof}

\subsection{Traces of module endomorphisms}\label{SEC.traces}
We will soon take the trace of the curvature in order to define the Karoubi--Chern form. Therefore, we first review traces in a more general context~\cite[Section~1.16]{Ka1}.


\begin{dfn}
\label{defn:Fdual}
Let $R$ be a unital (not necessarily commutative) ring and $F$ a finitely generated projective right $R$-module. Let $F^* := \Hom_R(F, R)$ be the \define{dual} consisting of right $R$-linear maps. $F^*$ is given the structure of a left $R$-module upon setting $(r\mu)(x):=r\mu(x)$ for all $r\in R$, $\mu\in F^*$, and $x\in F$.
\end{dfn}

Given any ring $R$ and any two right $R$-modules $F$ and $G,$ with $F$ finitely generated projective, the map
\[
\begin{split}
G\otimes_{R}F^*&\rightarrow\mathrm{Hom}_{R}(F,G)\\
y\otimes\mu&\mapsto\big(x\mapsto y\mu(x)\big)
\end{split}
\]
determines an isomorphism of abelian groups (under addition) because the finitely generated projective assumption on the modules is what guarantees this map is an isomorphism.

\begin{dfn}\label{DFN.trace.with.Ab}
Let $R$ be a ring and $F$ a finitely generated projective right $R$-module.  The \define{trace} map is the composite
$\tr:\text{End}_R(F)\isom F\tsr_R F^*\srl{\text{eval}}\lra R/[R,R]$, where $\mathrm{eval}(x\otimes\mu)$ takes the equivalence class of $\mu(x)$.
\end{dfn}

Note that the trace is defined to take values in $R/[R,R]$ (as opposed to $R$). There would otherwise be an ambiguity due to the identity $xr\tsr\mu=x\tsr r\mu$. The first suggests $\mu(xr)=\mu(x)r$, while the latter suggests $(r\mu)(x)=r\mu(x)$. Because it is valued in $R/[R,R]$, this trace map is guaranteed to satisfy $\tr(TS)=\tr(ST)$ for
$S\in\Hom_{R}(F,G)$ and $T\in\Hom_{R}(G,F)$, where $G$ is another finitely generated projective right $R$-module.

\begin{exa}
\label{exa.trace.cyclic}
Take $F=R^{n}$. In this case, we can write elements of $\mathrm{End}_{R}(F)$ as $n\times n$ matrices with coefficients in $R$. The standard unit vectors $\{e_{i}\}$ define the standard matrices $\{E_{ij}\}$, where $E_{ij}$ corresponds to $e_{i}\otimes e_{j}^*$ and takes the $j$-th entry of any element in $R^{n}$ and places it in the $i$-th entry of a new vector whose other elements are all zero. Given $T\in\mathrm{End}_{R}(F)$, let $\{a_{ij}\in R\}$ be the unique elements satisfying $T(e_{j})=\sum_{i=1}^{n}e_{i}a_{ij}$.
We can express $T$ in matrix form as $\sum_{i,j}E_{ij}a_{ij}$ or, using the isomorphism above, as
$\sum_{i,j}e_{i}a_{ij}\otimes e_{j}^*\equiv \sum_{i,j}e_{i}\otimes a_{ij}e_{j}^*$.
Since $e_{j}^*(e_{i})=\delta_{ij}$, the trace of $T$ is given by the equivalence class associated to $\sum_{i,j}e_{j}^*(e_{i}a_{ij})=\sum_{i,j}e_{j}^*(e_{i})a_{ij}=\sum_{i,j}\delta_{ij}a_{ij}=\sum_{i}a_{ii}.$ One might notice that in this case, the expressions $\sum_{i,j}e_{j}^*(e_{i})a_{ij}$ and $\sum_{i,j}a_{ij}e_{j}^*(e_{i})$ are actually equal as elements of $R$. However, if we wanted this to be a basis-independent expression, then this would require that $\tr$ satisfies the property $\tr(T)=\tr(UTU^{-1})$ for all isomorphisms $U\in\mathrm{End}_{R}(F)$. This property is guaranteed by choosing a cyclicly invariant formula. Such a formula necessarily factors through $R/[R,R]$. We will find that this cyclic property is useful for computations involving the Karoubi--Chern forms later.
\end{exa}

\begin{exa}\label{ex:ringofevenforms}
Let $\Om_{\bl}(A)$ be a DGA on top of $A$. The subset $\Om_{\even}(A)$ consisting of all even degree elements in $\Om_{\bl}(A)$ is a ring.
By taking $R:=\Om_{\even}(A)\equiv\bigoplus_{k}\Omega_{2k}(A)$ and $F:=M\tsr_A\Om_{\even}(A)$,
the integral powers of the curvature map $D^2$ restrict to right $R$-linear endomorphisms of $F$.
The corresponding trace of $(D^2)^k$ provides an element $\tr(D^{2k})$ in $R/[R,R]$ with representatives living in $\Omega_{2k}(A)$.
Note that $[R,R]$ only contains commutators of even forms, while $[\Omega_{\bullet}(A),\Omega_{\bullet}(A)]_{2k}$ consists of forms of the form $[\omega_{i},\omega_{j}]$ with $i+j=2k$ since $i$ and $j$ could both be odd. Hence, there is an injection $[R,R]\hookrightarrow\bigoplus_{k}[\Omega_{\bullet}(A),\Omega_{\bullet}(A)]_{2k}$, which induces a surjection $R/[R,R]\to(\Om_{\bl}(A)_{\ab})_{\even}:=\bigoplus_{k}\left(\Omega_{2k}(A)/[\Omega_{\bullet}(A),\Omega_{\bullet}(A)]_{2k}\right)$. Therefore,
to obtain an element of $(\Om_{\bl}(A)_{\ab})_{\even}$, we apply this canonical map
$R/[R,R]\ra  (\Om_{\bl}(A)_{\ab})_{\even}$.
By a slight abuse of notation, we also denote this element by $\tr(D^{2k})\in(\Om_{\bl}(A)_{\ab})_{\even}$.
\end{exa}

\begin{lem}\label{lem:ringmapsandtrace}
Let $\psi:R\to S$ be a ring homomorphism and let $F$ be an $R$-module. Then $\psi$ extends to a ring homomorphism $\psi:R/[R,R]\to S/[S,S]$ and
\[
\xy0;/r.25pc/:
(-15,7.5)*+{\mathrm{End}_{R}(F)}="1";
(-15,-7.5)*+{\mathrm{End}_{S}(F\otimes_{R}S)}="2";
(15,7.5)*+{R/[R,R]}="3";
(15,-7.5)*+{S/[S,S]}="4";
{\ar"1";"2"_{\mathfrak{E}_{\psi}}};
{\ar"1";"3"^{\mathrm{tr}}};
{\ar"2";"4"_(0.6){\mathrm{tr}}};
{\ar"3";"4"^{\psi}};
\endxy
\]
is a commutative diagram of abelian groups. Here, $\mathfrak{E}_{\psi}$ is the map defined in a similar way to Lemma~\ref{lem:incucedcurvature}, namely as the map sending $L$, an $R$-endomorphism of $F$, to the unique $S$-linear extension%
\footnote{The map $\mathfrak{E}_{\psi}(L)$ can also be viewed as the extension of scalars functor $\mathfrak{E}_{\psi}$ associated to the map $R\xrightarrow{\psi}S$ acting on the $R$-module homomorphism $L$ to produce an $S$-module homomorphism (cf.\ Remark~\ref{rmk:grothendieckextensionscalars} and Example~\ref{exa:extensionscalarsgrothendieck}).}
$\mathfrak{E}_{\psi}(L)$ of the composite $F\xrightarrow{L}F\xrightarrow{\cong}F\otimes_{R}R\xrightarrow{\mathrm{id}\otimes_{R}\psi}F\otimes_{R} S$.
\end{lem}

\begin{proof}
The fact that $\psi$ extends to a ring homomorphism on abelianizations is a standard fact. Setting
\begin{alignat*}{3}
F&\xrightarrow{i_{\psi}}F\otimes_{R}S
&\qquad\qquad F^*&\xrightarrow{i_{\psi}^*}(F\otimes_{R}S)^*\\
f&\mapsto f\otimes1_{S}
&\qquad\qquad\mu&\mapsto\Big(f\otimes s\mapsto\psi\big(\mu(f)\big)s\Big),
\end{alignat*}
one obtains $i_{\psi}(fr)=i_{\psi}(f)\psi(r)$ and $i_{\psi}^*(r \mu)=\psi(r)i_{\psi}^*(\mu)$ for all inputs ($i^*_{\psi}$ is not the usual dual map). Furthermore, the tensor product of these two abelian group homomorphisms induces a well-defined map $i_{\psi}\otimes i_{\psi}^*:F\otimes_{R}F^*\to(F\otimes_{R}S)\otimes_{S}(F\otimes_{R}S)^*$ such that
\[
\xy0;/r.25pc/:
(-25,7.5)*+{F\otimes_{R}F^*}="1";
(25,7.5)*+{\mathrm{End}_{R}(F)}="2";
(-25,-7.5)*+{(F\otimes_{R}S)\otimes_{S}(F\otimes_{R}S)^*}="3";
(25,-7.5)*+{\mathrm{End}_{S}(F\otimes_{R}S)}="4";
{\ar"1";"2"^{\cong}};
{\ar"1";"3"_{i_{\psi}\otimes i_{\psi}^*}};
{\ar"2";"4"^{\mathfrak{E}_{\psi}}};
{\ar"3";"4"^(0.56){\cong}};
\endxy
\]
is a commutative diagram of abelian groups, as one can check on applying the definitions to an element $f\otimes\mu\in F\otimes_{R}F^*$. Furthermore, by definition of the evaluation map, the diagram
\[
\xy0;/r.25pc/:
(-25,7.5)*+{F\otimes_{R}F^*}="1";
(-25,-7.5)*+{(F\otimes_{R}S)\otimes_{S}(F\otimes_{R}S)^*}="2";
(25,7.5)*+{R/[R,R]}="3";
(25,-7.5)*+{S/[S,S]}="4";
{\ar"1";"2"_{i_{\psi}\otimes i_{\psi}^*}};
{\ar"1";"3"^{\mathrm{eval}}};
{\ar"2";"4"_(0.6){\mathrm{eval}}};
{\ar"3";"4"^{\psi}};
\endxy
\]
commutes. Combining these two commutative diagrams and using the definition of the trace proves the lemma.
\end{proof}

\subsection{Karoubi's Chern character}\label{SEC.kch}

In this subsection, we define Karoubi's Chern character (compare Karoubi \cite[p.16]{Ka1}) and review some of its properties that we will use later.

\begin{dfn}[Karoubi's Chern character] Let $A$ be a unital algebra over a field $\K$ containing $\Q$ and let $\Omega_{\bullet}(A)$ be a DGA on top of $A$. Let $(M,D)\in \FgpD(A)$ with $D$ a connection on $M$ with respect to $\Omega_{\bullet}(A)$.
The $k^{\text{th}}$ \define{Karoubi's Chern character form} of $D$ is (cf.\ Example~\ref{ex:ringofevenforms}) $$\chk_k(D):=\frac{1}{k!}\tr(D^{2k})\in (\Om_{\ab}(A))_{2k}.$$
The \define{total Karoubi's Chern character form} is $\chk(D):=\sum_{k=0}^\infty \chk_k(D)$.
\end{dfn}

The sum in $\chk(D)$ is necessarily finite by the finitely generated assumption on $M$. Furthermore, there is some abuse of notation since the dependence of $\chk_{k}$ on the DGA $\Omega_{\bl}(A)$ is not made explicit.

\begin{exa}\label{exa:KaroubiChernformforGrassman}
Let $M\in\textrm{Fgp}(A)$, let $i:M\hookrightarrow A^{m}$ be an embedding, and set $M_{m}(A)$ to be the algebra of $m\times m$ matrices with entries in $A$. Let $p\in M_{m}(A)$ satisfy $p^2=p$ and $\mathrm{Im}(p)=\mathrm{Im}(i)$ (we will often conflate $p$ being viewed as a matrix with coefficients in $A$ as well as an $A$-endomorphism of $A^m$). Set $\mathcal{M}:=\mathrm{Im}(p)$.
Let $\tilde{D}:\mathcal{M}\to\mathcal{M}\tsr_{A}\Om_{1}(A)$ be the connection on $\mathcal{M}$ given by%
\footnote{Note that since $ds\in A^{m}\otimes_{A}\Omega_{1}(A)$, the $p\cdot ds$ here should technically be written as $(p\otimes 1_{\cA})\cdot s$ when viewed as a matrix equation or $(p\otimes_{A}\mathrm{id}_{\Omega_{1}(A)})(ds)$ when viewed as an operator acting on $ds$. In either case, $p$ has been extended in a natural way to act on $A^{m}\otimes_{A}\Omega_{1}(A)$ by acting trivially on the right factor. We will often exclude this additional identity to avoid cumbersome notation.}
$\tilde{D}(s):=p\cdot ds$ and
let $D:M\to M\tsr_{A}\Om_{1}(A)$ be the connection on $M$ given by%
\footnote{Just as $p$ has been extended to $A^{m}\otimes_{A}\Omega_{\bl}(A)$, the inclusion $i$ has naturally been extended to a map $M\otimes_{A}\Omega_{\bl}(A)\to A^{m}\otimes_{A}\Omega_{\bl}(A)$. This map is also denoted by $i$, as opposed to the more accurate $i\otimes\mathrm{id}_{\Omega_{\bl}(A)}$.}
$D:=i|_{\mathcal{M}}^{-1}\circ\tilde{D}\circ i$.
A simple calculation then shows $\tilde{D}^{2}(s)=p\cdot dp\cdot dp\cdot s$ for all $s\in\mathcal{M}$, exhibiting the right $A$-linearity of $\tilde{D}^{2}$ explicitly~\cite[Section~1.15 and~1.18]{Ka1}. In fact, $\tilde{D}^{2k}(\sigma)=p\cdot(dp)^{2k}\cdot \sigma$ for all $\sigma\in\mathcal{M}\tsr_{A}\Omega_{\bl}(A)$ so that $\tilde{D}^{2k}$ is  $\Omega_{\bl}(A)$-linear. Hence,
\[
\chk(D)=\chk\left(i|_{\mathcal{M}}^{-1}\circ\tilde{D}\circ i\right)
=\chk(\tilde{D})
=\sum_{k=0}^\infty \frac{1}{k!}\tr(p\cdot(dp)^{2k}),
\]
where we have used the cyclicity of the trace in the second identity, which is allowed since
\[
D^{2}=
i|_{\mathcal{M}}^{-1}\circ p\circ d\circ \underbrace{i\circ i|_{\mathcal{M}}^{-1}\circ p}_{=p}\circ d \circ i
=i|_{\mathcal{M}}^{-1}\circ\tilde{D}^{2}\circ i.
\]

As we will see later, for certain choices of DGA on top of $A$, every connection will be of this form and the Karoubi--Chern form will be expressible in this manner. This allows us to perform several calculations using ideas from linear algebra adapted to the setting where the ground ring is noncommutative. In particular, if $D$ is expressible in two ways via different projections $p$ or $q$ (possibly with embeddings into different freely generated modules), then $\tr(p\cdot dp^{2k})=\tr(q\cdot dq^{2k})$,
which shows that the Karoubi--Chern form expressed in terms of projectors is invariant.
\end{exa}

\begin{lem}\label{lem:directsumconnections}
If $D:M\to M\otimes_{A}\Omega_{1}(A)$ and $D':M'\to M'\otimes_{A}\Omega_{1}(A)$ are two connections (with respect to the same DGA on top of $A$), then the composite
\[
M\oplus M'\xrightarrow{D\oplus D'}\big(M\otimes_{A}\Omega_{1}(A)\big)\oplus\big(M'\otimes_{A}\Omega_{1}(A)\big)\xrightarrow{\cong}(M\oplus M')\otimes_{A}\Omega_{1}(A)
\]
is a connection on $M\oplus M'$, called the \define{direct sum} of the connections $D$ and $D'$, and is denoted $D\oplus D'$ (which is a slight abuse of notation since the isomorphism above is included in its definition). The latter isomorphism is the distributive natural isomorphism and is determined on elementary tensors via $(s\otimes\omega,s'\otimes\eta)\mapsto(s,0)\otimes\omega+(0,s')\otimes\eta$. Extending $D\oplus D'$ by Leibniz is done via the composite $(M\oplus M')\otimes_{A}\Omega_{\bullet}(A)\xrightarrow{\cong}\big(M\otimes_{A}\Omega_{\bullet}(A)\big)\oplus(M'\otimes_{A}\Omega_{\bullet}(A)\big)\xrightarrow{D\oplus D'}\big(M\otimes_{A}\Omega_{\bullet+1}(A)\big)\oplus(M'\otimes_{A}\Omega_{\bullet+1}(A)\big)\xrightarrow{\cong}(M\oplus M')\otimes_{A}\Omega_{\bullet+1}(A)$.
The curvature of this connection in terms of the curvatures $R_{D}$ and $R_{D'}$ is then given by $R_{D}\oplus R_{D'}$ (followed by a similar distributive natural isomorphism).
\end{lem}

\begin{rmk}
Although forming the direct sum of connections can be easily done as above, forming the tensor product is a subtle issue in the noncommutative setting. Indeed, if we attempt to define a connection on $M\otimes_{A} M'$,  we first realize that we should use $A$-bimodules~\cite{CQ}. Secondly, even if these are bimodules, $D\otimes \text{id}+\text{id}\otimes D'$ does not define a connection on $M\otimes M'$ because $D(m)\otimes m'$ lives in $M\otimes_{A}\Om_{1}(A)\otimes_{A} M'$ while $m\otimes D'(m')$ lives in $M\otimes_{A}M'\otimes_{A}\Om_{1}(A)$. We cannot simply swap these factors because of the non-commutativity.
Fortunately, we will not make use of a $\otimes$-product in this work, so these issues will not concern us.
\end{rmk}

\begin{prp}\label{PRP.chk.is.closed} Karoubi's Chern character form $\chk_k(D)$ is closed in ${\Om}_\bl(A)_{\ab}$.
\end{prp}

\begin{proof}
A proof is in \cite[p.16--17]{Ka1}.
\end{proof}

\begin{dfn}
Fix an algebra $A$.
The set of isomorphism classes of finitely generated projective $A$-modules together with the direct sum operation defines a commutative monoid.
Let $K_{0}(A)$ denote the Grothendieck group associated to this monoid. Thus, $K_{0}(A)$ has elements that are formal differences $[M]-[N]$ of
isomorphism
classes of $A$-modules.
\end{dfn}

\begin{prp}\label{PRP.chk.is.indep.choice.of.conn}
Let $A$ be an algebra and fix a DGA $\Omega_{\bl}(A)$ on top of $A$.
The assignment $\chk_k:K_0(A)\ra {H}^{\dR}_{2k}(\Omega_{\bl}(A))$, sending $[M]$ to $[\chk_k(D)]$, where $M\in \Fgp(A)$ and $D$ is some connection on $M$, is a group homomorphism.
\end{prp}
\begin{proof} See \cite[Theorem 1.22]{Ka1} for a proof.
\end{proof}

\begin{lem}\label{lem:naturalityK0A}
The assignment from Proposition~\ref{PRP.chk.is.indep.choice.of.conn} is functorial and natural in the following senses.
\begin{enumerate}
\item
Fix a DGA $\Omega_{\bl}:\Alg\to\DGA$ on top of $\Alg$ (cf.\ Definition~\ref{defn:DGAontopofAlg}).
If $\psi:B\to A$ is an algebra homomorphism, then the assignment
\[
\begin{split}
K_{0}(B)&\xrightarrow{\psi_{*}}K_{0}(A)\\
[N]&\mapsto[N\otimes_{B}A]
\end{split}
\]
is a homomorphism of abelian groups, $K_{0}$ defines a functor from algebras to abelian groups, and the diagram
\[
\xy0;/r.25pc/:
(-15,7.5)*+{K_{0}(B)}="1";
(15,7.5)*+{H^{\dR}_{2k}(\Omega_{\bl}(B))}="2";
(-15,-7.5)*+{K_{0}(A)}="3";
(15,-7.5)*+{H^{\dR}_{2k}(\Omega_{\bl}(A))}="4";
{\ar"1";"2"^(0.45){\chk_{k}}};
{\ar"1";"3"_{\psi_{*}}};
{\ar"3";"4"^(0.45){\chk_{k}}};
{\ar"2";"4"^{\psi_{*}}};
\endxy
\]
commutes.
In other words, for a fixed $k\in\N$ and DGA $\Omega_{\bl}(A)$, the Karoubi--Chern character defines a natural transformation $\chk_{k}:K_{0}\Rightarrow H_{2k}^{\dR}$ for the functors $K_{0},H_{2k}^{\dR}:\Alg\to\Ab$.
\item
Let $\sigma:\Omega_{\bl}\Rightarrow\Theta_{\bl}$ be a natural transformation of DGA's on top of $\Alg$ (functors from $\Alg$ to $\DGA$). Then the diagram
\[\xymatrix{ && H_{2k}^{\dR}(\Omega_{\bl}(A)) \ar[dd]^{\sigma_{A}} \\
K_0(A)\ar[drr]^{\chk_k}\ar[urr]^{\chk_k} &&  \\
 &&   H_{2k}^{\dR}(\Theta_{\bl}(A))
}\]
commutes for all algebras $A$.
\end{enumerate}
\end{lem}
\begin{proof}
Well-defininedness of $K_{0}(B)\xrightarrow{\psi_{*}}K_{0}(A)$ follows from the
fact that the extension of scalars is a functor and sends isomorphism classes to isomorphism classes.
The group homomorphism condition follows from
$
(N\oplus N')\otimes_{B}A\cong (N\otimes_{B}A)\oplus (N'\otimes_{B}A)
$
for any two $B$-modules $N$ and $N'$.
The fact that the first diagram commutes follows from Lemma~\ref{lem:incucedcurvature}, Lemma~\ref{lem:ringmapsandtrace}, and Proposition~\ref{PRP.chk.is.indep.choice.of.conn}.
The last diagram commutes by the definition of the horizontal composition of natural transformations
\[
\xy0;/r.25pc/:
(-30,0)*+{\Alg}="1";
(0,0)*+{\DGA}="2";
(15,10)*+{\Alg}="3";
(30,0)*+{\Ab}="4";
{\ar@/^1.5pc/"1";"2"^{\Omega_{\bl}}};
{\ar@/_1.5pc/"1";"2"_{\Theta_{\bl}}};
{\ar"2";"3"^{\ev_{0}}};
{\ar"3";"4"^{K_{0}}};
{\ar@/_1.5pc/"2";"4"_{H^{\dR}_{2k}}};
{\ar@{=>}(-15,3);(-15,-3)^{\sigma}};
{\ar@{=>}(15,5);(15,-3)^{\chk_{k}}};
\endxy
\]
and the fact that $K_{0}$ itself does not depend on the DGA but only on the underlying algebra (the functor $\ev_{0}:\DGA\to\Alg$ takes the zeroth degree algebra---see Definition~\ref{defn:DGAontopofAlg}).
\end{proof}

\subsection{Review of universal noncommutative differential forms} \label{SEC.NC.differential.forms}
In this subsection we shall first review algebraic differential forms for commutative algebras and then we will review universal noncommutative differential forms by Karoubi \cite{Ka1} and Connes \cite{Co}. As shown by Karoubi, noncommutative de~Rham homology is closely related to cyclic homology (see Remark~\ref{RMK.nc.diff.forms} (\ref{item:cyclic})).
We shall also compare the relationship between Karoubi's Chern character and Connes' Chern character. Our intent is to provide only a concise survey of what we need. We refer the reader to Loday \cite[Sections 1.3, 2.6]{Lo}, Karoubi \cite[Sections 1.24, 2.15]{Ka1}, and Landi \cite[Chapter~7]{La03} for self-contained accounts.

Let $A$ be a unital commutative algebra over a field $\K$ and $M$ an $A$-module, which will occasionally be regarded as a symmetric $A$-bimodule. A \emph{derivation} $d:A\ra M$ is a $\K$-linear map satisfying the Leibniz rule, that is, $d(ab) = d(a) b + a d(b)$ for any $a, b \in A$. Let $\mathrm{Der}(A,M)$ denote the (additive) abelian group of derivations from $A$ to $M$. A derivation $d:A\to M$ is \emph{universal} if, for any derivation $\de:A\ra N$, there is a unique $\phi\in \text{Hom}_A(M,N)$ such that $\de=\phi \circ d$. To see that a universal derivation exists, consider $m:A\tsr_\K A\ra A$, the algebra multiplication of $A$, and let $I:=\ker m$.
Then the $\K$-linear map $A\ra I/I^2$ defined by $a\mapsto (1\tsr a-a\tsr 1)\!\mod I^2$ is a universal derivation. The $A$-module $I/I^2$ is isomorphic to the $A$-module $\Om^1_{A/\K}$ of \emph{K\"ahler differentials} (also known as algebraic differential $1$-forms), which is generated by the symbols $a\,db$ for $a,b\in A$~\cite[Sections~1.1.9, 1.3.8]{Lo}.
By the universality of $\Om^1_{A/\K}$, derivations on $A$ are classified as follows.

\begin{prp}\label{PRP.Der.commutative.case.classfication} The $A$-linear map $\phi:\text{Hom}_A(\Om^1_{A/\K},M)\ra \text{Der}(A,M)$, defined by $f\mapsto f\circ d$, where $d:A\to\Omega_{A/\K}^{1}$ is the universal derivation, specifies a natural isomorphism
\[
\xy0;/r.25pc/:
(-25,0)*+{\mathrm{SymBiMod}_{A}}="1";
(25,0)*+{A\text{-}\mathrm{Mod}}="2";
{\ar@/^1.25pc/^{\text{Hom}_A(\Om^1_{A/\K},-)}"1";"2"};
{\ar@/_1.25pc/_{\mathrm{Der}(A,-)}"1";"2"};
{\ar@{=>}(0,3);(0,-3)^{\phi}};
\endxy
\]
between functors $\text{SymBiMod}_A\ra A\text{-}\text{Mod}$ from the category of symmetric bimodules over $A$ to the category of left $A$-modules.
\end{prp}
\begin{proof} See Loday~\cite[Proposition 1.3.9]{Lo}.
\end{proof}

\begin{dfn}\label{DFN.alg.diff.forms} Let $A$ be a commutative algebra over a field $\K$. The \define{algebraic differential forms} of degree $n$, denoted by $\Om^n_{A/\K}$, is $\La^n_A \Om^1_{A/\K}$ the $n^{\text{th}}$ exterior power of $\Om^1_{A/\K}$ over $A$. The $n^{\text{th}}$ \define{de~Rham homology group} of $A$ is the degree $n$ homology group $\texttt{H}_n^{\dR}(A)$ of the complex $(\Om^\bl_{A/\K},d)$ with differential $d:\Om^n_{A/\K}\ra \Om^{n+1}_{A/\K}$ sending $a_0da_1\we \cdots\we da_n$ to $da_0\we da_1\we \cdots\we da_n$.
\end{dfn}

\begin{rmk} In the case of smooth functions on manifolds, although algebraic differential forms of smooth functions are related to exterior differential forms of a manifold, they are not isomorphic.
To see this, let $X$ be a smooth manifold and set $A:=\Cinf(X)$. By universality, there is a map $\phi:\Om^1_{A/\R}\ra \Om^1_{\dR}(X)$
defined by $fdg\mapsto fd_{\dR}g$, where the former $d$ is the K\"ahler differential and the latter $d_{\dR}$ is the de~Rham exterior differential. The map $\phi$ is onto, because any differential $1$-form can be written as a finite sum $\sum f_i d_{\dR} g_i$ by a standard argument, for example using a partition of unity. However $\phi$ is not one-to-one in general. Even when $X=\R$ and $A=\Cinf(\R)$, we have $de^x \neq e^x dx$ in $\Om^1_{A/\R}$.
\end{rmk}

Now let $A$ be any unital algebra over a field $\K$. If $A$ is not commutative, the above method of defining algebraic differential forms is no longer valid. One of reasons is that Proposition \ref{PRP.Der.commutative.case.classfication} relies on the commutativity of $A$. When $A$ is noncommutative, we have the following generalization of Proposition \ref{PRP.Der.commutative.case.classfication}~\cite[Section~2.6.1]{Lo}.

\begin{prp} \label{PRP.Der.noncommutative.case.classfication} Let $M$ be a $A$-bimodule (not necessarily symmetric) and let $I=\ker(A\tsr_\K A\srl{m}\ra A)$. The natural transformation%
\footnote{The codomain of the functors $\mathrm{Der}(A,-)$ and $\text{Hom}_{A\tsr A^{\text{op}}}(I,-)$ can be taken to be the category of abelian groups. This is because the forgetful functor from $Z(A\tsr A^{\text{op}})$-modules to abelian groups reflects isomorphisms (here, $Z(A\tsr A^{\text{op}})$ is the center of $A\tsr A^{\text{op}}$). This also gives a natural isomorphism of $Z(A\tsr A^{\text{op}})$-modules.}
$\Phi:\text{Hom}_{A\tsr A^{\text{op}}}(I,-)\Rightarrow \text{Der}(A,-)$ defined by $\Phi_M:f\mapsto f\circ\dun$ for each bimodule $M$, where $\dun:A \ra I$, which is defined by $x\mapsto \dun(x):=1\tsr x-x\tsr 1$ for all $x\in A$, is a natural isomorphism between functors from $\text{BiMod}_A$ to abelian groups. Here, $\text{BiMod}_A$ denotes the category of $A$-bimodules.
\end{prp}

\begin{dfn}
Let $A$ be a unital algebra over a field $\K$. The totality of \define{universal noncommutative differential $1$-forms} of $A$ is the $A$-bimodule $\univf_1(A):=\mathrm{ker}(A\otimes_{\K}A\xrightarrow{m}A)$, where the left and right actions are induced from the multiplication on $A$. More explicitly, given a generator $\dun x=1\tsr x-x\tsr 1$ in $\univf_{1}(A)$, the left $A$-module structure on $\dun x$ is given by $y\cdot \dun x:=y\tsr x-yx\tsr 1$, while the right $A$-module structure is
given by $\dun x\cdot y=1\otimes xy-x\otimes y$, which equals $\dun(xy)-x\cdot \dun y$. The left and right module structures on $\univf_1(A)$ are uniquely determined by these special cases.
\end{dfn}

The terminology `universal noncommutative differential forms' refers to the \emph{noncommutative differential forms} of Loday~\cite[Section 2.6.1]{Lo}, the \emph{universal differential algebra of forms} of Landi~\cite[Section 7.1]{La03} and the \emph{universal complex} of Karoubi~\cite[Section 1.24]{Ka1}.

\begin{dfn}\label{DFN.nc.diff.n.forms} Let $A$ be a unital algebra over a field $\K$. The totality of \define{universal noncommutative differential $n$-forms} is  $\univf_n(A):=\univf_1(A) \tsr_A \cdots \tsr_A \univf_1(A)$ ($n$ times).
\end{dfn}

An element in $\univf_n(A)$ is of the form $a_0\dun a_1\cdots \dun a_n$ (the symbol $\tsr$ is omitted). A differential $\dun:\Om^u_n(A)\ra \univf_{n+1}(A)$ is defined by $a_0\dun a_1\cdots \dun a_n\mapsto \dun a_0 \dun a_1\cdots \dun a_n$, and we set $\univf_{0}(A):=A$.  The complex $(\univf_{\bl}(A),\dun,\cdot)$ is a DGA.
This DGA satisfies the following universal property.
\begin{lem}
\label{lem:universal.property.NCDF}
Let $(\Gamma_{\bl},\delta)$ be a DGA and let $\phi:A\to\Gamma_{0}$ be a morphism of unital algebras. Then there exists a unique extension of $\phi$ to a morphism  $\phi^{\mathrm{u}}:\univf_{\bl}(A)\to\Gamma_{\bl}$ of DGAs.
\end{lem}
\begin{proof}
See Landi~\cite[Proposition~34, p.108]{La03} or Loday \cite[Section~2.6.6]{Lo}.
\end{proof}

With this, we can prove the following fact, which provides another example of a DGA on top of $\Alg$.

\begin{prp}
The assignment $\univf:\Alg\to\DGA$ sending $A$ to $(\univf_{\bl}(A),\dun)$ defines a DGA on top of $\Alg$.
\end{prp}

\begin{proof}
The assignment on morphisms is canonically defined by the universal property in Lemma~\ref{lem:universal.property.NCDF}. In more detail, let $\phi:A=\univf_{0}(A)\to B=\univf_{0}(B)$ be an algebra map. Then Lemma~\ref{lem:universal.property.NCDF} guarantees a unique DGA extension $\phi^{\mathrm{u}}:\univf_{\bl}(A)\to\univf_{\bl}(B)$. Functoriality immediately follows from this uniqueness.
\end{proof}

\begin{rmk}
In fact, $\univf$ is an initial object in the category of DGA's on top of $\Alg$. In this category, a morphism is just a natural transformation. Given another DGA $\Omega$ on top of $\Alg$, there is a canonical natural transformation $\univf\Rightarrow\Omega$ given by extending the identity on $A$ by Lemma~\ref{lem:universal.property.NCDF} to a DGA map $\sigma_{A}:\univf(A)\to\Omega(A)$. In fact, that same lemma proves that this assignment is actually natural because for any map of algebras $\psi:B\to A$, the same universal property guarantees a unique extension $\psi^{\mathrm{u}}:\univf(B)\to\Omega(A)$, which means the naturality diagram
\[
\xy0;/r.25pc/:
(-15,7.5)*+{\univf(B)}="Bt";
(-15,-7.5)*+{\Omega(B)}="B";
(15,7.5)*+{\univf(A)}="At";
(15,-7.5)*+{{\Omega}(A)}="A";
{\ar"Bt";"B"_{\sigma_{B}}};
{\ar"At";"A"^{\sigma_{A}}};
{\ar"Bt";"At"^{\univf(\psi)}};
{\ar"B";"A"_{\Omega(\psi)}};
{\ar@{-->}"Bt";"A"|-{\psi^{\mathrm{u}}}};
\endxy
\]
automatically commutes since each composite in the diagram must equal $\psi^{\mathrm{u}}$.
\end{rmk}

\begin{dfn}\label{DFN.NC.deRham.homology.universal} The \define{universal noncommutative de Rham homology group} $\uhdr_n(A)$ is the $n^{\text{th}}$ homology group of the complex $(\univf_{\bl}(A)_{\ab},\dun)$ (cf.\ Notation~\ref{NTA.karoubi.notation.dga}).
\end{dfn}

\begin{rmk} \label{RMK.nc.diff.forms}
\begin{enumerate}
\item
\label{item:universaltoordinarydR}
Let $(\Om_\bl(A),d,\cdot)$ be an arbitrary DGA  with $\Om_0(A)=A$.
By Lemma~\ref{lem:universal.property.NCDF}, the identity map on $A$ is uniquely extended to a DGA map $\Phi:\univf_{\bl}(A) \ra \Om_\bl(A)$, $a_0\dun a_1\cdots \dun a_n\mapsto a_0 da_1\cdots da_n$. Hence, there is an induced homomorphism at the level of noncommutative de~Rham homology groups
$$\Phi_{\Om_\bl}: {\uhdr_n}(A) \ra {H}_n^{\dR}(\Omega_{\bl}(A)).$$

\item
When $A$ is a commutative unital $\K$-algebra, Remark~\ref{RMK.nc.diff.forms}~(\ref{item:universaltoordinarydR}) gives a homomorphism from the universal noncommutative de Rham homology to the de Rham homology in Definition \ref{DFN.alg.diff.forms} by choosing $\Om_\bl(A)$ to be $\Om_{A/\K}^\bl$. Note that this map is not an isomorphism in general, but the groups  $\uhdr_n(A)$ and $\texttt{H}_n^{\dR}(A)$ are closely related (cf.\ Definition~\ref{DFN.alg.diff.forms}). See \cite[Corollary 3.4.15]{Lo} for more details.

\item
When $X$ is a smooth manifold and $A=\Cinf(X)$, the map $\phi:\Om^1_{A/\K}\ra \Om_{\dR}^1(X)$ given by the universality of $\Om^1_{A/\K}$ uniquely extends to a DGA map $\phi:\Om^\bl_{A/\K}\ra \Om_{\dR}^\bl(X)$ inducing a homomorphism  $\phi: \texttt{H}_\bl^{\dR}(A)\ra H^{\bl}_{\dR}(X)$. Combining with (2) above we have the following commutative diagram.
\[ \xymatrix{
\uhdr_\bl(A)\ar[r]_{\Phi_{\Om^\bl_{A/\K}}} \ar@/^+1.5pc/[rr]^{\Phi_{\Om_{\dR}^\bl}} & \texttt{H}_\bl^{\dR}(A) \ar[r]_{\phi} &  H^{\bl}_{\dR}(X)}
\]

\item
As pointed out in \cite[Section 1.24]{Ka1}, Karoubi's Chern characters for the choice of DGA $\univf_\bl(A)$, denoted by $\uchk$, and Karoubi's Chern characters $\chk$ for a generic DGA $\Omega_{\bl}(A)$ are related by the map $\Phi_{\Omega_{\bullet}}$ considered in (\ref{item:universaltoordinarydR}) of this remark. By Lemma~\ref{lem:naturalityK0A}, the following diagram commutes:
\[\xymatrix{ && \uhdr_{2k}(A) \ar[dd]^{\Phi_{\Omega_{\bullet}}} \\
K_0(A)\ar[drr]^{\chk_k}\ar[urr]^{\uchk_k} &&  \\
 &&  {H}_{2k}^{\dR}(A)
}\]

\item
\label{item:cyclic}
Universal noncommutative de Rham homology groups are closely related to cyclic homology groups by Karoubi's theorem \cite[p.31 Theorem 2.15]{Ka1}: For a unital algebra $A$ over a field $\K$ containing $\Q$ and for $n\geq 1$, $\uhdr_n(A)\srl\isom\ra \ker\left(\overline{HC}_n(A)\srl{B}\ra \overline{HH}_{n+1}(A) \right)$ and when $n=0$, $\uhdr_0(A)$ is isomorphic to $\ker(HC_0(A)\ra HH_1(A))$. This theorem is the key for comparing two different Chern characters. The Connes' Chern character (see Loday \cite[Section 8.3]{Lo}) and Karoubi's Chern character are, up to a constant, the same in cyclic homology (compare Karoubi \cite[Remark 2.20]{Ka1} and Rosenberg \cite[Example 6.2.9]{Ro}); i.e. the following diagram commutes:
\[\xymatrix{
&& \uhdr_{2k}(A) \ar[dd]^{(-1)^k(2k)!\Phi} \\
K_0(A)\ar[drr]^{\Ch_k}\ar[urr]^{\uchk_k} &&  \\
 &&  \overline{HC}_{2k}(A)
}\]
It is interesting to consider a noncommutative differential $K$-theory constructed purely in cyclic theoretic setup, though this is not the subject of the present work.
\end{enumerate}
\end{rmk}

\section{Karoubi--Chern--Simons transgression forms} \label{SEC.KCS}
In \cite{Ka1}, Karoubi provided an algebraic codification of the standard argument for the Chern--Weil theorem. However, the algebraic analogue of Chern--Simons forms were treated only implicitly (compare Proposition~\ref{PRP.chk.is.indep.choice.of.conn} above and \cite[Remark~1.23]{Ka1}) because the argument was meant to hold for any choice of DGA whose degree zero term is $A$. The purpose in this section is to obtain transgression formulas for Karoubi's Chern character forms.
Throughout this subsection, our base field $\K$ contains $\Q$.

\subsection{Polynomial paths of connections} In this subsection, we shall define and study the properties of an algebraic analogue of a path of connections on a smooth vector bundle. In what follows, $C_{\bl}\dtsr D_{\bl}$ denotes the graded tensor product of chain complexes $C_{\bl}$ and $D_{\bl}$ over the same field.

\begin{dfn}[Karoubi \cite{Ka1}, Section 1.20] \label{DFN.homotopy.operator} Let $\La_\bl$ be the chain complex defined by $\La_0:=\K[t]$ (formal polynomials in the variable $t$), $\La_1:=\K[t]dt$, $\La_{n\neq 0,1}=0$, $d_0(P(t))=P'(t)dt$, and $d_{n\neq 0}=0$. The \define{homotopy operator} $\kappa:\La_\bl \ra \K$ is defined so that $\kappa(P(t))=0$ and $\kappa(P(t)dt)=\int_0^1 P(t)dt$. Let $A$ be an algebra over $\K$ and let $\Om_\bl(A)$ be a DGA on top of $A$. The homotopy operator on $\La_\bl \dtsr \Om_\bl(A)$ is the (degree $-1$) map $K:\La_\bl \dtsr \Om_\bl(A)\to\Om_\bl(A)$ uniquely determined by $K(\omega\otimes\theta):=\kappa(\omega)\theta$.
We define the \define{evaluation at the boundary} maps $\mathrm{ev}_{n}:\La_\bl \dtsr  \Om_\bl(A)\to\Om_\bl(A)$ for $n\in\{0,1\}$ that are determined by $\mathrm{ev}_{n}(\omega\otimes\theta):=\omega(n)\theta$ if $\omega\in\La_{0}$ and $0$ if $\omega\in\La_{1}$. If $\varphi\in\La_{\bl}\dtsr\Om_{\bl}(A),$ then we write $\varphi(n):=\mathrm{ev}_{n}(\varphi)$ for brevity. Such evalation maps $\mathrm{ev}_{n}$ are similarly defined for every $n\in\K$. Furthermore, one can define similar such evaluation maps $\mathrm{ev}_{n}:\La_{0}\tsr_{\K}M\to M$ for any $\K$-vector space $M$.
\end{dfn}

We will think of the complex $\Lambda_{\bullet}$ as an object in the category of DGAs representing the standard unit interval (the de~Rham complex of polynomial differential forms).
Note that $\La_\bl \dtsr\Om_\bl(A)$ together with the tensor product differential is itself a DGA on top of
\[
\wtl{A}:=\La_0\tsr_\K A.
\]
Henceforth, we will set $\Om_{\bl}(\wtl{A})\equiv\Om_\bl(\La_0\tsr_\K A):=\La_\bl\dtsr\Om_\bl(A)$. We will often switch between the notation $\wtl{A}$ and $\La_{0}\otimes_{\K}A$ depending on the context (the former is often used for statements and to condense notation, while the latter in proofs to facilate computations). We will use the same notation $d$ to stand for the differential on $\Omega_{\bl}(A), \Lambda_{\bl},$ and $\Omega_{\bl}(\wtl{A})$ since it will always be made clear which one is being applied based on the input.
The abelianization of $\Om_\bl(\wtl{A})$ is given by $\Om_\bl(\La_0\tsr_\K A)_{\ab}=\La_\bl\dtsr\Om_\bl(A)_{\ab}$ because $\La_\bl$ is a commutative DGA. Furthermore, the evaluation maps and homotopy operators descend to the abelianizations.

\begin{lem}\label{LEM.homotopy.formula}[Homotopy formula] For any $\vph\in \La_\bl \dtsr  \Om_\bl(A)$,
\beqs
(Kd+dK)\varphi=\vph(1)-\vph(0).
\eeqs
\end{lem}

The following proposition of Karoubi \cite[Lemma 1.21]{Ka1} is readily seen by Lemma \ref{LEM.homotopy.formula}.
\begin{prp}\label{PRP.homotopy.formula.cycle} Let $\phi\in \La_\bl \tsr  \Om_\bl(A)_{\ab}$ be a cycle. Then $\phi(1)-\phi(0)$ is a boundary.
\end{prp}

\begin{dfn}
\label{defn:polynomial.family.connections}
Let $M\in\Fgp(A)$ and let $\Omega_{\bullet}(A)$ be a DGA on top of $A$. A \define{polynomial path of connections} on $M$ is a $\K$-linear map $M\xrightarrow{\mathcal{D}}\Lambda_{0}\otimes_{\K}M\otimes_{A}\Omega_{1}(A)$ satisfying a modified form of the Leibniz rule given by $\mathcal{D}(ma)=1\otimes m\otimes da +\mathcal{D}(m)a$ for all $a\in A$ and $m\in M$. Equivalently, $\mathcal{D}$ satisfies the condition that the composite $\mathcal{D}_{t}:=M\xrightarrow{\mathcal{D}}\Lambda_{0}\otimes_{\K}M\otimes_{A}\Omega_{1}(A)\xrightarrow{\mathrm{ev}_{t}}M\otimes_{A}\Omega_{1}(A)$ is a connection on $M$ for every $t\in\K$.
\end{dfn}

\begin{rmk}
\label{rmk:spaceofconnectionsisaffine}
The fact that these two properties of $\mathcal{D}$ are equivalent follows from the fact that an element of $\Lambda_{0}\otimes_{\K}M\otimes_{A}\Omega_{1}(A)$ is determined by its evaluations (since $\Q\subseteq\K$). Also, note that given any two connections $D_{0}$ and $D_{1}$ on $M$, there exists a polynomial path of connections $\mathcal{D}$ such that $\mathcal{D}_{t}=D_{t}$ for $t\in\{0,1\}$. Indeed, such a polynomial path of connections can be defined by $\mathcal{D}_{t}:=t D_{0}+(1-t)D_{1}$ for all $t\in\K$. More precisely, if we let $u$ denote the polynomial representing $t\mapsto t$, then $\mathcal{D}(m):=u\otimes D_{0}(m)+(1-u)\otimes D_{1}(m)$ for all $m\in M$. In fact, the set of connections is an affine $\K$-space.
\end{rmk}

Just as one can associate a bundle with connection over a space cross the interval given a \emph{path} of connections on a \emph{single} bundle, one can construct a module with connection over the module extended by polynomials on the interval given a (polynomial) path of connections on that module. We first provide the module definition and justify it by showing it specializes properly to the smooth setting.

\begin{lem}\label{lem:intervalcrossspacenc}
Let $M$ be a finitely generated projective $A$-module and $\Omega_{\bl}(A)$  a DGA on top of $A$ and let  $\mathcal{D}:M\to\Lambda_{0}\otimes_\K M\otimes_A\Omega_{1}(A)$ be a polynomial path of connections on $M$. Set $\wtl{M}:=\La_{0}\tsr_\K M$.
\begin{enumerate}
\item
Then $\wtl{M}$, equipped with the right action of $\wtl{A}:=\La_{0}\otimes_{\K}A$ uniquely determined by
\[
(q\otimes m)(p\otimes a):=(qp)\otimes(ma)\qquad\forall\;q,p\in\Lambda_{0},\ m\in M,\ a\in A,
\]
is a finitely generated projective $\wtl{A}$-module.
\item\label{item:pathsofconnections2}
Setting $\Om_\bl(\wtl{A}):=\La_{\bl}\dtsr\Om_{\bl}(A)$ to be the DGA on top of $\wtl{A}$, the map $\wtl{M}\otimes_{\wtl{A}}\Omega_{1}(\wtl{A})\to(\La_{1}\otimes_{\K}M)\oplus(\La_{0}\otimes_{\K}M\otimes_{A}\Omega_{1}(A))$ uniquely determined by
\[
\begin{split}
(\La_0\otimes_{\K}M)\otimes_{\wtl{A}}\big((\La_1\tsr_\K A)\oplus(\La_0\tsr_\K\Omega_{1}(A))\big)&\to(\La_{1}\otimes_{\K}M)\oplus(\La_{0}\otimes_{\K}M\otimes_{A}\Omega_{1}(A))\\
(q\otimes m)\tsr\big((\eta\tsr a)+(p\tsr\omega)\big)&\mapsto
q\eta\tsr ma+qp\tsr m\tsr\omega.
\end{split}
\]
is a canonical isomorphism.
More generally, there is a canonical isomorphism
    \[
\wtl{M}\otimes_{\wtl{A}}\Omega_{k}(\wtl{A})\cong\big(\La_{1}\otimes_{\K}M\tsr\Omega_{k-1}(A)\big)\oplus\big(\La_{0}\otimes_{\K}M\otimes_{A}\Omega_{k}(A)\big)
\]
for all $k\in\N$.
\item
Using the first of these isomorphisms from (\ref{item:pathsofconnections2}), the unique linear map $\wtl{\mathcal{D}}$ determined by
\[
\begin{split}
\wtl{M}&\xrightarrow{\wtl{\mathcal{D}}}
(\La_{1}\otimes_{\K}M)\oplus(\La_{0}\otimes_{\K}M\otimes_{A}\Omega_{1}(A))\xrightarrow{\cong}\wtl{M}\otimes_{\wtl{A}}\Omega_{1}(\wtl{A})\\
q\otimes m&\xmapsto{\;\;\;}\quad\;\; dq\otimes m + q\mathcal{D}(m)
\end{split}
\]
defines a connection on $\wtl{M}$ with respect to the DGA $\Om_\bl(\wtl{A})$. In fact, $\wtl{\mathcal{D}}$ is determined by its action on elements of the form $1\otimes m$, which gets sent to $\mathcal{D}(m)$.
\item
The action of the curvature,  $R_{\wtl{\mathcal{D}}}=\wtl{\mathcal{D}}^2$, is determined by its action on $1\otimes m\in\wtl{M}$, which is given by
\[
\wtl{\mathcal{D}}^2(1\otimes m)=\sum_{\alpha}dq_{\alpha}\tsr m_{\alpha}\tsr\omega_{\alpha}+\sum_{\alpha}q_{\alpha}\mathcal{D}(m_\alpha)\omega_{\alpha}+\sum_{\alpha}q_{\alpha}\tsr m_{\alpha}\tsr d\omega_{\alpha}
\]
where
\[
\mathcal{D}(m)=:\sum_{\alpha}q_{\alpha}\otimes m_{\alpha}\otimes\omega_{\alpha}\in\La_0\tsr_\K M\tsr_A\Omega_{1}(A)
\]
is a tensor product expansion of $\mathcal{D}(m)$.
\item
The curvature $R_{\wtl{\mathcal{D}}}=\wtl{\mathcal{D}}^2$ of the connection $\wtl{\mathcal{D}}$ satisfies the relation
\[
\mathrm{ev}_{t}\circ\wtl{\mathcal{D}}^{2}\circ i_{M}=\mathcal{D}_{t}^{2}\qquad\forall\;t\in\K,
\]
where $\mathrm{ev}_{t}:\wtl{M}\tsr_{\wtl{A}}\Omega_{\bl}(\wtl{A})\to M\tsr_{A}\Omega_{\bl}(A)$ evaluates all expressions at $t\in\K$ (making use of the isomorphism from (\ref{item:pathsofconnections2}) above) and $i_{M}:M\otimes_{A}\Omega_{\bl}(A)\to\La_0\otimes_\K M\tsr_{A}\Omega_{\bl}(A)\hookrightarrow\wtl{M}\otimes_{\wtl{A}}\Omega_{\bl}(\wtl{A})$ is the inclusion sending $m\otimes\omega$ to $1\otimes m\otimes\omega$ (again making use of the isomorphism from (\ref{item:pathsofconnections2}) above).
In other words, the diagram
\[
\xy0;/r.25pc/:
(-25,7.5)*+{M\otimes_{A}\Omega_{\bl}(A)}="1";
(25,7.5)*+{M\otimes_{A}\Omega_{\bl}(A)}="2";
(-25,-7.5)*+{\wtl{M}\otimes_{\wtl{A}}\Omega_{\bl}(\wtl{A})}="3";
(25,-7.5)*+{\wtl{M}\otimes_{\wtl{A}}\Omega_{\bl}(\wtl{A})}="4";
{\ar"1";"2"^{\mathcal{D}_{t}^{2}}};
{\ar"3";"4"^{\wtl{\mathcal{D}}^{2}}};
{\ar"1";"3"^{i_{M}}};
{\ar"4";"2"_{\mathrm{ev}_{t}}};
\endxy
\]
commutes. Furthermore, $\mathrm{ev}_{t}$ satisfies $\mathrm{ev}_{t}\circ i_{M}=\mathrm{id}$ and it also restricts to a well-defined map $\mathrm{ev}_{t}:\wtl{M}\tsr_{\wtl{A}}\Omega_{\even}(\wtl{A})\to M\tsr_{A}\Omega_{\even}(A)$.
\item
Given any $\Omega_{\bl}(\wtl{A})$-linear map $L$ from $\wtl{M}\otimes_{\wtl{A}}\Omega_{\bl}(\wtl{A})$ to itself, set
$
L_{t}:=\mathrm{ev}_{t}\circ L\circ i_{M}
$
for each $t\in\K$. Then $L_{t}$ is a $\Omega_{\bl}(A)$-linear map from $M\otimes_{A}\Omega_{\bl}(A)$ to itself satisfying
\[
\mathrm{ev}_{t}\circ L^{k}\circ i_{M}=L_{t}^{k}\qquad\forall\;k\in\N.
\]
In particular, $\mathrm{ev}_{t}\circ R_{\wtl{\mathcal{D}}}^{k}\circ i_{M}=R^{k}_{\mathcal{D}_{t}}$ for all $t\in\K$ and $k\in\N$.
\item
The identity
\[
\mathrm{ev}_{t}\Big(\tr\big(R_{\wtl{\mathcal{D}}}^k\big)\Big)=\tr\big(R_{\mathcal{D}_{t}}^{k}\big), \quad\text{ i.e.\ }\quad
\mathrm{ev}_{t}\Big(\tr\big(\wtl{\mathcal{D}}^{2k}\big)\Big)=\tr\big(\mathcal{D}_{t}^{2k}\big),
\]
holds for all $k\in\N$ and $t\in\K$. Here, the two traces have different codomains (and domains). The one on the left side of each equality lands in $(\Omega_{\bl}(\wtl{A})_{\ab})_{\even}$,
while the one on the right side of each equality lands in $(\Om_{\bl}(A)_{\ab})_{\even}$. Consequently,
\[
\mathrm{ev}_{t}\Big(\chk\big(\wtl{\mathcal{D}}\big)\Big)=\chk\big(\mathcal{D}_{t}\big)\qquad\forall\;t\in\K.
\]
\item\label{item:pathofconnections8}
Let $\La_{0}\xrightarrow{r}\La_{0}$ be the map that reverses the formal input variable, namely $(rq)(t):=q(1-t)$ for all $q\in\La_{0}\equiv\K[t]$. Then $r$ extends uniquely to a DGA map satisfying%
\footnote{This is analogous to how the pullback of smooth functions and differential forms along a map (which in this case is the map $\K\ni t\mapsto1-t$) is a chain map on the de~Rham complex.}
$r(pdq)=r(p)d(r(q))$, and the composite $M\xrightarrow{\mathcal{D}}\La_{0}\tsr_{\K}M\tsr_{A}\Omega_{1}(A)\xrightarrow{r\otimes_{\K}\mathrm{id}}\La_{0}\tsr_{\K}M\tsr_{A}\Omega_{1}(A)$ defines a path of connections $\overline{\mathcal{D}}$, called the \define{reversed path of connections}, whose associated connection $\wtl{\overline{\mathcal{D}}}:\wtl{M}\otimes_{\wtl{A}}\Omega_{\bl}(\wtl{A})\to\wtl{M}\otimes_{\wtl{A}}\Omega_{\bl}(\wtl{A})$ satisfies
\[
\overline{\mathcal{D}}_{t}=\mathcal{D}_{1-t}\quad\forall\;t\in\K\qquad\text{ and }\qquad
\wtl{\overline{\mathcal{D}}}=(r\otimes\mathrm{id})\circ\wtl{\mathcal{D}}\circ(r\otimes\mathrm{id}),
\]
where $r\otimes\mathrm{id}:=r\otimes_{\K}\mathrm{id}_{M\otimes_{A}\Omega_{\bl}(A)}$ is the $\Omega_{\bl}(A)$-linear endomorphism%
\footnote{The map $r\otimes_{\K}\mathrm{id}_{M\otimes_{A}\Omega_{\bl}(A)}$ is \emph{not} $\Omega_{\bl}(\wtl{A})$-linear.}
of $\wtl{M}\otimes_{\wtl{A}}\Omega_{\bl}(\wtl{A})$ whose action on degree $(k+1)$-forms is obtained via the isomorphism
\[
\wtl{M}\otimes_{\wtl{A}}\Omega_{k+1}(\wtl{A})\cong \Big(\La_0\otimes_{\K}M\otimes_{A}\Omega_{k+1}(A)
\Big)\oplus\Big(\La_1\otimes_{\K}M\otimes_{A}\Omega_{k}(A)\Big)
\]
with $r$ acting on the first factor ($\La_0$ and $\La_1$) in each component of the direct sum (in other words, $r$ acts on the first factor of the $\Omega_{\bl}(A)$-module $\La_{\bl}\otimes_{\K}M\otimes_{A}\Omega_{\bl}(A)$).
Furthermore, the Karoubi--Chern character satisfies
\[
\chk\left(\wtl{\overline{\mathcal{D}}}\right)=
(r\otimes\mathrm{id})\chk\left(\wtl{\mathcal{D}}\right).
\]
\end{enumerate}
\end{lem}

\begin{proof}
Most of these claims amount to unraveling the definitions.
\begin{enumerate}
\item
To see that $\wtl{M}$ is a finitely generated projective $\wtl{A}$-module, note that since $M$ is a finitely generated projective $A$-module, there exists an $A$-module $Q$ and an $n\in\N$ such that $M\oplus Q\cong A^{n}$. Hence,
\[
\wtl{M}\oplus(\La_{0}\tsr_{\K}Q)
=(\La_{0}\tsr_{\K}M)\oplus(\La_{0}\tsr_{\K}Q)
\cong\La_{0}\tsr_{\K}(M\oplus Q)
\cong\La_{0}\tsr_{\K}A^{n}\cong(\La_{0}\tsr_{\K}A)^{n},
\]
which proves the claim.
\item
Because
\[
\Omega_{1}(\La_0\tsr_\K A)=(\La_1\tsr_\K A)\oplus(\La_0\tsr_\K\Om_1(A))
\]
(by definition) together with the properties of $\otimes$ and $\oplus$,
\[
\begin{split}
\wtl{M}&\otimes_{\wtl{A}}\Omega_{1}(\wtl{A})=(\La_{0}\otimes_{\K}M)\tsr_{\La_{0}\tsr A}\Big((\La_1\tsr_\K A)\oplus\big(\La_0\tsr_\K\Om_1(A)\big)\Big)\\
&\cong\Big((\La_{0}\otimes_{\K}M)\tsr_{\La_{0}\tsr A}(\La_1\tsr_\K A)\Big)\oplus\Big((\La_{0}\otimes_{\K}M)\tsr_{\La_{0}\tsr A}\big(\La_0\tsr_\K\Om_1(A)\big)\Big)\\
&\cong\Big(\big(\La_{0}\otimes_{\La_0}\La_{1}\big)\otimes_{\K}\big(M\otimes_{A}A\big)\Big)\oplus\Big(\big(\La_{0}\otimes_{\La_0}\La_{0}\big)\otimes_{\K}\big(M\otimes_{A}\Omega_{1}(A)\big)\Big)\\
&\cong(\La_{1}\otimes_{\K}M)\oplus\big(\La_{0}\otimes_{\K}M\otimes_{A}\Omega_{1}(A)\big)
\end{split}
\]
provides the claimed isomorphism. A completely analogous calculation proves the result for tensoring with $k$-forms.
\item
Note that $q\otimes m=(1\otimes m)(q\otimes1_A)$. Therefore, $\wtl{\mathcal{D}}$ is determined by its values on elements of the form $1\otimes m$ by the Leibniz rule.
\item
By definition (cf.\ Definition~\ref{defn:extendingconnection}), $\wtl{\mathcal{D}}$ is extended to $\wtl{M}\otimes_{\wtl{A}}\Omega_{\bl}(\wtl{A})$ by the formula
\begin{equation}
\label{eq:extendingDtilde}
\wtl{\mathcal{D}}(\wtl{m}\otimes\eta)=\wtl{\mathcal{D}}(\wtl{m})\eta+\wtl{m}\otimes d\eta
\end{equation}
for all $\wtl{m}\in\wtl{M}$ and $\eta\in\Omega_{k}(\wtl{A})$.
The fact that the action of $\wtl{\mathcal{D}}^2$ is determined by its action on $1\otimes m\in\wtl{M}$ follows from $\Omega_{\bl}(\wtl{A})$-linearity of the curvature (cf.\ Definition~\ref{defn:extendingconnection}). Then
\[
\begin{split}
\wtl{\mathcal{D}}^2(1\otimes m)&=\wtl{\mathcal{D}}\big(\mathcal{D}(m)\big)
=\sum_{\alpha}\wtl{\mathcal{D}}(q_{\alpha}\otimes m_{\alpha}\otimes\omega_{\alpha})\\
&=\sum_{\alpha}\wtl{\mathcal{D}}(q_\alpha\tsr m_\alpha)\omega_{\alpha}+\sum_{\alpha}q_{\alpha}\tsr m_{\alpha}\tsr d\omega_{\alpha}\quad\text{ by (\ref{eq:extendingDtilde})}\\
&=\sum_{\alpha}\Big(dq_{\alpha}\tsr m_{\alpha}+q_{\alpha}\mathcal{D}(m_{\alpha})\Big)\omega_{\alpha}+\sum_{\alpha}q_{\alpha}\tsr m_{\alpha}\tsr d\omega_{\alpha}\\
&=\sum_{\alpha}dq_{\alpha}\tsr m_{\alpha}\tsr\omega_{\alpha}+\sum_{\alpha}q_{\alpha}\mathcal{D}(m_\alpha)\omega_{\alpha}+\sum_{\alpha}q_{\alpha}\tsr m_{\alpha}\tsr d\omega_{\alpha},
\end{split}
\]
where the definition of $\wtl{\mathcal{D}}$ is used in the third line and the fact that $m_{\alpha}$ identifies with $m_{\alpha}\tsr1_{A}$ under the canonical isomorphism $M\cong M\otimes_{A}A$ is used in the last line (in the first of three terms).
\item
It is sufficient to prove the equality for forms of the lowest degree. Hence, as before, fix $m\in M$ and let $\mathcal{D}(m)=:\sum_{\alpha}q_{\alpha}\otimes m_{\alpha}\otimes\omega_{\alpha}$. Then
\[
\mathrm{ev}_{t}\Big(\wtl{\mathcal{D}}^2\big(i_{M}(m)\big)\Big)
=\mathrm{ev}_{t}\big(\wtl{\mathcal{D}}^2(1\otimes m)\big)
=\sum_{\alpha}q_{\alpha}(t)\Big(\mathcal{D}_{t}(m_{\alpha})\omega_{\alpha}+m_{\alpha}\otimes d\omega_{\alpha}\Big)
\]
by item (4).
On the other hand, since $\mathcal{D}_{t}$ is a connection on $M$ with respect to $\Omega_{\bl}(A)$, the Leibniz rule gives
\[
\begin{split}
\mathcal{D}^{2}_{t}(m)&=\mathcal{D}_{t}\left(\sum_{\alpha}q_{\alpha}(t)m_{\alpha}\otimes\omega_{\alpha}\right)
=\sum_{\alpha}q_{\alpha}(t)\mathcal{D}_{t}(m_{\alpha}\otimes\omega_{\alpha})\\
&=\sum_{\alpha}q_{\alpha}(t)\Big(\mathcal{D}_{t}(m_{\alpha})\omega_{\alpha}+m_{\alpha}\otimes d\omega_{\alpha}\Big)
\end{split}
\]
The identity $\mathrm{ev}_{t}\circ i_{M}=\mathrm{id}$ is immediate. The fact that $\mathrm{ev}_{t}$ restricts to a well-defined map in even degrees follows from the fact that $\mathrm{ev}_{t}(\theta)=0$ for all $\theta\in\La_{1}$.
\item
$\Omega_{\bl}(A)$-linearity of $L_{t}$ follows from
\[
L_{t}(m\otimes\omega)
=\mathrm{ev}_{t}\big(L(1\otimes m\otimes\omega)\big)
=\mathrm{ev}_{t}\big(L(1\otimes m\otimes1_{A})(1\otimes\omega)\big)
=L_{t}(m)\omega
\]
for $m\in M$ and $\omega\in\Omega_{\bl}(A)$.
For the next claim, it suffices to prove the claim for $k=2$. Furthermore, since $L(1\otimes m\otimes\omega)=L(1\otimes m)(1\otimes\omega)$ by right $\Omega_{\bl}(\La_0\tsr_{\K}A)$-linearity, it suffices to prove $(\mathrm{ev}_{t}\circ L^{2} \circ i_{M})(m)=L_{t}^{2}(m)$. Write
\begin{equation}
\label{eq:generalelementLambdaMA}
L\big(i_{M}(m)\big)\equiv L(1\otimes m\otimes 1_{A})=:\sum_{\alpha}q_{\alpha}\otimes m_{\alpha}\otimes\omega_{\alpha}+\sum_{\beta}\theta_{\beta}\otimes n_{\beta}\otimes\xi_{\beta}
\end{equation}
with $q_{\alpha}\in\La_{0},\theta_{\beta}\in\La_{1}$, $m_{\alpha},n_{\beta}\in M$, and $\omega_{\alpha},\xi_{\beta}\in\Omega_{\bl}(A)$. Then,
\[
\begin{split}
(\mathrm{ev}_{t}\circ L^{2}\circ i_{M})(m)
&=\mathrm{ev}_{t}\left(\sum_{\alpha}L(1\otimes m_{\alpha}\otimes1_{A})(q_{\alpha}\otimes\omega_{\alpha})+\sum_{\beta}L(1\otimes n_{\beta}\otimes1_{A})(\theta_{\beta}\otimes \xi_{\beta})\right)\\
&=\sum_{\alpha}q_{\alpha}(t)L_{t}(m_{\alpha})\omega_{\alpha}=L_{t}\left(\sum_{\alpha}q_{\alpha}(t)m_{\alpha}\otimes\omega_{\alpha}\right)=L_{t}^2(m).
\end{split}
\]
The claim $\mathrm{ev}_{t}\circ R_{\wtl{\mathcal{D}}}^{k}\circ i_{M}=R^{k}_{\mathcal{D}_{t}}$ follows from setting $L_{t}:=R_{\mathcal{D}_{t}}$ and $L:=R_{\wtl{\mathcal{D}}}$.
\item
The identity itself comes from commutativity of the diagram
\[
\xy0;/r.25pc/:
(-35,-7.5)*+{\mathrm{End}_{\Omega_{\even}(A)}\Big(M\otimes_{A}\Omega_{\even}(A)\Big)}="1";
(35,-7.5)*+{(\Om_{\bl}(A)_{\ab})_{\even}}="2";
(-35,7.5)*+{\mathrm{End}_{\Omega_{\even}(\wtl{A})}\Big(\wtl{M}\otimes_{\wtl{A}}\Omega_{\even}(\wtl{A})\Big)}="3";
(35,7.5)*+{(\Om_{\bl}(\wtl{A})_{\ab})_{\even}}="4";
{\ar"1";"2"^(0.65){\tr}};
{\ar"3";"4"^(0.65){\tr}};
{\ar"3";"1"^{\mathrm{ev}_{t}\circ\;\cdot\;\circ i_{M}}};
{\ar"4";"2"^{\mathrm{ev}_{t}}};
\endxy
,
\]
which follows from a similar calculation to the one in the proof of (6) by using the canonical isomorphism $F\otimes_{R}F^*\cong\mathrm{End}_{B}(F)$ (cf.\ Definition~\ref{defn:Fdual} and afterwards). In more detail, setting $R:=\Omega_{\even}(A)$, $F:=M\otimes_{A}R$, $\wtl{R}:=\Omega_{\even}(\wtl{A})$, and $\wtl{F}:=\wtl{M}\tsr_{\wtl{A}}\wtl{R}$, the claim follows from commutativity of
\[
\xy0;/r.25pc/:
(-20,-7.5)*+{F\otimes_{R}F^*}="1";
(20,-7.5)*+{R/[R,R]}="2";
(-20,7.5)*+{\wtl{F}\otimes_{\wtl{R}}\wtl{F}^*}="3";
(20,7.5)*+{\wtl{R}/[\wtl{R},\wtl{R}]}="4";
{\ar"1";"2"^{\mathrm{eval}}};
{\ar"3";"4"^{\mathrm{eval}}};
{\ar"3";"1"^{\mathrm{Ev}_{t}}};
{\ar"4";"2"^{\mathrm{ev}_{t}}};
\endxy
,
\]
where the left $\mathrm{Ev}_{t}$ is defined by sending $\wtl{f}\otimes\wtl{\varphi}$ to $\mathrm{ev}_{t}(\wtl{f})\otimes(\mathrm{ev}_{t}\circ\wtl{\varphi}\circ i_{M})$. The commutativity of this last diagram can be proved by choosing an element $\wtl{f}$ of the form similar to the right-hand-side of~(\ref{eq:generalelementLambdaMA}) and explicitly computing the image along the two directions.
\item
A short calculation shows $\overline{\mathcal{D}}$ is a path of connections. With the associated connection on $\wtl{M}$, a direct calculation gives
\[
\begin{split}
\wtl{\overline{\mathcal{D}}}&(q\otimes m)=dq\otimes m+q\overline{\mathcal{D}}(m)\qquad\text{by item (3) of this lemma for $\overline{\mathcal{D}}$}\\
&=dq\otimes m+q\big((r\otimes\mathrm{id})\mathcal{D}(m)\big)\qquad\text{by definition of $\overline{\mathcal{D}}$}\\
&=(r\otimes\mathrm{id})\Big(d\big(r(q)\big)\otimes m+r(q)\mathcal{D}(m)\Big)\qquad\text{since $r$ is a DGA map with $r^2=\mathrm{id}$}\\
&=\big((r\otimes\mathrm{id})\circ\wtl{\mathcal{D}}\circ(r\otimes\mathrm{id})\big)(q\otimes m)\qquad\text{by item (3) of this lemma for $\mathcal{D}$},
\end{split}
\]
proving the main relationship at the lowest degree (since $q\otimes m\in\wtl{M}$). This then extends to higher degree forms, namely $\wtl{M}\otimes_{\wtl{A}}\Omega_{\bl}(\wtl{A})$, using the isomorphism stated in the claim. Since $r^2=\mathrm{id}$, the curvature satisfies $R_{\wtl{\overline{\mathcal{D}}}}^{k}=(r\otimes\mathrm{id})\circ R_{\wtl{\mathcal{D}}}^{k}\circ(r\otimes\mathrm{id})$ for all $k\in\N$.%
\footnote{At this point, one cannot naively invoke the cyclicity of trace and conclude that the Karoubi--Chern form of $\wtl{\overline{\mathcal{D}}}$ is the same as that of $\wtl{\mathcal{D}}$ because $r\otimes\mathrm{id}$ is not $\Omega_{\bl}(\wtl{A})$-linear (indeed, it modifies the $\La_{\bl}$ component in a nontrivial fashion).}
Since
\[
R_{\wtl{\overline{\mathcal{D}}}}^{k}(\lambda\otimes m\otimes\omega)=\big(R_{\wtl{\overline{\mathcal{D}}}}^{k}(1\otimes m\otimes1_{A})\big)(\lambda\otimes\omega)
\]
for all $\lambda\in\La_{\bl},m\in M$, and $\omega\in\Omega_{\bl}(A)$, the trace only depends on how $R_{\wtl{\overline{\mathcal{D}}}}^{k}$ acts on the $M$ factor. But by the relation we just obtained, $R_{\wtl{\overline{\mathcal{D}}}}^{k}(1\otimes m\otimes1_{A})=(r\otimes\mathrm{id})R_{\wtl{\mathcal{D}}}^{k}(1\otimes m\otimes1_{A})$. Hence, $\tr\big(R_{\wtl{\overline{\mathcal{D}}}}^{k}\big)=(r\otimes\mathrm{id})\tr\big(R_{\wtl{\mathcal{D}}}^{k}\big)$, which gives the claim. \qedhere
\end{enumerate}
\end{proof}

\begin{exa}
\label{rmk:path.of.connections}
Let $\K$ be $\R$ or $\C$.
The construction of $\wtl{\mathcal{D}}$ in Lemma~\ref{lem:intervalcrossspacenc} is an analogue of the following construction from connections on vector bundles (cf.\ \cite[pg~197]{Mo01}). Let $E\xrightarrow{\pi}X$ be a $\K$-vector bundle over a smooth compact manifold $X$, and let $\Gamma(E)$ be the associated $C^{\infty}(X)$-module of sections. Let $\wtl{E}\xrightarrow{\wtl{\pi}}[0,1]\times X$ be the vector bundle obtained as the pullback of $E$ via the projection $[0,1]\times X\to X$.
\[
\xy0;/r.25pc/:
(-12.5,7.5)*+{\wtl{E}}="1";
(12.5,7.5)*+{E}="2";
(-12.5,-7.5)*+{[0,1]\times X}="3";
(12.5,-7.5)*+{X}="4";
(-8.5,5)*{\lrcorner};
{\ar"1";"2"};
{\ar"1";"3"_{\wtl{\pi}}};
{\ar"2";"4"^{\pi}};
{\ar"3";"4"};
\endxy
\]
The fiber $\wtl{E}_{(t,x)}$ over $(t,x)\in[0,1]\times X$ can be taken to be $E_{x}$, and $\Gamma(\wtl{E})$ is a $C^{\infty}([0,1]\times X)$-module. Note that $C^{\infty}([0,1]\times X)$ is a Fr\'echet space~\cite[Section~44 Corollary~1 pg~447]{Tr16}, which is isomorphic to $C^{\infty}([0,1])\ctsr_{\epsilon}C^{\infty}(X)$~\cite[Theorem~44.1 pg~449]{Tr16}. Here, $\epsilon$ denotes the $\epsilon$-topology~\cite[Definition~43.1 pg~434]{Tr16} and $C^{\infty}([0,1])\ctsr_{\epsilon}C^{\infty}(X)$ denotes the completion of $C^{\infty}([0,1])\otimes_{\K} C^{\infty}(X)$ with respect to the $\epsilon$-topology~\cite[Definition~43.5 pg~439]{Tr16}. A similar isomorphism holds for the modules of sections, namely $\Gamma(\wtl{E})\cong C^{\infty}([0,1])\ctsr_{\epsilon}\Gamma(E)$. This comes from equipping the space of smooth sections of any bundle with a natural family of seminorms that turns it into a Fr\'echet space (see Section~2.1 in~\cite{Ba15}).

In the smooth setting, $C^{\infty}([0,1])$ replaces $\La_{0}$ from our algebraic setup. By Wierstrauss' approximation theorem, the space $\La_{0}$ of polynomials over $[0,1]$ is dense in $C([0,1])$ with respect to the uniform topology, and therefore $\La_{0}$ is also dense in $C^{\infty}([0,1])$ with respect to the Fr\'echet topology. From this, it follows that $\La_{0}\otimes_{\K}\Gamma(E)$ is dense in $C^{\infty}([0,1])\otimes_{\K}\Gamma(E)$. Since the latter is dense in $C^{\infty}([0,1])\ctsr_{\epsilon}\Gamma(E)$ (by definition), $\La_{0}\otimes_{\K}\Gamma(E)$ is dense in $C^{\infty}([0,1])\ctsr_{\epsilon}\Gamma(E)$.
Therefore, there is no loss of generality by specifying a path of connections using $\La_{0}$ instead of $C^{\infty}([0,1])$.

Now, associated to a (polynomial) path of connections $\nabla:\Gamma(E)\to C^{\infty}([0,1])\otimes_{\K}\Gamma(E)\otimes_{C^{\infty}(X)}\Omega_{\dR}^{1}(X)$ is a family of connections obtained via evaluation at $t\in[0,1]\subseteq\R$ giving $\nabla^{t}:\Gamma(E)\to\Gamma(E)\otimes_{C^{\infty}(X)}\Omega_{\dR}^{1}(X)$. There is an associated connection $\wtl{\na}$ on $\wtl{E}$ uniquely determined by
\[
\wtl{\na}_{\frac{\partial}{\partial t}}s=0
\quad\text{ and }\quad
\wtl{\na}_{V}s=\Big([0,1]\times X\ni(t,x)\mapsto\na_{V_{(t,x)}}^{t}s\in E_{x}\Big)
\]
for sections $s\in\Gamma(\wtl{E})$ that do not depend on $t\in[0,1]$ (this means $s$ is the pullback of a section of $E$ along the projection) and for vector fields $V$ such that $V_{(t,x)}\in  T_{(t,x)}(\{t\}\times X)\subseteq T_{(t,x)}([0,1]\times X)$ (there is some abuse of notation in the right-most term, where $s$ in $\na_{V_{(t,x)}}^{t}s$ should be interpreted as the pullback of $s$ along the inclusion $\iota_{t}:X\xrightarrow{\cong}\{t\}\times X\subseteq[0,1]\times X$). This is the analogue of specifying $\wtl{\mathcal{D}}$ on elements of the form $1\otimes m$ from Lemma~\ref{lem:intervalcrossspacenc}.
\end{exa}

\subsection{Karoubi--Chern--Simons forms}
\begin{dfn}
\label{defn:KCSform}
The $k^{\text{th}}$ \define{Karoubi--Chern--Simons form} of a polynomial path of connections $\mathcal{D}$ is $\KCS_k(\mathcal{D}):=K\chk_k(\wtl{\mathcal{D}})$, where $K$ is the homotopy operator defined in Definition~\ref{DFN.homotopy.operator} and $\wtl{\mathcal{D}}$ is the connection on $\wtl{M}$ induced by $\mathcal{D}$ as in Lemma~\ref{lem:intervalcrossspacenc}.
The \define{total Karoubi--Chern--Simons form} of $\mathcal{D}$ is $\KCS(\mathcal{D}):=K\chk(\wtl{\mathcal{D}})=\sum_{k=0}^\infty K\chk_k(\wtl{\mathcal{D}})$.
\end{dfn}

\begin{rmk}
\label{RMK.KCS.is.a.generalization.of.CS}
Karoubi--Chern--Simons forms generalize the standard Chern--Simons forms in the following sense.
Let $X$ be a smooth manifold, $E$ a complex vector bundle over $X$,  $\K:=\C$, and $A:=C^{\infty}(X;\C)$. Let $\na$ be a polynomial path of connections on $\Gamma(E)$ with associated connection $\wtl{\na}$ on $\La_{0}\otimes_{\C}\Gamma(E)$, similar to Example~\ref{rmk:path.of.connections}. The Karoubi--Chern form is related to the standard Chern form via $\frac{1}{(2\pi i)^{k}}\ch^{\Ka}_{k}(\wtl{\na})=\ch_{k}(\wtl{\na}).$
Therefore, the Karoubi--Chern--Simons form is related to the standard Chern--Simons form $\CS(\na)$ via $\frac{1}{(2\pi i)^{k}}\KCS_k(\na)=\CS_{k}(\na)$.
Since $\La_{0}$ is dense inside $C^{\infty}([0,1];\C)$, there is no loss of generality in using polynomial paths of connections instead of smooth paths of connections (cf.\ Example~\ref{rmk:path.of.connections}).
\end{rmk}

\begin{lem} \label{LEM.Transgression.formulae.chk}
Let $\mathcal{D}$ be a polynomial path of connections on $M$ going from $D_{0}$ to $D_{1}$. Then the following facts hold.
\begin{enumerate}
\item
\label{item:dKCS}
$d\KCS(\mathcal{D})=\chk(D_1)-\chk(D_0)$.
\item
\label{item:KCSreversepath}
If $\overline{\mathcal{D}}$ is the reversed path of connections (cf. Lemma~\ref{lem:intervalcrossspacenc}), then
$\KCS(\overline{\mathcal{D}})=-\KCS(\mathcal{D})$.
\item
\label{item:KCSconstantpath}
$\KCS(\mathcal{C})=0$ for a constant polynomial path of connections $\mathcal{C}$, i.e.\ one for which $C_{t}=C_{0}$ for all $t\in\K$.
\item
\label{item:KCSbigon}
If $\mathcal{E}$ is another polynomial of connections satisfying $D_{0}=E_{0}$ and $D_{1}=E_{1}$, then  $\KCS(\mathcal{D})-\KCS(\mathcal{E})\in\mathrm{Im}(d)$.
\item
\label{item:KCSloop}
$\KCS(\mathcal{D})\in \text{Im}(d)$ if $D_0=D_1$.
\end{enumerate}
\end{lem}
\begin{proof}
The exterior derivative of the Karoubi--Chern--Simons form is given by
\[
\begin{split}
d\KCS(\mathcal{D})&=dK\ch^{\Ka}(\wtl{\mathcal{D}})\quad\text{ by Definition~\ref{defn:KCSform}}\\
&=\mathrm{ev}_{1}\big(\ch^{\Ka}(\wtl{\mathcal{D}})\big)-\mathrm{ev}_{0}\big(\ch^{\Ka}(\wtl{\mathcal{D}})\big)-Kd\ch^{\Ka}(\wtl{\mathcal{D}})\quad\text{ by Lemma~\ref{LEM.homotopy.formula}}\\
&=\mathrm{ev}_{1}\big(\ch^{\Ka}(\wtl{\mathcal{D}})\big)-\mathrm{ev}_{0}\big(\ch^{\Ka}(\wtl{\mathcal{D}})\big)\quad\text{by Proposition~\ref{PRP.chk.is.closed}}\\
&=\chk(D_1)-\chk(D_0)\quad\text{ by Lemma~\ref{lem:intervalcrossspacenc}}.
\end{split}
\]
This proves claim (1). Claim (2) follows from item (\ref{item:pathofconnections8}) of Lemma~\ref{lem:intervalcrossspacenc} since $r(dt)=-dt$ (using the notation introduced in that lemma). The proofs of the other claims will be given after
Construction~\ref{rmk:bigonconnectionscurvature}.
\end{proof}

\begin{lem}\label{lem:KCSdirectsumpaths}
Let $M$ and $M'$ be two $A$-modules.
Given two polynomial paths of connections $\mathcal{D}$ on $M$ and $\mathcal{D}'$ on $M'$ (with respect to the same DGA $\Omega_{\bl}(A)$ on top of $A$),
\[
\KCS(\mathcal{D}\oplus\mathcal{D}')=\KCS(\mathcal{D})+\KCS(\mathcal{D}'),
\]
where $\mathcal{D}\oplus\mathcal{D}'$ is the polynomial path of connections on $M\oplus M'$ obtained by the composite
\[
M\oplus M'\xrightarrow{\mathcal{D}\oplus\mathcal{D}'}\big(\La_{0}\otimes_{\K}M\otimes_{A}\Omega_{1}(A)\big)\oplus\big(\La_{0}\otimes_{\K}M'\otimes_{A}\Omega_{1}(A)\big)\cong \La_{0}\otimes_{\K}(M\oplus M')\otimes_{A}\Omega_{1}(A).
\]
\end{lem}
Since the last isomorphism is canonical, it will often be dropped from the notation.
\begin{proof}
This follows from Lemma~\ref{lem:directsumconnections} and the fact that $\widetilde{\mathcal{D}\oplus\mathcal{D}'}(1\otimes m)=\mathcal{D}(m)\oplus\mathcal{D}'(m')=\widetilde{\mathcal{D}}(1\otimes m)\oplus\widetilde{\mathcal{D}'}(1\otimes m')=\widetilde{\mathcal{D}}\oplus\widetilde{\mathcal{D}'}\big(1\otimes(m\oplus m')\big)$ for all $m\in M$ and $m'\in M'$.
\end{proof}

Recall Definition \ref{DFN.homotopy.operator}, where we introduced the complex $\La_{\bl}$ of (formal) polynomials and their differentials in a single variable, which can be thought of as representing the DGA of polynomial differential forms on the unit interval. Also recall that the homotopy operator $\kappa:\La_{\bl}\to\K$ induces a homotopy operator $K:\Omega_{\bl}(\wtl{A})\to\Omega_{\bl}(A)$. Similarly to this setup, $\La_{\bl}\dtsr\La_{\bl}$ is (isomorphic to) the DGA of (formal) polynomials in two (commuting) variables,  which can be thought of as representing the DGA of polynomial differential forms over the unit square. Note that when viewed as polynomials in two variables $s$ and $t$, the relation $ds\wedge dt=-dt\wedge ds=$ holds by the definition of the graded tensor product. Set
\[
\dwtl{A}:=\La_{0}\tsr_{\K}\La_{0}\tsr_{\K}A.
\]
The DGA $\Omega_{\bl}(\dwtl{A}):=\La_\bl\dtsr\La_\bl\dtsr\Om_\bl(A)$ is equipped with the differential determined on homogeneous elements by
\[
d(\eta\tsr\om\tsr\te)=d\eta\tsr\om\tsr\te+(-1)^{|\eta|}\eta\tsr d\om\tsr\te+(-1)^{|\eta|+|\om|} \eta\tsr\om\tsr d\te,
\]
where $|\eta|$, $|\om|$ are the degrees of $\eta$, $\om$, respectively.
There are two homotopy operators $K_{1},K_{2}:\Omega_{\bl}(\dwtl{A})\to\Omega_{\bl}(\wtl{A})$ uniquely determined by the formulas
\beqs
K_{1}(\eta\tsr\om\tsr\te):=\kappa(\eta)\om\tsr\te\quad\text{ and }\quad
K_{2}(\eta\tsr\om\tsr\te):=\kappa(\om)\eta\tsr\te.
\eeqs

Notice that $KK_{1}=KK_{2}$. With this, we get the following homotopy formula.
\begin{lem}
\label{LEM.homotopy.formula.secondary}
For any homogeneous $\eta\tsr\om\tsr\te\in \La_\bl\dtsr\La_\bl\dtsr\Om_\bl(A)$,
\beqs
\begin{split}
(dKK_{1}-KK_{1}d)(\eta\tsr\om\tsr\te)&=
\Big(\kappa(\eta)\big(\omega(1)-\omega(0)\big)-\big(\eta(1)-\eta(0)\big)\kappa(\omega)\Big)\theta\\
&=
\left\{ \begin{array}{ll}
0\quad&\trm{if } \eta\otimes\om\in\La_0\otimes_{\K}\La_0\\
-(\eta(1)-\eta(0))\kappa(\om)\te\quad&\trm{if } \eta\otimes\om\in\La_0\otimes_{\K}\La_1\\
\kappa(\eta)(\om(1)-\om(0))\te\quad&\trm{if } \eta\otimes\om\in\La_1\otimes_{\K}\La_0\\
0 \quad&\trm{if } \eta\otimes\om\in\La_1\otimes_{\K}\La_1\\
\end{array} \right.
\end{split}
\eeqs
A similar result holds upon abelianization.
\end{lem}

\begin{proof}
This follows from
\[
\begin{split}
KK_{1}d(\eta\tsr\om\tsr\te)&=K\Big(\big(\eta(1)-\eta(0)\big)\omega\otimes\theta\Big)-KdK_{1}(\eta\otimes\omega\otimes\theta)\\
&=\big(\eta(1)-\eta(0)\big)\kappa(\omega)\theta+dKK_{1}(\eta\otimes\omega\otimes\theta)-\kappa(\eta)\big(\omega(1)-\omega(0)\big)\theta
\end{split}
\]
by applying Lemma~\ref{LEM.homotopy.formula} twice.
\end{proof}

It is useful to point out that the homotopy formula in Lemma~\ref{LEM.homotopy.formula.secondary} can be expressed as
\[
dKK_{1}-KK_{1}d=\mathrm{ev}_{0}K_{2}+\mathrm{ev}_{1}K_{1}-\mathrm{ev}_{1}K_{2}-\mathrm{ev}_{0}K_{1},
\]
\begin{wrapfigure}{r}{0.20\textwidth}
\vspace{-4mm}
  \centering
    \begin{tikzpicture}
    \draw (0,0) -- node[below]{\scriptsize$\mathrm{ev}_{0}K_{2}$} (2,0) -- node[left]{\rotatebox{-90}{\scriptsize$\mathrm{ev}_{1}K_{1}$}} (2,-2) -- node[above]{\scriptsize$-\mathrm{ev}_{1}K_{2}$} (0,-2) -- node[right]{\rotatebox{-90}{\scriptsize$-\mathrm{ev}_{0}K_{1}$}} cycle;
    \draw[->] (-0.2,-0.2) -- (-0.2,-1.3) node[below]{\small$s$};
    \draw[->] (0.2,0.2) -- (1.3,0.2) node[right]{\small$t$};
    \coordinate (TLbeg) at (0.7,-0.6);
    \coordinate (TR) at (1.4,-0.6);
    \coordinate (BR) at (1.4,-1.4);
    \coordinate (BL) at (0.6,-1.4);
    \coordinate (TLend) at (0.6,-0.7);
    \draw[rounded corners=3pt,->] (TLbeg) -- (TR) -- (BR) -- (BL) -- (TLend);
    \end{tikzpicture}
\end{wrapfigure}
which provides a geometric picture in terms of integration and evaluation on the sides of a square. Namely, $\mathrm{ev}_{0}K_{2}$ integrates along the positive $t$ direction at $s=0$, $\mathrm{ev}_{1}K_{1}$ integrates along the $s$ direction at $t=1$, $-\mathrm{ev}_{1}K_{2}$ integrates along the $-t$ direction (hence, the negative sign) at $s=1$, and $-\mathrm{ev}_{0}K_{1}$ integrates along the $-s$ direction at $t=0$, thus giving a loop around the square.

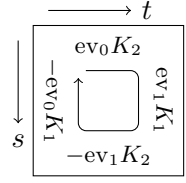
\begin{dfn}
\label{defn:bigonofconnections}
A \define{polynomial square of connections} on the $A$-module $M$ with respect to the DGA $\Omega_{\bl}(A)$ is a map $\Delta:M\to\La_{0}\otimes_{\K}\La_{0}\otimes_{\K}M\otimes_{A}\Omega_{1}(A)$ satisfying a modified form of the Leibniz rule given by $\Delta(ma)=1\otimes1\otimes m\otimes da +\Delta(m)a$ for all $a\in A$ and $m\in M$. Equivalently, $\Delta$ satisfies the condition that the composite $\Delta_{(s,t)}:=M\xrightarrow{\Delta}\Lambda_{0}\otimes_{\K}\Lambda_{0}\otimes_{\K}M\otimes_{A}\Omega_{1}(A)\xrightarrow{\mathrm{ev}_{(s,t)}}M\otimes_{A}\Omega_{1}(A)$ is a connection on $M$ for every $s,t\in\K$.
Here, there are two evaluation maps, which will be written as%
\footnote{The $s$ variable will always be used for the left $\La_{0}$ factor, while $t$ will be used for the right one, so a more precise way of writing these would be as $\mathrm{ev}_{s}\tsr\mathrm{id}$ and $\mathrm{id}\tsr\mathrm{ev}_{t}$, though we avoid this cluttered notation.}
$\mathrm{ev}_{s}$ and $\mathrm{ev}_{t}$ satisfying $\mathrm{ev}_{s}\circ\mathrm{ev}_{t}=\mathrm{ev}_{t}\circ\mathrm{ev}_{s}$ so that this last map is written as $\mathrm{ev}_{(s,t)}$.
Let $\mathcal{D},\mathcal{E}:M\to\La_{0}\tsr_{\K}M\tsr_{A}\Omega_{1}(A)$ be two polynomial paths of connections such that $\mathcal{D}_{0}=\mathcal{E}_{0}$ and $\mathcal{D}_{1}=\mathcal{E}_{1}$. A \define{polynomial bigon of connections from $\mathcal{D}$ to $\mathcal{E}$} is a polynomial square of connections $\Delta:M\to\La_{0}\otimes_{\K}\La_{0}\otimes_{\K}M\otimes_{A}\Omega_{1}(A)$ such that $\mathrm{ev}_{s=0}(\Delta)=\mathcal{D}$ and $\mathrm{ev}_{s=1}(\Delta)=\mathcal{E}$, where the evaluation maps are for the left $\La_{0}$ factor and where the diagrams
\[
\xy0;/r.25pc/:
(20,7.5)*+{M}="M";
(-20,-7.5)*+{M\tsr_{A}\Omega_{\bl}(A)}="MA";
(20,-7.5)*+{\Lambda_{0}\tsr_{\K}M\tsr_{A}\Omega_{\bl}(A)}="LMA";
{\ar"M";"LMA"^{\mathrm{ev}_{t=0}(\Delta)}};
{\ar"M";"MA"_{\mathcal{E}_{0}=\mathcal{D}_{0}}};
{\ar@{^{(}->}"MA";"LMA"_(0.45){i_{M}}};
\endxy
\quad\text{ and }\quad
\xy0;/r.25pc/:
(-20,7.5)*+{M}="M";
(20,-7.5)*+{M\tsr_{A}\Omega_{\bl}(A)}="MA";
(-20,-7.5)*+{\Lambda_{0}\tsr_{\K}M\tsr_{A}\Omega_{\bl}(A)}="LMA";
{\ar"M";"LMA"_{\mathrm{ev}_{t=1}(\Delta)}};
{\ar"M";"MA"^{\mathcal{E}_{1}=\mathcal{D}_{1}}};
{\ar@{_{(}->}"MA";"LMA"^(0.45){i_{M}}};
\endxy
\]
both commute.
\end{dfn}

\begin{lem}
\label{lem:polybigonsexist}
Let $\mathcal{D},\mathcal{E}$ be two polynomial paths of connections such that $\mathcal{D}_{0}=\mathcal{E}_{0}$ and $\mathcal{D}_{1}=\mathcal{E}_{1}$. Then there exists a polynomial bigon of connections from $\mathcal{D}$ to $\mathcal{E}$.
\end{lem}

\begin{proof}
This lemma follows from the following basic fact about polynomials. Let $p$ and $q$ be two polynomials on the unit interval satisfying $p(0)=q(0)=:a$ and $p(1)=q(1)=:b$. Then there exists a polynomial $r$ on the unit square such that
\[
r(0,t)=p(t),\quad r(1,t)=q(t),\quad r(s,0)=a, \quad r(s,1)=b \qquad\forall\; s,t\in[0,1].
\]
Indeed, the polynomial
\[
[0,1]\times[0,1]\ni (s,t)\mapsto(1-s)p(t)+sq(t)
\]
satisfies this property. Let $u$ be the element of $\La_{0}$ representing the identity polynomial, namely $u(s)=s$ for all $s\in[0,1]$. Then
\[
\begin{split}
M&\xrightarrow{\Delta}\La_{0}\tsr_{\K}\La_{0}\tsr_{\K}M\tsr_{A}\Omega_{1}(A)\\
m&\mapsto(1-u)\tsr\mathcal{D}(m)+u\tsr\mathcal{E}(m)
\end{split}
\]
defines such a polynomial bigon of connections.
\end{proof}

\begin{cns}
\label{rmk:bigonconnectionscurvature}
Using the notation of Definition~\ref{defn:bigonofconnections}, by a construction analogous to the one in Lemma~\ref{lem:intervalcrossspacenc}, a polynomial square of connections $\Delta$ defines a connection
\[
\wtl{\Delta}:
\dwtl{M}\to\dwtl{M}\otimes_{\dwtl{A}}\Omega_{1}(\dwtl{A})
\]
on the
$\dwtl{A}$-module $\dwtl{M}:=\La_{0}\tsr_{\K}\La_{0}\tsr_{\K}M$. This connection is defined by
\[
\dwtl{M}\ni p\otimes q\otimes m\mapsto dp\otimes q\otimes m+p\otimes dq\otimes m+(p\otimes q)\Delta(m)
\]
and makes use of the canonical isomorphism
\[
\begin{split}
\dwtl{M}\tsr_{\dwtl{A}}\Omega_{k}\big(\dwtl{A}\big)&\cong\Big(\La_{1}\otimes_{\K}\La_{1}\tsr_{\K}M\tsr_{A}\Omega_{k-2}(A)\Big)\oplus\Big(\La_{1}\otimes_{\K}\La_{0}\tsr_{\K}M\tsr_{A}\Omega_{k-1}(A)\Big)\\
&\oplus\Big(\La_{0}\otimes_{\K}\La_{1}\tsr_{\K}M\tsr_{A}\Omega_{k-1}(A)\Big)\oplus\Big(\La_{0}\otimes_{\K}\La_{0}\tsr_{\K}M\tsr_{A}\Omega_{k}(A)\Big)
\end{split},
\]
which is valid for all $k\in\N$ (negative degree spaces are set to be zero).
Using the two evaluations, one can define the connections $\Delta_{\downarrow t}:=\mathrm{ev}_{t}\circ\wtl{\Delta}\circ i_{M}$ and $\Delta_{s\rightarrow}:=\mathrm{ev}_{s}\circ\wtl{\Delta}\circ j_{M}$ on $\wtl{M}$, where $i_{M}:\La_{0}\tsr_{\K}M\to\dwtl{M}$ and $j_{M}:\La_{0}\tsr_{\K}M\to\dwtl{M}$ are given by $i_{M}(q\otimes m):=q\tsr1\tsr m$ and $j_{M}(p\tsr m):=1\tsr p\tsr m$. One should interpret $\Delta_{\downarrow t}$ as the connection for a fixed $t$ but varying in the down ($s$) direction. Similarly, $\Delta_{s\rightarrow}$ is the connection at a fixed $s$ and varying in the right ($t$) direction. The following picture may help visualize all these different connections
\[
    \begin{tikzpicture}[scale=1.25]
    \draw (0,0) -- node[below]{\scriptsize$\Delta_{0\rightarrow}$} (2,0) -- node[left]{\scriptsize$\Delta_{\downarrow1}$} (2,-2) -- node[above]{\scriptsize$\Delta_{1\rightarrow}$} (0,-2) -- node[right]{\scriptsize$\Delta_{\downarrow0}$} cycle;
    \draw[->] (-0.2,-0.2) -- (-0.2,-1.3) node[below]{\small$s$};
    \draw[->] (0.2,0.2) -- (1.3,0.2) node[right]{\small$t$};
    \node at (1,-1) {$\wtl{\Delta}$};
    \end{tikzpicture}
\]
The curvature of $\wtl{\Delta}$ satisfies
\[
\mathrm{ev}_{t}\circ\wtl{\Delta}^{2}\circ i_{M}=\Delta_{\downarrow t}^{2},\quad
\mathrm{ev}_{s}\circ\wtl{\Delta}^{2}\circ j_{M}=\Delta_{s\rightarrow}^{2},\quad
\mathrm{ev}_{(s,t)}\circ\wtl{\Delta}^{2}\circ i\!j_{M}=\Delta_{(s,t)}^{2}\qquad\forall\;s,t\in\K,
\]
where $i\!j_{M}:M\tsr_{A}\Omega_{\bl}(A)\to\dwtl{M}\tsr_{\dwtl{A}}\Omega_{\bl}\big(\dwtl{A}\big)$ is the inclusion sending $m\otimes\omega$ to $1\otimes1\otimes m\otimes\omega$. As a corollary of the first two of these equations,
\[
\mathrm{ev}_{t}\big(\chk(\wtl{\Delta})\big)=\chk(\Delta_{\downarrow t})
\quad\text{ and }\quad
\mathrm{ev}_{s}\big(\chk(\wtl{\Delta})\big)=\chk(\Delta_{s\rightarrow})
\qquad\forall\;s,t\in\K
\]
as elements of $\Om_\bl(\wtl{A})_{\text{ab}}$.
\end{cns}

\begin{proof}[Proof of the rest of Lemma~\ref{LEM.Transgression.formulae.chk}]
{\color{white}{you found me!}}
\begin{enumerate}
\setcounter{enumi}{2}
\item
Let $\mathcal{C}$ be a constant polynomial path of connections. Then $\chk(\wtl{\mathcal{C}})\in\Omega_{\bl}(\wtl{A})_{\ab}\cong\La_{\bl}\dtsr\Omega_{\bl}(A)_{\ab}$ factors through $\Omega_{\bl}(A)_{\ab}$. In more detail, $\mathcal{C}$ constant means that $\mathcal{C}$ is equal to the composite $M\xrightarrow{C}M\otimes_{A}\Omega_{1}(A)\xrightarrow{i_{M}}\La_{0}\otimes_{\K}M\otimes_{A}\Omega_{1}(A)$, where $C=\mathrm{ev}_{t}(\mathcal{C})$ for any $t\in\K$ (since they are all equal). Hence, the curvature of $\wtl{\mathcal{C}}$ is obtained from $\wtl{\mathcal{C}}^{2}(1\otimes m)=\wtl{\mathcal{C}}(\mathcal{C}(m))=\wtl{\mathcal{C}}(1\otimes C(m))=\mathcal{C}(C(m))=1\otimes C^{2}(m)$. Hence, $\KCS(\mathcal{C})=K\ch^{\Ka}(\wtl{\mathcal{C}})=0$.
One can also prove this more simply using item (\ref{item:KCSreversepath}) of Lemma~\ref{LEM.Transgression.formulae.chk} and the fact that the reverse of a constant path is itself.
\item
Consider a polynomial bigon of connections $\Delta:M\to\La_{0}\otimes_{\K}\La_{0}\otimes_{\K}M\otimes_{A}\Omega_{1}(A)$ from $\mathcal{D}$ to $\mathcal{E}$, with $\mathcal{D}_{0}=\mathcal{E}_{0}$ and $\mathcal{D}_{1}=\mathcal{E}_{1}$ (one exists by Lemma~\ref{lem:polybigonsexist}). The fact that $\Delta$ is a bigon from $\mathcal{D}$ to $\mathcal{E}$ means that $\mathrm{ev}_{s=0}\Delta=\mathcal{D}$ and $\mathrm{ev}_{s=1}\Delta=\mathcal{E}$.
By Lemma~\ref{lem:intervalcrossspacenc}, $\mathcal{D}$ and $\mathcal{E}$ define connections $\wtl{\mathcal{D}},\wtl{\mathcal{E}}:\wtl{M}\to\wtl{M}\otimes_{\wtl{A}}\Omega_{1}(\wtl{A})$, and
by Construction~\ref{rmk:bigonconnectionscurvature}, $\Delta$ defines a connection $\wtl{\Delta}:\dwtl{M}\to\dwtl{M}\otimes_{\dwtl{A}}\Omega_{1}(\dwtl{A})$.
We also note that
\begin{equation}
\label{eq:evtKcommutes}
\mathrm{ev}_{t}\circ K_{1}=K\circ\mathrm{ev}_{t}
\qquad\text{ and }\qquad
\mathrm{ev}_{s}\circ K_{2}=K\circ\mathrm{ev}_{s}
\qquad\forall\;s,t\in\K.
\end{equation}
For the moment, set $\Theta:=\chk(\wtl{\Delta})$. The exterior derivative of $KK_{1}\Theta$ is given by
\[
\begin{split}
d&KK_{1}\Te=\mathrm{ev}_{0}K_{2}\Theta+\mathrm{ev}_{1}K_{1}\Theta-\mathrm{ev}_{1}K_{2}\Theta-\mathrm{ev}_{0}K_{1}\Theta+KK_{1}d\Theta\quad\text{ by Lemma~\ref{LEM.homotopy.formula.secondary}}\\
&=K\mathrm{ev}_{s=0}\Theta+K\mathrm{ev}_{t=1}\Theta-K\mathrm{ev}_{s=1}\Theta-K\mathrm{ev}_{t=0}\Theta\quad\text{ by Proposition~\ref{PRP.chk.is.closed} and~(\ref{eq:evtKcommutes})}\\
&=K\chk(\Delta_{0\rightarrow})+K\chk(\Delta_{\downarrow 1})-K\chk(\Delta_{1\rightarrow})-K\chk(\Delta_{\downarrow 0})\quad\text{by Construction~\ref{rmk:bigonconnectionscurvature}}\\
&=K\chk(\Delta_{0\rightarrow})-K\chk(\Delta_{1\rightarrow})\quad\text{ by Lemma~\ref{LEM.Transgression.formulae.chk} part (\ref{item:KCSconstantpath})}.
\end{split}
\]
Hence,
\[
dKK_{1}\chk(\wtl{\Delta})
=K\chk(\Delta_{0\rightarrow})-K\chk(\Delta_{1\rightarrow})
=\KCS(\mathcal{D})-\KCS(\mathcal{E}),
\]
which not only proves the claim but provides an explicit differential form realizing exactness of the right-hand-side.
\item
The case where $\mathcal{D}$ satisfies $D_{0}=D_{1}$ can be obtained from the previous case (\ref{item:KCSbigon}) by taking $\mathcal{E}:=\mathcal{C}$ to be a constant polynomial path of connections and applying part (\ref{item:KCSconstantpath}) of this lemma. \qedhere
\end{enumerate}
\end{proof}

\begin{rmk}
A common argument in the smooth setting is to prove item (\ref{item:KCSbigon}) of Lemma~\ref{LEM.Transgression.formulae.chk} using the other items, including (\ref{item:KCSloop}). This is because for smooth paths, one can take two paths with the same source and target and turn them into a loop (with some appropriate restrictions to be discussed momentarily). However, polynomial paths of connections \emph{do not compose} in this sense. This is because the concatenation of two polynomials is not a polynomial in general. A similar situation happens for smooth paths actually. However, there are at least two succesful methods for dealing with the smooth case. One is to use smooth paths with sitting instants. Another is to use piecewise smooth paths. Neither of these approaches seem to be available for polynomial paths. The first case is not an option because polynomials are analytic and so a polynomial with sitting instants is necessarily a constant. The second case is not a simple option in the algebraic setting either because it would be awkward to replace $\La_{0}$ to include piecewise polynomials. In the smooth setting, one typically uses a larger class of paths (such as continuous paths) and then places restrictions on those paths to define piecewise smooth paths as an appropriate subset. We do not have an obvious candidate that replaces continuous paths.
\end{rmk}

\begin{nta}\label{NTA.KCS.of.two.connections}
Let $M\in \Fgp(A)$ and let $D_0,D_1$ be connections on $M$ with respect to some DGA on top of $A$. By Remark~\ref{rmk:spaceofconnectionsisaffine}, there exists a polynomial path of connections $\mathcal{D}$ interpolating from $D_{0}$ to $D_{1}$.
Let $\KCS(D_0,D_1)$ denote the equivalence class of the $\KCS$-form of the path of connections $\mathcal{D}$ up to $\textrm{Im}(d)$. This does not depend on $\mathcal{D}$ by Lemma \ref{LEM.Transgression.formulae.chk} (\ref{item:KCSbigon}). By $\KCS_k(D_0,D_1)$ we mean the $(2k-1)^{\text{st}}$ degree
term in $\KCS(D_0,D_1)$.
\end{nta}

\begin{lem}\label{lem:KCStriangle}
Let $D_{1},D_{2},D_{3}$ be three connections on an $A$-module $M$ with respect to some DGA on top of $A$. Then
\[
\KCS(D_{1},D_{3})=\KCS(D_{1},D_{2})+\KCS(D_{2},D_{3}).
\]
\end{lem}
\begin{proof}
First recall that given any distinct $k$ points $t_{1},\dots,t_{k}\in\Q$, there exist polynomials $p_{1},\dots,p_{k}\in\La_{0}$ such that
$p_{i}(t_{j})=\delta_{ij}$ for all $i,j\in\{1,\dots,k\}$. In this case, $k=3$ and $t_{1}=0, t_{2}=\frac{1}{2}, t_{3}=1$.
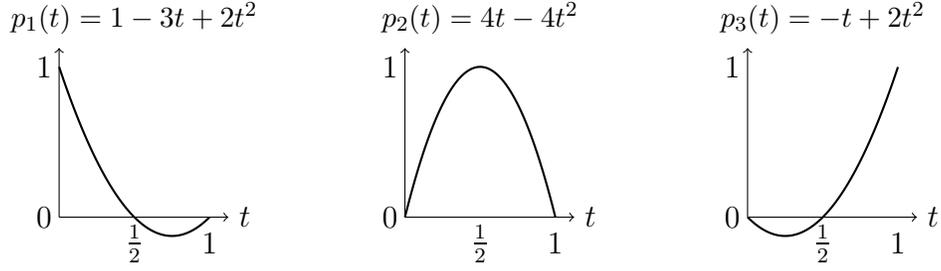
\begin{center}
\begin{tikzpicture}
      \draw[->,thin] (0,0) -- (2.25,0) node[right] {$t$};
      \draw[->,thin] (0,0) -- (0,2.25);
      \node at (-0.2,2) {$1$};
      \node at (-0.2,0) {$0$};
      \node at (1,-0.35) {$\frac{1}{2}$};
      \node at (2,-0.35) {$1$};
      \draw[samples=50,domain=0:2,thick,smooth,variable=\t] plot (\t,{2-3*(\t)+pow(\t,2)});
      \node at (1,2.65) {\small$p_{1}(t)=1-3t+2t^2$};
\end{tikzpicture}
\qquad\quad
\begin{tikzpicture}
      \draw[->,thin] (0,0) -- (2.25,0) node[right] {$t$};
      \draw[->,thin] (0,0) -- (0,2.25);
      \node at (-0.2,2) {$1$};
      \node at (-0.2,0) {$0$};
      \node at (1,-0.35) {$\frac{1}{2}$};
      \node at (2,-0.35) {$1$};
      \draw[samples=50,domain=0:2,thick,smooth,variable=\t] plot (\t,{4*(\t)-2*pow(\t,2)});
      \node at (1,2.65) {\small$p_{2}(t)=4t-4t^2$};
\end{tikzpicture}
\qquad\quad
\begin{tikzpicture}
      \draw[->,thin] (0,0) -- (2.25,0) node[right] {$t$};
      \draw[->,thin] (0,0) -- (0,2.25);
      \node at (-0.2,2) {$1$};
      \node at (-0.2,0) {$0$};
      \node at (1,-0.35) {$\frac{1}{2}$};
      \node at (2,-0.35) {$1$};
      \draw[samples=50,domain=0:2,thick,smooth,variable=\t] plot (\t,{-(\t)+pow(\t,2)});
      \node at (1,2.65) {\small$p_{3}(t)=-t+2t^2$};
\end{tikzpicture}
\end{center}
Let $\mathcal{D}$ be the path of connections determined by
\[
\mathcal{D}(m):=p_{1}\otimes D_{1}(m)+p_{2}\otimes D_{2}(m)+p_{3}\otimes D_{3}(m).
\]
Then by construction, $\mathrm{ev}_{n}\mathcal{D}=D_{2n+1}$ for all $n\in\{0,\frac{1}{2},1\}$. In particular $\mathcal{D}$ is a polynomial path of connections interpolating from $D_{1}$ to $D_{3}$ passing through $D_{2}$. Second, we generalize the homotopy operator $\kappa:\La_{\bl}\to\K$ to a family $\kappa_{[a,b]}:\La_{\bl}\to\K$ depending on $a,b\in\Q$ with $a<b$, by setting $\kappa_{[a,b]}(p):=0$ for all $p\in\La_{0}$ and $\kappa_{[a,b]}(\omega):=\int_{a}^{b}\omega$ for all $\omega\in\La_{1}.$ The usual $\kappa$ is obtained with $a=0$ and $b=1$. Some immediate properties of this family are the familiar properties of the integral from calculus such as $\kappa_{[a,c]}=\kappa_{[a,b]}+\kappa_{[b,c]}$ with $a<b<c$. Another property is affine covariance (a special case of reparametrization covariance), which says
\begin{equation}
\label{eq:kappaab}
\kappa_{[a,b]}\big(p(t)dt\big)=\beta\kappa_{[\beta^{-1}(a-\alpha),\beta^{-1}(b-\alpha)]}\big(p(\alpha+\beta t)dt\big)
\end{equation}
for all $\alpha,\beta\in\Q$ with $\beta>0$.
Similarly, just as $\kappa$ extends to $K$ on $\La_{\bl}\dtsr\Omega_{\bl}(A)$-valued forms, $\kappa_{[a,b]}$ extends to $K_{[a,b]}$ and obeys similar properties. Hence (all equalities
take place in $\Omega_{\bl}(A)/\mathrm{Im}(d)$),
\begin{equation}
\label{eqn:KCStriangle}
\begin{split}
\KCS(D_{1},D_{3})&=\KCS(\mathcal{D})=K\chk(\widetilde{\mathcal{D}})=K_{[0,\frac{1}{2}]}\chk(\widetilde{\mathcal{D}})+K_{[\frac{1}{2},1]}\chk(\widetilde{\mathcal{D}})\\
&=K\chk(\widetilde{\mathcal{D}_{12}})+K\chk(\widetilde{\mathcal{D}_{23}})=\KCS(D_{1},D_{2})+\KCS(D_{2},D_{3}),
\end{split}
\end{equation}
where $\mathcal{D}_{12}$ and $\mathcal{D}_{23}$ are the paths of connections defined by reparametrizing $\mathcal{D}$ with double speed, with the latter shifted to begin at $0$, namely
\[
\widetilde{\mathcal{D}_{12}}(m):=\sum_{j=1}^{3}p'_{j}\otimes D_{j}(m)
\quad
\text{ and }
\quad
\widetilde{\mathcal{D}_{23}}(m):=\sum_{j=1}^{3}p''_{j}\otimes D_{j}(m),
\]
where
\[
p'_{j}(t):=p_{j}\left(\frac{1}{2}t\right)
\quad\text{ and }\quad
p''_{j}(t):=p_{j}\left(\frac{1}{2}t+\frac{1}{2}\right).
\]
The second last equality in~(\ref{eqn:KCStriangle}) still requires further justification, and it suffices to consider one of the two terms, such as $K_{[0,\frac{1}{2}]}\chk(\widetilde{\mathcal{D}})=K\chk(\widetilde{\mathcal{D}_{12}})$, the latter of which equals $K_{[0,1]}\chk(\widetilde{\mathcal{D}_{12}})$ by definition. The only terms contributing to the integration from $K$ are those that involve derivatives of the $p'_{j}$. Since these produce an additional factor of $\beta=\frac{1}{2}$ as an overall coefficient, this factor automatically produces the one from~(\ref{eq:kappaab}), exhibiting this equality.
\end{proof}

\begin{dfn}\label{dfn:inducedpathofconnections}
Let $B\xrightarrow{\psi}A$ be an algebra map, let $N$ be a $B$-module, let $\Omega$ be a DGA on top of $\Alg$, and let $\mathcal{D}:N\to\La_{0}\tsr_{\K}N\tsr_{B}\Omega_{1}(B)$ be a polynomial path of connections. Then the \define{induced (polynomial) path of connections} $\mathcal{D}_{\psi}:N\tsr_{B}A\to\La_{0}\tsr_{\K}N\tsr_{B}\Omega_{1}(A)\cong\La_{0}\tsr_{\K}(N\tsr_{B}A)\tsr_{A}\Omega_{1}(A)$ is given by the assignment
\[
N\tsr_{B}A\ni n\otimes a\mapsto\Big((\mathrm{id}\otimes_{B}\psi)(\mathcal{D}(n))\Big)a+1\otimes n\otimes da.
\]
In the special case of a cocartesian morphism of the form $(A,M)\xrightarrow{(\mathrm{id}_{A},\varphi)}(A,L)$, the map $\varphi:M\to L$ is an isomorphism (cf.\ Definition~\ref{dfn:inducedconnection}).  In this case, if $\mathcal{D}$ is a polynomial path of connections on $L$, the \define{pullback (polynomial) path (of connections)} $\vph^*\mathcal{D}$ on $M$ is the path obtained via the composite $M\xrightarrow{\vph}L\xrightarrow{\mathcal{D}}\La_{0}\otimes_{\K}L\otimes_{A}\Omega_{1}(A)\xrightarrow{\text{id}\otimes\vph^{-1}\otimes\text{id}}\La_{0}\otimes_{\K}M\otimes_{A}\Omega_{1}(A)$.
\end{dfn}

It is a simple exercise (compare Lemma~\ref{lem:pullbackconnection}) to verify that $\mathcal{D}_{\psi}$ in Definition~\ref{dfn:inducedpathofconnections} indeed defines a polynomial path of connections on the extension of scalars $N\tsr_{B}A$ with respect to the DGA $\Omega_{\bl}(A)$. The following lemma describes some naturality properties of the above constructions.

\begin{lem} \label{LEM.KCS.invariant.under.pullback}
Let $M$ and $L$ be $A$-modules and let $\vph:M\ra L$ be an $A$-module isomorphism (equivalently a cocartesian morphism of the form $(\mathrm{id}_{A},\varphi):(A,M)\to(A,L)$). Also, let $B$ be an algebra, $N$ a $B$-module, $B\xrightarrow{\psi}A$ an algebra map, and $\Omega$ a DGA on top of $\Alg$.
\begin{enumerate}
\item\label{item:pullbackconnection}
If $D$ is a connection on $L$, then $\chk(\vph^*D)=\chk(D)$ (cf.\ Definition~\ref{dfn:inducedconnection}).
\item\label{item:pullbackpathconnection}
If $\mathcal{D}$ is a polynomial path of connections on $L$, then $\KCS(\vph^*\mathcal{D})=\KCS(\mathcal{D})$.
\item\label{item:inducedconnection}
If $D$ is a connection on $N$, then $\chk(D_{\psi})=\psi(\chk(D))$ (cf.\ Lemma~\ref{lem:pullbackconnection} and Definition~\ref{dfn:inducedconnection}).
\item\label{item:inducedpathconnection}
If $\mathcal{D}$ is a polynomial path of connections on $N$, then $\KCS(\mathcal{D}_{\psi})=\psi(\KCS(\mathcal{D}))$ (cf.\ Definition~\ref{dfn:inducedpathofconnections}).
\end{enumerate}
\end{lem}
\begin{proof} We first prove the first pair of claims. Since $\vph^*D=(\vph^{-1}\tsr\text{id})\circ D\circ (\vph\tsr\text{id})$, squaring gives $(\vph^{-1}\tsr\text{id})\circ D^2\circ (\vph\tsr\text{id})$. Since the trace is cyclic (cf.\ Example~\ref{exa.trace.cyclic}), $\chk(\vph^*D)=\chk(D)$, which proves (1). Set $\La_{0}\tsr_{\K}M\xrightarrow{\phi:=\text{id}\tsr\vph}\La_{0}\tsr_{\K}L$. Then a simple calculation shows that $\wtl{\vph^*\mathcal{D}}=\phi^{*}\wtl{\mathcal{D}}$ (using the notation of Lemma~\ref{lem:intervalcrossspacenc}). Thus, (\ref{item:pullbackpathconnection}) follows from (\ref{item:pullbackconnection}).

We now prove the second pair of claims. For item (\ref{item:inducedconnection}),
\[
\begin{split}
\chk(D_{\psi})&=\sum_{k}\frac{1}{k!}\tr(R_{D_{\psi}}^{2k})
\overset{\text{Lem~\ref{lem:incucedcurvature}}}{=\joinrel=\joinrel=\joinrel=\joinrel=}\sum_{k}\frac{1}{k!}\tr\big(\mathfrak{E}_{\psi}(R_{D}^{2k})\big)\\
&\overset{\text{Lem~\ref{lem:ringmapsandtrace}}}{=\joinrel=\joinrel=\joinrel=\joinrel=}\sum_{k}\frac{1}{k!}\psi\big(\tr(R_{D}^{2k})\big)
=\psi\big(\chk(D)\big).
\end{split}
\]
In this calculation, Lemma~\ref{lem:ringmapsandtrace} applies due to Example~\ref{ex:ringofevenforms}. In more detail, in terms of the notation of Lemma~\ref{lem:ringmapsandtrace}, set $R:=\Omega_{\even}(B)$ and $S:=\Omega_{\even}(A)$, set $\psi:R\to S$ to be the map $\psi:B\to A$ extended to forms (which can be done since $\Omega$ is a DGA on top of $\Alg$), and set $F:=N\otimes_{B}\Omega_{\even}(B)$, noting that $F\otimes_{R}S\cong N\otimes_{B}\Omega_{\even}(A)$.
Claim (\ref{item:inducedpathconnection}) follows from the identity $(\mathrm{id}_{\La_{0}}\otimes\psi)\chk\big(\wtl{\mathcal{D}}\big)=\chk\big(\wtl{\mathcal{D}_{\psi}}\big)$, whose proof we omit since it is similar to calculations that have already been done.
\end{proof}

\section{Noncommutative differential $K$-theory}\label{SEC.NCDKT.construction}
In this section, we construct a noncommutative differential $K$-group and verify that our construction recovers the differential $K$-theory of a smooth manifold when the algebra is complex-valued smooth functions. Furthermore, we shall see that our $K$-group fits into a noncommutative analogue of the differential cohomology hexagon diagram.

\subsection{Noncommutative differential $K$-theory with Karoubi's Chern character}

Let $A$ be a unital $\K$-algebra and let $(M,D)$ be an $A$-module with connection $D$ with respect to a DGA $\Om_{\bl}(A)$ on top of $A$. Set $(\Om_{\bl}(A)_{\ab})_{\odd}:=\Dsum_{k=0}^\infty \Om_{2k+1}(A)/[\Omega_{\bl}(A),\Omega_{\bl}(A)]_{2k+1}$. Let $\om\in(\Om_{\bl}(A)_{\ab})_{\odd}/\mathrm{Im}(d)$. Then the triple $(M,D,\om)$ will be called a \define{$\widehat K_0$-generator} of $\Om_\bl(A)$.

\begin{dfn}
Two $\widehat K_0$-generators of $\Om_\bl(A)$, $(M_0,D_0,\om_0)$ and $(M_1,D_1,\om_1)$, are \define{$\KCS$-equivalent}, denoted by $\sim_{\KCS}$, if there is an $A$-module $N$ with connection $D$ and an $A$-module isomorphism $M_0\dsum N\srl{\vph}\ra M_1\dsum N$ such that $\KCS(D_0\dsum D,\vph^*(D_1\dsum D))=(\om_0-\om_1)\!\!\mod\mathrm{Im}(d)$ (cf.\ Lemma~\ref{lem:directsumconnections}, Notation~\ref{NTA.KCS.of.two.connections}, and Lemma~\ref{LEM.KCS.invariant.under.pullback}).
\end{dfn}

\begin{lem} The relation $\sim_{\KCS}$ is an equivalence relation. The direct sum is well-defined on $\KCS$-equivalence classes, where the direct sum is defined on representatives as $(M,D,\omega)\oplus(M',D',\omega'):=(M\oplus M',D\oplus D',\omega+\omega')$.
\end{lem}

\begin{proof}
Reflexivity $(M_0,D_0,\om_0)\sim_{\KCS}(M_0,D_0,\om_0)$ holds since one can take $N:=0$,  $D:=0$, and $\varphi:=\mathrm{id}$. Indeed, by Lemma~\ref{LEM.Transgression.formulae.chk} (\ref{item:KCSloop}), $0=\KCS(D_{0}\oplus0,D_{0}\oplus0)=(\omega_{0}-\omega_{0})\mod\mathrm{Im}(d)$.
Symmetry holds by taking the inverse module isomorphism. In more detail, suppose $(M_0,D_0,\om_0)\sim_{\KCS}(M_1,D_1,\om_1)$. Then there exist an $A$-module $N$ with connection $D$ and an $A$-module isomorphism $M_{0}\oplus N\xrightarrow{\varphi}M_{1}\oplus N$ such that $\KCS(D_0\dsum D,\vph^*(D_1\dsum D))=(\om_0-\om_1)\!\!\mod\mathrm{Im}(d)$. As a result,
\[
\begin{split}
\KCS(D_1\dsum D,(\vph^{-1})^*(D_0\dsum D))&=\KCS(\varphi^*(D_1\dsum D),D_0\dsum D)\quad\text{ by Lemma~\ref{LEM.KCS.invariant.under.pullback} (\ref{item:pullbackpathconnection})}\\
&=-\KCS(D_0\dsum D,\varphi^*(D_1\dsum D))\quad\text{by Lemma~\ref{LEM.Transgression.formulae.chk} item~\ref{item:KCSreversepath}}\\
&=(\omega_{1}-\omega_{0})\mod\mathrm{Im}(d).
\end{split}
\]
As for transitivity, suppose $(M_{0},D_{0},\omega_{0})\sim_{\KCS}(M_{1},D_{1},\omega_{1})$ and $(M_{1},D_{1},\omega_{1})\sim_{\KCS}(M_{2},D_{2},\omega_{2})$.
By assumption, there exist $A$-modules with connections $(N_{01},D_{01})$, $(N_{12},D_{12})$ and module isomorphisms $M_{0}\oplus N_{01}\xrightarrow{\varphi_{01}}M_{1}\oplus N_{01}$, $M_{1}\oplus N_{12}\xrightarrow{\varphi_{12}}M_{2}\oplus N_{12}$ such that
\[
\KCS\big(D_{0}\oplus D_{01},\varphi_{01}^*(D_{1}\oplus D_{01})\big)=(\omega_{0}-\omega_{1})\mod\mathrm{Im}(d)
\quad\text{ and}\]
\[
\KCS\big(D_{1}\oplus D_{12},\varphi_{12}^*(D_{2}\oplus D_{12})\big)=(\omega_{1}-\omega_{2})\mod\mathrm{Im}(d).
\]
Set $N_{02}:=N_{01}\oplus N_{12}$, $D_{02}:=D_{01}\oplus D_{12}$, and let $\varphi_{02}:M_{0}\oplus N_{02}\xrightarrow{\varphi_{02}}M_{2}\oplus N_{02}$ be the composite
\[
\xy0;/r.25pc/:
(-60,7)*+{M_{0}\oplus N_{01}\oplus N_{12}}="1";
(-40,-7)*+{M_{1}\oplus N_{01}\oplus N_{12}}="2";
(0,-7)*+{N_{01}\oplus M_{1}\oplus N_{12}}="3";
(40,-7)*+{N_{01}\oplus M_{2}\oplus N_{12}}="4";
(60,7)*+{M_{2}\oplus N_{01}\oplus N_{12}}="5";
{\ar"1";"2"^{\varphi_{01}\oplus\mathrm{id}_{12}}};
{\ar"2";"3"^{\sigma\oplus\mathrm{id}_{12}}};
{\ar"3";"4"^{\mathrm{id}_{01}\oplus\varphi_{12}}};
{\ar"4";"5"^{\sigma\oplus\mathrm{id}_{12}}};
\endxy
,
\]
where $\sigma$ is the swap isomorphism.
Then
\[
\begin{split}
\KCS&\big(D_{0}\oplus D_{02},\varphi_{02}^{*}(D_{2}\oplus D_{02})\big)\overset{\text{Lem~\ref{lem:KCStriangle}}}{=\joinrel=\joinrel=\joinrel=\joinrel=}\KCS\big(D_{0}\oplus D_{02},(\varphi_{01}\oplus\mathrm{id}_{12})^{*}(D_{1}\oplus D_{01}\oplus D_{12})\big)\\
&+\KCS\big((\varphi_{01}\oplus\mathrm{id}_{12})^{*}(D_{1}\oplus D_{01}\oplus D_{12}),\varphi_{02}^{*}(D_{2}\oplus D_{02})\big)\\
&\overset{\text{Lem~\ref{LEM.KCS.invariant.under.pullback}}}{=\joinrel=\joinrel=\joinrel=\joinrel=}\KCS\big(D_{0}\oplus D_{02},\varphi_{01}^{*}(D_{1}\oplus D_{01})\oplus D_{12}\big)\\
&+\KCS\big(D_{01}\oplus D_{1}\oplus D_{12},(\mathrm{id}_{01}\oplus\varphi_{12})^*(D_{01}\oplus D_{2}\oplus D_{12})\big)\\
&\overset{\text{Lem~\ref{lem:KCSdirectsumpaths}}}{=\joinrel=\joinrel=\joinrel=\joinrel=}\KCS\big(D_{0}\oplus D_{01},\varphi_{01}^{*}(D_{1}\oplus D_{01})\big)+\KCS(D_{12},D_{12})\\
&+\KCS(D_{01},D_{01})+\KCS(D_{1}\oplus D_{12},\varphi_{12}^*(D_{2}\oplus D_{12})\big)\\
&=(\omega_{0}-\omega_{1})\mod\mathrm{Im}(d)+0+0+(\omega_{1}-\omega_{2})\mod\mathrm{Im}(d)\\
&=(\omega_{0}-\omega_{2})\mod\mathrm{Im}(d),
\end{split}
\]
which completes the proof of transitivity.

The final claim of this lemma is that the monoid structure on $\widehat{K}_{0}$-generators descends to KCS equivalence classes. We omit the proof because it follows a similar (in fact simpler) argument to the proof of transitivity.
\end{proof}

\begin{nta}\label{NTA.monoid.diff.k.karoubi}
The KCS equivalence class of $(M,D,\om)$, a $\widehat K_0$-generator of $\Om_\bl(A)$, will be denoted by $[(M,D,\om)]$.
We will denote by $\cM(\Om_\bl(A))$ the commutative monoid of $\KCS$-equivalence classes of $\widehat K_0$-generators of $\Om_\bl(A)$ with the monoid operation given by the direct sum. Note that the identity element of this monoid is the equivalence class of the trivial $A$-module $0$ with $0$ connection and the $0$ odd form.
\end{nta}

\begin{dfn}\label{DFN.karoubi.diff.k} Let $A$ be a unital algebra over $\K$, with $\K$ a field containing $\Q$, and let $\Omega_{\bl}(A)$ be a DGA on top of $A$. The \define{noncommutative differential $K_0$ group} of $A$, denoted by $\widehat{K}_0(A)$, is defined by the Grothendieck group of the commutative monoid $\cM(\Om_\bl(A))$.
\end{dfn}

\begin{lem}
Let $\Omega_{\bl}$ be a DGA on top of $\Alg$.
Given an algebra homomorphism $\psi:B\to A$, the assignment
\[
\begin{split}
\cM(\Om_\bl(B))&\xrightarrow{\psi_{*}}\cM(\Om_\bl(A))\\
[(N,D,\omega)]&\mapsto[(N\otimes_{B}A,D_{\psi},\psi(\omega))]
\end{split}
\]
is well-defined and uniquely determines an abelian group homomorphism $\widehat{K}_{0}(B)\xrightarrow{\psi_{*}}\widehat{K}_{0}(A)$ ($D_{\psi}$ is the induced connection from Lemma~\ref{lem:pullbackconnection} and Definition~\ref{dfn:inducedconnection}).
\end{lem}

\begin{proof}
Suppose that $(N,D,\omega)\sim_{\KCS}(N',D',\omega')$, i.e.\ there exists a $B$-module $N''$ with  connection $D''$ and an isomorphism $N\oplus N''\xrightarrow{\varphi}N'\oplus N''$ such that $\KCS(D\oplus D'',\varphi^*(D'\oplus D''))=(\omega-\omega')\mod\mathrm{Im}(d)$.
Let $\mathcal{D}$ be a polynomial path of connections connecting $D\oplus D''$ to $\varphi^*(D'\oplus D'')$.
The extension of scalars functor (cf.\ Remark~\ref{rmk:grothendieckextensionscalars}) $\mathfrak{E}_{\psi}:\Fgp(B)\to\Fgp(A)$ associated to $\psi$ takes the isomorphism $\varphi$ to the isomorphism $\Phi$ given by the composite
\[
(N\otimes_{B}A)\oplus(N''\otimes_{B}A)\cong (N\oplus N'')\otimes_{B}A\xrightarrow{\mathfrak{E}_{\psi}(\varphi)}(N'\oplus N'')\otimes_{B}A\cong (N'\otimes_{B}A)\oplus(N''\otimes_{B}A),
\]
where we have included canonical isomorphisms due to the distributive law of extension of scalars over direct sums. The functor $\mathfrak{E}_{\psi}$ also takes the path of connections $\mathcal{D}$ to $\mathcal{D}_{\psi}$ (cf.\ Definition~\ref{dfn:inducedpathofconnections}), which is a path of connections from $(D\oplus D'')_{\psi}\cong D_{\psi}\oplus D''_{\psi}$ to $(\varphi^*(D'\oplus D''))_{\psi}$, the latter of which is natural isomorphic to $\Phi^{*}(D'_{\psi}\oplus D_{\psi}'')$.
Hence
\[
\begin{split}
\KCS(D_{\psi}\oplus D_{\psi}'',\Phi^{*}(D'_{\psi}\oplus D_{\psi}''))&=\KCS(\mathcal{D}_{\psi})\mod\mathrm{Im}(d)\\
&\overset{\text{Lem~\ref{LEM.KCS.invariant.under.pullback}}}{=\joinrel=\joinrel=\joinrel=\joinrel=}\psi(\KCS(\mathcal{D}))\mod\mathrm{Im}(d)\\
&=(\psi(\omega)-\psi(\omega'))\mod\mathrm{Im}(d),
\end{split}
\]
which proves well-definedness of $\psi_{*}$. The fact that $\psi_{*}$ is a monoid homomorphism follows from the fact that $\mathfrak{E}_{\psi}$ preserves the direct sum (up to canonical isomorphism). Finally, the fact that $\psi_{*}$ induces a unique group homomorphism $\widehat{K}_0(B)\xrightarrow{\psi_{*}}\widehat{K}_0(A)$ follows from the universal property of the Grothendieck group.
\end{proof}

The noncommutative differential $K_0$-group $\widehat{K}_0(A)$ has the forgetful map $I$ and the characteristic form map $R$ that fit into a square diagram as known as the differential cohomology square diagram.

\begin{prp}\label{prop:IandR}
Let $\Omega_{\bl}$ be a DGA on top of $\Alg$. The assignments
\begin{align*}
I: \widehat{K}_0(A)&\ra K_0(A) & R: \widehat{K}_0(A)&\ra (\Om_{\bl}(A)_{\ab})_{\text{even}}^{\text{cl}}\\
[(M,D,\om)] &\mapsto [M] &
[(M,D,\om)] &\mapsto \chk(D)+d\om
\end{align*}
uniquely define group homomorphisms, where $(\Om_{\bl}(A)_{\ab})_{\text{even}}^{\text{cl}}$ is the subgroup of $(\Om_{\bl}(A)_{\ab})_{\even}$ (cf.\ Example \ref{ex:ringofevenforms}) consisting of $d$-closed elements. Furthermore, the homomorphisms above are natural in $A$ in the following sense. Given an algebra homomorphism $\psi:B\to A$, the diagrams
\[
\xy0;/r.25pc/:
(-12.5,7.5)*+{\widehat{K}_{0}(B)}="1";
(12.5,7.5)*+{K_{0}(B)}="2";
(-12.5,-7.5)*+{\widehat{K}_{0}(A)}="3";
(12.5,-7.5)*+{K_{0}(A)}="4";
{\ar"1";"2"^{I}};
{\ar"1";"3"_{\psi_{*}}};
{\ar"3";"4"^{I}};
{\ar"2";"4"^{\psi_{*}}};
\endxy
\qquad\text{ and }\qquad
\xy0;/r.25pc/:
(-15,7.5)*+{\widehat{K}_{0}(B)}="1";
(15,7.5)*+{(\Omega_{\bl}(B)_{\ab})_{\even}^{\text{cl}}}="2";
(-15,-7.5)*+{\widehat{K}_{0}(A)}="3";
(15,-7.5)*+{(\Omega_{\bl}(A)_{\ab})_{\even}^\text{cl}}="4";
{\ar"1";"2"^(0.40){R}};
{\ar"1";"3"_{\psi_{*}}};
{\ar"3";"4"^(0.40){R}};
{\ar"2";"4"^{\psi_{*}}};
\endxy
\]
commute.
\end{prp}

\begin{proof}
The group homomorphism property of $I$ and $R$ are determined by the preservation of the monoidal structure (by the universal property of the Grothendieck group construction). Hence, it suffices to show $I$ and $R$ are homomorphisms on the underlying monoids.
To see that $R$ is well-defined, suppose that $(M_0,D_0,\om_0)\sim_{\KCS}(M_1,D_1,\om_1)$. Then there exist an $A$-module $N$ with connection $D$ and an $A$-module isomorphism $M_{0}\oplus N\xrightarrow{\varphi}M_{1}\oplus N$ such that $\KCS(D_0\dsum D,\vph^*(D_1\dsum D))=(\om_0-\om_1)\!\!\mod\mathrm{Im}(d)$.
Then, by applying $d$ to this last equation, one obtains
\[
\begin{split}
d\omega_{0}-d\omega_{1}&=d\KCS(D_0\dsum D,\vph^*(D_1\dsum D))\\
&\overset{\text{Lem~\ref{LEM.Transgression.formulae.chk}}}{=\joinrel=\joinrel=\joinrel=\joinrel=}\chk(\varphi^*(D_{1}\oplus D))-\chk(D_{1}\oplus D)\\
&\overset{\text{Lem~\ref{LEM.KCS.invariant.under.pullback}}}{=\joinrel=\joinrel=\joinrel=\joinrel=}\chk(D_{1}\oplus D)-\chk(D_{0}\oplus D)\\
&\overset{\text{Lem~\ref{lem:directsumconnections}}}{=\joinrel=\joinrel=\joinrel=\joinrel=}\chk(D_{1})+\chk(D)-\chk(D_{0})-\chk(D)=\chk(D_{1})-\chk(D_{0}),
\end{split}
\]
which proves $R$ is well-defined.
The fact that $R$ is a group homomorphism follows from Lemma~\ref{lem:directsumconnections}.

Finally, naturality of $I$ is immediate, while
naturality of $R$ follows from Lemma~\ref{LEM.KCS.invariant.under.pullback} and the fact that $\Omega_{\bl}$ is a DGA on top of $\Alg$.
\end{proof}

\begin{prp}\label{prop:IRPrCHK}
Let $\Omega_{\bl}(A)$ be a DGA on top of some algebra $A$.
The following diagram commutes: \[\xymatrix{
    \widehat{K}_0(A)\ar[d]^{R} \ar[r]^{I} & K_0(A)\ar[d]^{\chk} \\
(\Om_{\bl}(A)_{\ab})_{\text{even}}^{\text{cl}} \ar[r]^(0.55){\text{Pr}} & H^{\dR}_{\even}(A)
    }\] where $\text{Pr}$ is the map taking de Rham homology class.
\end{prp}

\begin{proof}
This follows immediately from the definitions.
\end{proof}

\begin{rmk}\label{rmk:KaroubiMKA} Historically, the archetype of noncommutative differential $K$-theory has been the multiplicative $K$-group $MK(A)$ by Karoubi \cite[Section 7.3]{Ka2}. The group $MK(A)$ is the \emph{flat subgroup} of our $\widehat{K}_0(A)$ when $F^r=0$ for $r>0$, where $F^r$ is the filtration considered in \cite{Ka2}; i.e. $$MK(A)=\ker\left(\widehat{K}_0(A)\srl{R}\ra \Om_{\bl}(A)_{\ab}\right).$$
Therefore, our noncommutative differential $K$-theory generalizes Karoubi's multiplicative $K$-theory and promotes it to the context of differential extensions of $K$-theory in the noncommutative framework.
\end{rmk}

\subsection{Cycle maps} \label{SEC.cycle.maps.in.karoubi.diff.k}
We show that the noncommutative differential $K$-group $\widehat{K}_0(A)$
recovers what is known as the even differential $K$-theory of a manifold (Theorem \ref{THM.cycle.map.1} below). Furthermore, there is a universal cycle map from the group $\widehat{K}^{\text{u}}_0(A)$, constructed by using the DGA of universal noncommutative differential forms, into the group $\widehat{K}_0(A)$, defined with respect to an arbitrary DGA $\Om_\bl(A)$ on top of $A$. For example, one such DGA could be the de~Rham complex if $A$ is smooth functions over a manifold (Definition \ref{DEF.cycle.map.2}).

In this subsection, we will use notations and constructions of the geometric model of differential $K$-theory discussed in Section \ref{SEC.FLK}. In particular, recall the notation $\f{M}(X)$ as in Definition \ref{DFN.FLK} as well as $\cM(\Om_\bl(A))$ in Notation \ref{NTA.monoid.diff.k.karoubi}.

\begin{lem}
\label{lem:CStoKCS}
Let $X$ be a compact smooth manifold, set $A:=\Cinf(X;\C)$, and let $\Om_\bl(A)=\Om^\bl_{\dR}(X;\C)$ be the de~Rham complex of $X$. The assignment
\beq\label{eqn:CStoKCS} \f{M}(X) &\ra \cM(\Om_\bl(A))\\
[\left(E,\na,(\om_{2k-1})_{k\in\Z^+}\right)]_{\CS}&\mapsto [\left(\Ga(E),\na,((2\pi i)^k\om_{2k-1})_{k\in\Z^+}\right)]_{\KCS} \eeq is an isomorphism of commutative monoids that is natural in $X$.
\end{lem}
\begin{proof} To streamline notations in our proof, we will omit subindices in the sequence of odd degree differential forms modulo exact forms. For example, $\om$ denotes a sequence $(\om_{2k-1})_{k\in\Z^+}$, $C\cdot \om$ means $((2\pi i)^k\om_{2k-1})_{k\in\Z^+}$, and $C^{\text{inv}}\cdot \om$ represents $((2\pi i)^{-k}\om_{2k-1})_{k\in\Z^+}$. Here, the use of the letter $C$ is to express that a constant is multiplied (for each $k$). Now, suppose $(E_0,\na_0,\om_0)$ and $(E_1,\na_1,\om_1)$ are equivalent $\widehat{K}^0$-generators. We claim $(\Ga(E_0),\na_0,C\cdot\om_0)$ and $(\Ga(E_1),\na_1,C\cdot\om_1)$ are equivalent $\widehat{K}_0$-generators of $\Om_\bl(A)$. Since there exists a smooth vector bundle $G\to X$ with connection $\na_G$ and an isomorphism $E_0\dsum G \srl\phi\ra E_1\dsum G$, we obtain $(\Ga(G),\na_G)\in \FgpD(X)$ (cf. Example \ref{EXA.examples.of.connection}) and an isomorphism of $A$-modules $\vph:\Ga(E_0)\dsum \Ga(G)\ra \Ga(E_1)\dsum \Ga(G)$. Since $\vph^*(\na_1\dsum\na_G)=\phi^*(\na_1\dsum\na_G)$, Remark \ref{RMK.KCS.is.a.generalization.of.CS} implies $(2\pi i)^{k}\CS_k(\na_0\dsum \na_G,\na_1\dsum \na_G)=\KCS_k(\na_0\dsum \na_G,\na_1\dsum \na_G)$. Hence, the assignment~(\ref{eqn:CStoKCS}) is well-defined.

We next show that~(\ref{eqn:CStoKCS}) is surjective. Take any $[(M,D,\om)]\in \cM(\Om_\bl(A))$. Since $\Gamma^\con$ from Proposition \ref{PRP.Differential.SS} is an equivalence of categories,
there is a $(E,\na)\in \Bun^{\text{iso}}_\na(X)$ such that $\al:\Ga(E)\srl\isom\ra M$ and $\na=(\al\tsr\text{id})\circ D\circ\al^{-1}$. Then~(\ref{eqn:CStoKCS}) maps $[(E,\na,C^{\text{inv}}\cdot\om)]\in \f{M}(X)$ to $[(\Ga(E),\na,\om)]\in \cM(\Om_\bl(A))$, which is equal to $[(M,D,\om)]$ because ${\al^{-1}}^*\na=D$, and $\KCS(D,{\al^{-1}}^*\na)=0\mod\text{Im}(d)$. Thus the map is onto.

We next show that~(\ref{eqn:CStoKCS}) is injective. Suppose $(M_0,D_0,C\cdot\om_0)$ and $(M_1,D_1,C\cdot\om_1)$ are equivalent $\widehat{K}_0$-generators of $\Om_\bl(A)$. Then there exists a $(M_2,D_2)\in\FgpD(X)$ and an $A$-module isomorphism $\vph:M_0\dsum M_2\ra M_1\dsum M_2$ satisfying $\KCS(D_0\dsum D_2,D_1\dsum D_2)=\om_0-\om_1 \mod\text{Im}(d)$. Again, since $\Ga^{\text{con}}$ is an equivalence by Proposition \ref{PRP.Differential.SS}, for each $(M_i,D_i)\in\FgpD(X)$ there exist $(E_i,\na_i)\in \BunD(X)$ and $\al_i:\Ga(E_i)\srl\isom\ra M_i$ satisfying $\na_i=(\al_i\tsr\text{id})\circ D_i\circ\al_i^{-1}$ for $i=0,1,2$. The isomorphism $\vph$ induces a vector bundle isomorphism $\phi:E_0\dsum E_2\ra E_1\dsum E_2$ satisfying $(\al_1\dsum\al_2)\circ\phi=\vph\circ (\al_0\dsum\al_2)$. We note that $\phi^*(\na_1\dsum\na_2)=\phi^*(\al_1\dsum\al_2)^*(D_1\dsum D_2)=(\al_0\dsum\al_2)^*\vph^*(D_1\dsum D_2)$. Hence by Lemma \ref{LEM.KCS.invariant.under.pullback}, $\KCS(D_0\dsum D_2,\vph^*(D_1\dsum D_2))=\KCS(\na_0\dsum \na_2,\phi^*(\na_1\dsum \na_2))$, which is equal to $C\cdot\CS(\na_0\dsum \na_2,\phi^*(\na_1\dsum \na_2))$ by Remark \ref{RMK.KCS.is.a.generalization.of.CS}. Hence, (\ref{eqn:CStoKCS}) is one-to-one.

Finally, verifying that the given map is a monoid homomorphism that is natural in $X$ follows similar calculations to those we have already seen.
\end{proof}

Since a natural isomorphism of commutative monoids induces a natural isomorphism of corresponding Grothendieck groups, we obtain the following theorem.

\begin{thm} \label{THM.cycle.map.1} Let $A=\Cinf(X;\C)$. The equivalence $\BunD^{\text{iso}}(X)\xrightarrow{\Gamma^\con} \FgpD^{\text{iso}}(X)$ from~\eqref{EQN.GammaCon} and Proposition~\ref{PRP.Differential.SS} induces a cycle map
$\texttt{cycle}:\widehat{K}^0(X)\ra \widehat{K}_0(A)$, which is an isomorphism of abelian groups natural in $X$.
\end{thm}

The noncommutative differential $K$-group $\widehat{K}_0(A)$ is therefore far more general than the differential $K$-theory of a compact smooth manifold. The construction is valid for $\FgpD(A)$ of finitely generated projective $A$-modules with connections associated to a unital $\K$-algebra $A$ and any choice of DGA $\Om_\bl(A)$ on top of $A$.

As we have seen in Remark~\ref{RMK.nc.diff.forms}, the universal property of the complex $\univf_{\bl}(A)$ of universal noncommutative differential forms induces a map on noncommutative de~Rham homology groups and a compatibility between Karoubi's Chern characters. We would like to extend those observations to noncommutative differential $K$-groups by constructing a cycle map by using the universality of noncommutative differential forms.

Let $A$ be a unital algebra over $\C$ and $M\in \Fgp(A)$. Let $\univf_{\bl}(A)$ be the complex of universal noncommutative differential forms and $\Om_{\bl}(A)$ any DGA on top of $A$. Using these DGAs we can consider two different connections on $M$ and hence two different categories of finitely generated projective $A$-modules with connection $\FgpuD(A)$ and $\FgpD(A)$, respectively.

Recall the DGA map $\Phi:\univf_{\bl}(A) \ra \Om_\bl(A)$, $a_0\dun a_1\cdots \dun a_n\mapsto a_0 da_1\cdots da_n$ from Remark~\ref{RMK.nc.diff.forms} (\ref{item:universaltoordinarydR}). The map $\Phi$ induces a functor
\beqs
\Phi:\FgpuD(A)&\ra \FgpD(A)\\
(M, D)&\mapsto (M,\Phi D)\\
((M_0, D_0)\ra (M_1, D_1))&\mapsto ((M_0,\Phi D_0)\ra (M_1,\Phi D_1)).
\eeqs

\begin{lem}\label{LEM.universality.functoriality} Let $A$ be a unital algebra over $\C$ and $\Om_\bl(A)$ a DGA on top of $A$.
The assignment
\beqs \Phi:\cM(\univf_\bl(A)) &\ra \cM(\Om_\bl(A))\\
[(M, D,\om)]_{\KCS}&\mapsto [(M,\Phi D,\Phi\om)]_{\KCS}
\eeqs
is a homomorphism of commutative monoids uniquely induced by the map $\Phi:\univf_{\bl}(A) \ra \Om_\bl(A)$.
\end{lem}
\begin{proof} Recall notations from Section \ref{SEC.KCS}. The DGA map $\Phi:\univf_{\bl}(A) \ra \Om_\bl(A)$ induces the DGA map $\Phi_{\La_\bl}:(\La_\bl\dtsr\univf_{\bl}(A),\de^{\mathrm{u}}) \ra (\La_\bl\dtsr\Om_\bl(A),\de)$ sending $\om\tsr\te$ to $\om\tsr\Phi\te$, and the homotopy operator $K$ commutes with $\Phi$. That is $K\circ\Phi_{\La_\bl}=\Phi \circ K$.
Since $\Phi\chk(D)=\chk(\Phi D)$, we conclude $\Phi\KCS(\mathcal{D})=\KCS(\Phi\mathcal{D})$, where $\mathcal{D}$ is a polynomial path of connections with respect to the DGA $\univf_{\bl}(A)$ and where $\Phi\mathcal{D}$ is the associated polynomial path of connections with respect to $\Om_\bl(A)$ induced by $\Phi$.

Now suppose $(M_0, D_0, \om_0)$ and $(M_1, D_1, \om_1)$ are equivalent $\widehat{K}_0$-generators of $\univf_\bl(A)$ as realized by some $(N,D)\in \Fgp_{\D^{\mathrm{u}}}(A)$ and an $A$-module isomorphism $M_0\dsum N \srl\vph\ra M_1\dsum N$. Then $\Phi\KCS( D_0\dsum  D,  D_1\dsum  D)=\KCS(\Phi D_0\dsum \Phi D, \Phi D_1\dsum \Phi D)$, and the triples $(M_0,\Phi D_0,\Phi \om_0)$ and $(M_1,\Phi D_1,\Phi \om_1)$ are equivalent $\widehat{K}_0$-generators of $\Om_\bl(A)$ with the choice $(N,\Phi  D)$ and the same $A$-module isomorphism $\vph$. Thus, the assignment in the statement of this lemma is well-defined. It is readily seen that the map is a homomorphism of commutative monoids.
\end{proof}

\begin{rmk} By a similar proof, any morphism $\Xi_{\bl}(A)\to\Omega_{\bl}(A)$ of DGA's on top of $A$ induces a homomorphism $\cM(\Xi_\bl(A))\to\cM(\Om_\bl(A))$ of commutative monoids.
\end{rmk}

Since a monoid homomorphism induces a homomorphism on their Grothendieck groups, we have the following cycle map. We will use the notation $\widehat{K}^{\mathrm{u}}_0(A)$ to denote the group completion of $\cM(\univf_\bl(A))$.

\begin{dfn} \label{DEF.cycle.map.2}  Let $A$ be a unital algebra over $\C$ and $\Om_\bl(A)$ any DGA on top of $A$. The \define{universal cycle map} is the map $\texttt{cycle}_\Phi:\widehat{K}^{\mathrm{u}}_0(A)\ra \widehat{K}_0(A)$ induced by the canonical DGA map $\Phi:\univf_{\bl}(A) \ra \Om_\bl(A)$.
\end{dfn}

\subsection{Hexagon diagram} \label{SEC.hexagon.diagram}

\begin{nta} Let $A$ be an algebra over $\K$. We denote by $\cP$ the category whose objects are pairs $(M,\phi)$ consisting of $M\in \Fgp(A)$ and $\phi\in \text{Aut}(M)$. A morphism from $(M,\phi)$ to $(M',\phi')$ is an isomorphism $\vph:M\ra M'$ satisfying that $\phi'\circ \vph=\vph\circ \phi$.
\end{nta}

\begin{dfn}\label{DFN.algebraic.K1} Let $A$ be an algebra over $\K$. The \define{algebraic $K_1$-group} of $A$ is the free abelian group generated by the set of isomorphism classes of $\cP$ modulo the following relations:
\begin{itemize}
    \item[(1)] $(M_1\dsum M_2,\phi_1\dsum\phi_2)=(M_1,\phi_1)+(M_2,\phi_2)$.
    \item[(2)] $(M,\phi_1\circ \phi_2)=(M,\phi_1)+(M,\phi_2)$.
\end{itemize}
\end{dfn}

\begin{dfn}
Let $A$ be an algebra over $\K$ and fix a DGA $\Omega_{\bl}(A)$ on top of $A$. Consider a triple $(M,D,\phi)$ consisting of $(M,D)\in \FgpD(A)$ and $\phi\in\text{Aut}(M)$. The straight line path joining $D$ and the pullback connection $\phi^*D$ is a polynomial path of connections (see Definition \ref{dfn:inducedconnection} and Remark \ref{rmk:spaceofconnectionsisaffine}).
The \define{total odd Chern character form} of a triple $(M,D,\phi)$ is
$\chk_{1}(M,D,\phi):= \KCS(t\mapsto (1-t)D+t\phi^*D)$.
\end{dfn}

\begin{dfn} Let $A$ be an algebra over $\K$ and fix a DGA $\Omega_{\bl}(A)$ on top of $A$. The \define{total odd Chern character} is the assignment \beqs
\chk_1:K_1(A)&\ra H^{\dR}_{\odd}(A)\\
[(M,\phi)]&\mapsto [\KCS(t\mapsto (1-t)D+t\phi^*D)].
\eeqs where $D$ is an arbitrary connection on $M$.
\end{dfn}

\begin{prp} Let $\Omega_{\bl}$ be a DGA on top of $\Alg$. The assignment $\chk_1:K_1(A)\ra H^{\dR}_{\odd}(A)$ is a group homomorphism that does not depend on the choice of connection and is natural in $A$.
\end{prp}
\begin{proof}
Let $(M,\phi)$ and $(M',\phi')$ be equivalent pairs under an isomorphism $\vph:M\ra M'$ so that $\phi'\circ\varphi=\varphi\circ\phi$.
Choose connections $D$ on $M$ and $D'$ on $M'$. Then
\[
\begin{split}
\KCS(t\mapsto(1-t)D+t\phi^*D)&=\KCS(t\mapsto(1-t)D+t\vph^*D')+\KCS(t\mapsto(1-t)\vph^*D'+t\vph^*\phi'^*D')\\
&\quad+\KCS(t\mapsto(1-t)\vph^*\phi'^*D'+t\phi^*D) \mod \text{Im}(d)
\end{split}
\]
by Lemma~\ref{lem:KCStriangle} applied to paths schematically depicted as
\[
\xy0;/r.25pc/:
(-10,7.5)*+{D}="1";
(10,7.5)*+{\phi^*D}="2";
(-10,-7.5)*+{\varphi^*D'}="3";
(10,-7.5)*+{\varphi^*\phi'^*D'}="4";
{\ar@/^/"1";"2"};
{\ar@/_/"1";"3"};
{\ar@/_/"3";"4"};
{\ar@/_/"4";"2"};
\endxy
\]
The third term in these expressions can be simplified to (all terms modulo the image of $d$)
\[
\begin{split}
\KCS(t\mapsto(1-t)\vph^*\phi'^*D'+t\phi^*D)&=\KCS(t\mapsto(1-t)\phi^*\vph^*D'+t\phi^*D) \text{ since $\phi'\circ\varphi=\varphi\circ\phi$}\\
&=\KCS(t\mapsto(1-t)\vph^*D'+tD)\quad\text{by Lemma~\ref{LEM.KCS.invariant.under.pullback} (\ref{item:pullbackconnection})}\\
&=-\KCS(t\mapsto(1-t)D+t\vph^*D')\quad\text{by Lemma~\ref{LEM.Transgression.formulae.chk} (\ref{item:KCSreversepath})}.
\end{split}
\]
Combining this with the previous equation gives
\[
\KCS(t\mapsto(1-t)D+t\phi^*D)=\KCS(t\mapsto(1-t)\vph^*D'+t\vph^*\phi'^*D')=\KCS(t\mapsto(1-t)D'+t\phi'^*D')
\]
by another application of Lemma~\ref{LEM.KCS.invariant.under.pullback} (\ref{item:pullbackconnection}). This proves that $\chk_1$ is well-defined on the set $\texttt{isom}(\cP)$ of isomorphism classes of $\mathcal{P}$. By the universal property of the free group functor, there exists a canonical extension to the free group of $\texttt{isom}(\cP)$. To prove that it is well-defined on $K_{1}(A)$, it therefore suffices to show that the normal subgroup generated by the relations in Definition~\ref{DFN.algebraic.K1} get sent to zero under $\chk_{1}$. Note that $\chk_1([(M_1\dsum M_2,\phi_1\dsum\phi_2)])=\chk_1([(M_1,\phi_1)])+\chk_1([(M_2,\phi_2)])$ by the additivity of trace on block sum (cf.\ Lemma~\ref{lem:KCSdirectsumpaths}), and $\chk_1(M,\phi_1\circ \phi_2)=\chk_1(M,\phi_1)+\chk_1(M,\phi_2)$ follows from Lemma~\ref{LEM.KCS.invariant.under.pullback} (\ref{item:pullbackconnection}). This shows that $\chk_1$ descends to a unique group homomorphism on the quotient $K_{1}(A)$.

To see the naturality, let $B$ be an algebra over $\K$ and $A\srl\psi\ra B$ an algebra map. We see that $\Om_\bl(\psi)(\chk_1(M,\phi))=[\KCS(t\mapsto (1-t)D_\psi + t (\phi^*D)_\psi)]$ by Lemma~\ref{LEM.KCS.invariant.under.pullback} (\ref{item:inducedpathconnection}).
\end{proof}

\begin{nta} Let $\text{Im}(\chk_1)$ be the subgroup of the abelian group $\Omega_{\odd}(A)$
generated by all total odd Chern character forms. We shall write $\Om_{\chk_1}(A)$ for the abelian group generated by $\text{Im}(\chk_1)+\Omega_{\odd}^{\text{exact}}(A)$.
\end{nta}

\begin{dfn}\label{defn:amap} Let $A$ be an algebra over $\K$ and fix a DGA $\Omega_{\bl}(A)$ on top of $A$. Define \beqs
a:\Om_{\odd}(A)/\Om_{\chk_1}(A) &\ra \widehat{K}_0(A)\\
[\om]&\mapsto [(O,0,\om)]-[(O,0,0)],
\eeqs where $O$ denotes the zero module $\{0\}$ and $0$ the trivial connection on $O$.
\end{dfn}

\begin{lem} Let $\Omega_{\bl}$ be a DGA on top of $\Alg$. Then the map $a:\Om_{\odd}(A)/\Om_{\chk_1}(A) \ra \widehat{K}_0(A)$ is a group homomorphism natural in $A$.
\end{lem}
\begin{proof}
Suppose $\om-\eta=\chk_1(M,D,\phi)$ for some pair $(M,\phi)\in \cP$ and an arbitrary choice of connection $D$ on $M$.
\beqs
a([\om])&=[(O,0,\om)]-[(O,0,0)]\\
&=[(O,0,\eta+\chk_1(M,D,\phi))]+[(M,D,0)]-[(M,D,0)]-[(O,0,0)]\\
&=[(O,0,\eta)]+[(M,D,\chk_1(M,D,\phi))]-[(M,D,0)]-[(O,0,0)]\\
&=[(O,0,\eta)]-[(O,0,0)]=a([\eta]).\eeqs So the map is well-defined. It is clearly a group homomorphism, and by Lemma~\ref{LEM.KCS.invariant.under.pullback} (\ref{item:inducedpathconnection}), the map is natural in $A$.
\end{proof}

\begin{prp} The sequence
$$0\ra \Omega_{\odd}(A)/\Omega_{\chk_1}(A)\srl{a}\ra \widehat K_0(A)\srl{I}\ra K_0(A)\ra 0$$
is exact. Here, $a$ and $I$ are from Definition~\ref{defn:amap} and Proposition~\ref{prop:IandR}, respectively.
\end{prp}
\begin{proof}
We first prove the injectivity of the map $a$. Let $\omega\in\Omega_{\odd}(A)$ and suppose $a([\om])=[(O,0,\om)]-[(O,0,0)]=0$. Then there exists $(M,D)\in \FgpD(A)$ and an $A$-module automorphism $\vph:M\ra M$ such that $\om=\KCS(t\mapsto (1-t)D+t\phi^*D)=\chk_1(M,D,\vph).$ Hence, $\omega=0\mod\Om_{\chk_1}(A)$, proving injectivity of $a$. Now we verify $\ker I\subseteq \text{Im}(a)$. Suppose $[(M_0,D_0,\om_0)]-[(M_1,D_1,\om_1)]\in \widehat{K}_0(A)$ satisfies $[M_0]=[M_1]$ in $K_0(A)$. Then there exists an $A$-module $M$ and an $A$-module isomorphism $\phi:M_0\dsum M\ra M_1\dsum M$. Take any connection $D$ on $M$ and note that $[(M_1\dsum M,D_1\dsum D,\om_1)]=[(M_0\dsum M,D_0\dsum D,\te)]$, where $\te=\om_1+\KCS(t\mapsto (1-t)(D_0\dsum D)+t\phi^*(D_1\dsum D)$. Hence, \beqs \phantom{}
[(M_0,D_0,\om_0)]-[(M_1,D_1,\om_1)] &=[(M_0\dsum M,D_0\dsum D,\om_0)]-[(M_1\dsum M,D_1\dsum D,\om_1)]\\
&=[(M_0\dsum M,D_0\dsum D,\om_0)]-[(M_0\dsum M,D_0\dsum D,\te)]\\
&=[(O,0,\om_0)]-[(O,0,\te)]=a([\om_0-\te]).
\eeqs
It is clear that $I\circ a=0$. Finally, $I$ is onto by the fact that every finitely generated projective $A$-module admits a connection.
\end{proof}

\begin{lem}\label{lem:alphabetar} The assignments
\begin{align*}
H_{\odd}^{\dR}(A)&\xrightarrow{\al}MK(A)&
MK(A)&\xrightarrow{\be} K_{0}(A)\\
[\om]&\mapsto [(O,0,\om)]-[(O,0,0)]\qquad,\quad&
[(M_0,D_0,\om)]-[(M_1,D_1,\eta)]&\mapsto [M_0]-[M_1],
\end{align*}
\beqs
\text{ and}\qquad
H_{\odd}^{\dR}(A) &\xrightarrow{r}\Om_{\odd}(A)/\Om_{\chk_1}(A)\\
[\om]&\mapsto \om+\Om_{\chk_1}(A)
\eeqs
are group homomorphisms. Here, $O$ denotes the zero module $\{0\}$, $0$ denotes the trivial connection on $O$, and $MK(A):=\ker(R)$ (cf.\ Remark~\ref{rmk:KaroubiMKA}).
\end{lem}

\begin{prp}
\label{prop:hexagonpieces}
The following facts hold regarding the maps $I$ and $R$ from Proposition~\ref{prop:IandR}, $\text{Pr}$ from Proposition~\ref{prop:IRPrCHK}, $a$ from Definition~\ref{defn:amap}, and $\alpha, \beta$, and $r$ from Lemma~\ref{lem:alphabetar}.
\begin{enumerate}
\item $R \circ a=d$.
\item $a \circ r=\trm{incl}\circ \al$, where $\trm{incl}:MK(A)\hookrightarrow \widehat{K}_{0}(A)$ denotes the inclusion.
\item\label{item:betaIincl} $\be =I\circ \trm{incl}$.
\item\label{item:hexagonexact} The following sequences are exact:
\beqs
H_{\odd}^{\dR}(A) \xrightarrow{\al} & MK(A) \xrightarrow{\be} K_{0}(A) \xrightarrow{\chk} H_{\even}^{\dR}(A)\\
H_{\odd}^{\dR}(A) \xrightarrow{r}&\Om_{\odd}(A)/\Om_{\chk_1}(A) \xrightarrow{d} \text{Im}(R) \xrightarrow{\trm{Pr}} H_{\even}^{\dR}(A)
\eeqs
\end{enumerate}
\end{prp}
\begin{proof}
We shall give proofs for $\ker(\be)\subseteq \text{Im}(\al)$ and $\ker(\chk)\subseteq \text{Im}(\be)$ in (\ref{item:hexagonexact}). All other claims are straightforward. To prove the former, suppose $[(M_0,D_0,\om_0)]-[(M_1,D_1,\om_1)]\in \ker(\be)$. This means that there exists $M\in \Fgp(A)$ and an $A$-module isomorphism $\vph:M_0\dsum M\ra M_1\dsum M$. Take an arbitrary connection $D$ on $M$ and let $\zeta:=\KCS\left(t\mapsto (1-t)(D_0\dsum D)+t\vph^*(D_1\dsum D)\right)$. Then \beqs \phantom{}
[(M_0,D_0,\om_0)]-[(M_1,D_1,\om_1)]&=[(M_0,D_0,\om_0)]-([(M_1,D_1,\om_0-\zeta)]+[(O,0,-\om_0+\zeta+\om_1)])\\
&=[(M_0,D_0,\om_0)]-([(M_0,D_0,\om_0)]+[(O,0,-\om_0+\zeta+\om_1)])\\
&=a([\om_0-\zeta-\om_1]).
\eeqs
We have to verify that $\om_0-\zeta-\om_1$ represents an odd degree de Rham homology class. This follows from  $$d(\om_0-\zeta-\om_1)=-\chk(D_0)-\big(\chk(D_1\dsum D)-\chk(D_0\dsum D)\big)+\chk(D_1)=0.$$

To prove $\ker(\chk)\subseteq \text{Im}(\be)$, suppose $\chk\left([M_0]-[M_1]\right)=0$. If we choose connections $D_0$ and $D_1$ on $M_0$ and $M_1$, respectively, then $\chk(D_0)-\chk(D_1)=d\xi$ for some odd degree form $\xi$.
Hence, $[(M_0,D_0,0)]-[(M_1,D_1,\xi)]\in\ker(R)$. Therefore,
\[
\begin{split}
[M_0]-[M_1]&=I\left([(M_0,D_0,0)]-[(M_1,D_1,\xi)]\right)\quad\text{ by Proposition~\ref{prop:IandR}}\\
&=\be\left([(M_0,D_0,0)]-[(M_1,D_1,\xi)]\right)\quad\text{by Proposition~\ref{prop:hexagonpieces} (\ref{item:betaIincl})}.
\end{split}
\]
Hence, $[M_0]-[M_1]\in\text{Im}(\be)$, as needed.
\end{proof}

\begin{cor}\label{COR.hexagon.diagram} In the following diagram for $\widehat K_0(A)$, all square and triangles are commutative and all sequences are exact.
\[
\xy0;/r.25pc/:
(-35,30)*+{0}="0TL";
(-35,0)*+{H_{\odd}^{\dR}(A)}="L";
(-17.5,15)*+{MK(A)}="TL";
(-35,-30)*+{0}="0BL";
(-17.5,-15)*+{\Omega_{\odd}(A)/\Omega_{\chk_1}(A)}="BL";
(0,0)*+{\widehat{K}_{0}(A)}="M";
(17.5,15)*+{K_{0}(A)}="TR";
(35,30)*+{0}="0TR";
(17.5,-15)*+{\Omega_{\chk_0}(A)}="BR";
(35,0)*+{H_{\even}^{\dR}(A)}="R";
(35,-30)*+{0}="0BR";
(-17.5,0)*{\scriptstyle\circlearrowleft};
(0,9)*{\scriptstyle\circlearrowleft};
(0,-8)*{\scriptstyle\circlearrowright};
(17.5,0)*{\scriptstyle\circlearrowleft};
{\ar"0TL";"TL"};
{\ar"L";"TL"^{\alpha}};
{\ar"TL";"TR"^{\beta}};
{\ar@{}"TL";"M"^(0.2){}="a"^(0.8){}="b"};
{\ar@{^{(}->}"a";"b"};
{\ar"M";"TR"^{I}};
{\ar"TR";"0TR"};
{\ar"TR";"R"^{\chk}};
{\ar"L";"BL"^{r}};
{\ar"0BL";"BL"};
{\ar"BL";"M"^{a}};
{\ar"BL";"BR"^(0.60){d}};
{\ar@{->>}"BR";"R"};
{\ar"BR";"0BR"};
{\ar"M";"BR"^{R}};
{\ar"BR";"R"^{\text{Pr}}};
\endxy
\]
Here, $\Omega_{\chk_0}(A):=\mathrm{Im}(R)$ is the subgroup of $d$-closed even forms whose de~Rham homology class is in the image of Karoubi's Chern character.
\end{cor}

\appendix

\section{The Serre--Swan theorem and its differential refinements}\label{SEC.DSS}

In this appendix, we shall review the Serre--Swan theorem and its variants. Secondly, we extend those results to include the data of connections in their respective categories. As an application, we provide an algebraic formulation of the Narasimhan--Ramanan theorem.

\subsection{Review of Serre--Swan theorem}

\begin{dfn}
Let $X$ be a topological space and $\cA$ a sheaf of rings on $X$. A \define{sheaf of $\cA$-modules} $\cM$ on $X$ is a sheaf of abelian groups on $X$
such that for each open $U \subseteq X,$
\begin{enumerate}
\item
$\cM(U)$ is an $\cA(U)$-module and
\item
$(a\cdot m)\vert_V = a\vert_V \cdot m\vert_V$ for all $a
\in \cA(U)$, $m \in \cM(U)$, and $V \subseteq U$.
\end{enumerate}
A \define{morphism of $\cA$-modules} $f\colon \cM \to
\cN$ is a morphism of sheaves such that $f_U\colon \cM(U)
\to \cN(U)$ is $\cA(U)$-linear for each open $U \subseteq X$.
\end{dfn}

\begin{dfn}
An $\cA$-module $\cM$ is \define{free} if $\cM \cong \bigoplus_{i \in I} \cA$ for some set $I$, in which case $\lvert I\rvert$ is the \define{rank} of $\cM$. An $\cA$-module $\cM$ is \define{locally free} if there is an open cover $\{ U_j \subset X\}_{j \in J}$ such that $\cM\vert_{U_j}$ is free for every $j \in J$.  A locally free $\cA$-module $\cM$ has \define{bounded rank} if the open cover can be chosen so that $\operatorname{rk}\cM\vert_{U_j} \leq n$, for some fixed integer $n$ and all $j \in J$.
\end{dfn}

In what follows, $X$ will be a smooth manifold, and we take $\cA =
C^\infty_X$ to be its sheaf of smooth functions.

\begin{nta} Let $X\in \Man$.  We denote by $\Lfb(X)$ the
  full subcategory of the category of sheaves of $C^\infty_X$-modules consisting of locally free sheaves of $C^\infty_X$-modules of
  bounded rank. Recall that $\Bun^{\text{iso}}(X)$ denotes the category of finite rank complex vector bundles over $X$ with bundle maps that are fiberwise isomorphisms and $\Fgp^{\text{iso}}(\Cinf(X))$ denotes the category of finitely generated left $\Cinf(X)$-modules with $\Cinf(X)$-linear isomorphisms (see Notation~\ref{NTA.karoubi.notation.dga}). In the remainder of this appendix only (and not the body of the paper), $\Omega_{X}^{k}$ denotes the sheaf of smooth complex-valued de~Rham $k$-forms over $X$.
\end{nta}

For each manifold $X$, the diagram
\begin{equation}
\label{eq:bun-lfb-fpg}
\vcenter{\hbox{
\xymatrix{
\Bun^{\text{iso}}(X) \ar[d]_{\Ga_{\text{loc}}} \ar[r]^-\Gamma & \Fgp^{\text{iso}}(\Cinf(X))\\
\Lfb(X) \ar[ur]_{\Ga_{X}} &
}
}}
\end{equation}
commutes. Here, the vertical map assigns to a vector bundle its sheaf of local sections. It is locally free and of the same rank as the vector bundle. The diagonal map $\Ga_X$ assigns to a sheaf of $C^\infty_X$-modules its $C^\infty(X)$-module of global sections. The horizontal map $\Gamma$ is also a global sections functor.

The following result is a smooth version of the Serre--Swan theorem.

\begin{prp}
\label{thm:serre-swan}
Each of the arrows in \eqref{eq:bun-lfb-fpg} is an equivalence of categories.
\end{prp}
\begin{proof}
See Nestruev \cite[Theorem 11.32]{Nes}.
\end{proof}

This proposition belongs to a family of several related results.
Serre's original theorem was formulated for an affine scheme $X$ \cite{Se}. It asserted an
equivalence between the categories of locally free sheaves of
$\mathcal O_X$-modules and finitely generated projective modules over
its coordinate ring $\Gamma(X, \mathcal O_X)$.  Swan extended Serre's result
replacing $X$ by a paracompact topological space of bounded topological
dimension and $\mathcal O_X$ by its sheaf of continuous functions \cite{Sw}.
Several other variations exist.  See Morye~\cite{Mor} for a general
theorem from which many classical results follow.

\subsection{Differential refinements of the smooth Serre--Swan theorem}

\begin{dfn}\label{def:connection-on-a-sheaf}
Let $X$ be a smooth manifold. A \define{connection} on a sheaf of $C^\infty_X$-modules $\mathcal M$ is a $\C$-linear map%
\footnote{The notation $\mathcal{D}$ is used in this appendix only to refer to a connection on a sheaf. This notation is not to be conflated with that of a polynomial path of connections, which only appears in the body of this work.}
$\cD\colon \mathcal{M}\ra \mathcal{M}\tsr_{C^\infty_X} \Om^1_X$ satisfying $\cD(vf)=\cD(v)f+v\tsr df$ for every $U\subset X$, $v\in \mathcal{M}(U)$, and $f\in C^\infty(U)$.  If $\mathcal M$ and $\mathcal N$ are $C^\infty_X$-modules with connection, denoted by $\mathcal D_{\cM}$ and $\mathcal D_{\cN}$, respectively, a morphism $F\colon \mathcal M \to \mathcal N$ is \define{compatible with the connections} if the diagram
\[
\xy0;/r.25pc/:
(-20,7.5)*+{\cM}="M";
(20,7.5)*+{\cN}="N";
(-20,-7.5)*+{\cM\otimes_{C^{\infty}_{X}}\Omega^{1}_{X}}="MW";
(20,-7.5)*+{\cN\otimes_{C^{\infty}_{X}}\Omega^{1}_{X}}="NW";
{\ar"M";"N"^{F}};
{\ar"M";"MW"_{\mathcal{D}_{\cM}}};
{\ar"N";"NW"^{\mathcal{D}_{\cN}}};
{\ar"MW";"NW"_{F\otimes_{C^{\infty}_{X}}\text{id}_{\Omega^{1}_{X}}}};
\endxy
\]
commutes.  In a formula, this means that
\begin{equation*}
\mathcal D_\cN F(v) = (F\tsr_{C^{\infty}_{X}} \text{id}_{\Om_X^1})(\mathcal D_{\cM} v)
\end{equation*}
holds for all local sections $v$ of $\mathcal M$.
\end{dfn}

\begin{nta}
Let $\LfbD(X)$ denote the category of locally free sheaves of $C^\infty_X$-modules of bounded rank endowed with a connection. Morphisms are $C^\infty_X$-linear homomorphisms compatible with the given connections. Recall, $\FgpD^{\text{iso}}(X)$ denotes the category of finitely-generated projective modules with connection over $C^\infty(X;\C)\equiv C^\infty(X)$ with respect to the DGA $\Om_{\dR}^\bl(X;\C)$ (cf.\ Notation~\ref{NTA.karoubi.notation.dga},  Definition~\ref{DFN.Fgp.conn.}, and Example~\ref{EXA.examples.of.connection}).
\end{nta}

The tensor product of vector bundles and tensor product of modules over $C^\infty(X)$ make $\Bun^{\text{iso}}(X)$ and $\Fgp^{\text{iso}}(C^\infty(X))$ into symmetric monoidal categories. Recall that if $\mathcal M$ and $\mathcal N$ are sheaves of $\mathcal A$-modules, we define $\mathcal M \otimes_{\mathcal A} \mathcal N$ as the sheafification of the presheaf $U \mapsto \mathcal M(U) \otimes_{\mathcal A(U)} \mathcal N(U)$.
In general, the sheafification step in the definition of tensor product of sheaves is essential, and the global sections functor is not monoidal. In the context of $C^\infty_X$-modules, however, the situation is simplified.

\begin{lem}
\label{lem:Gamma-vs-otimes}
Given sheaves of $C^\infty_X$-modules $\mathcal M$ and $\mathcal N$, there is a natural isomorphism of $C^\infty(X)$-modules
\begin{equation*}
\mathcal M(X) \otimes_{C^\infty(X)} \mathcal N(X)
\xrightarrow{\cong}
(\mathcal M \otimes_{C^\infty_X} \mathcal N)(X).
\end{equation*}
\end{lem}

Before proving the lemma, we recall some notation and fundamental facts.  Let $\mathcal M$, $\mathcal N$ and $\mathcal Z$ be sheaves of $C^\infty_X$-modules.
\begin{enumerate}
\item
The set of homomorphisms $\mathcal M \to \mathcal N$ is denoted $\Hom_{C^\infty_X}(\mathcal M, \mathcal N)$.  It has a natural structure of a $C^\infty(X)$-module: for any $f\in C^\infty(X)$, $F\in \Hom_{C^\infty_X}(\mathcal M, \mathcal N)$, and an open set $U$ of $X$, set $(f\cdot F)(U):=f|_{U}F(U)$.
\item
The sheaf of homomorphisms from $\mathcal M$ to $\mathcal N$ assigns to an open set $U \subseteq X$ the $C^\infty(U)$-module $\Hom_{C^\infty_U}(\mathcal M\vert_U, \mathcal N\vert_U)$.  This sheaf is denoted $\underline\Hom_{C^\infty_X}(\mathcal M, \mathcal N)$, and it has a natural structure of a $C^\infty_X$-module. Note that $\Hom_{C^\infty_X}(\mathcal M, \mathcal N)$ is the set of global sections of the sheaf $\underline\Hom_{C^\infty_X}(\mathcal M, \mathcal N)$.
\item
The previous construction is the internal hom in the category of sheaves of $C^\infty_X$-modules, in the sense that there is a natural isomorphism (hom-tensor adjunction)
  \begin{equation*}
  \Hom_{C^\infty_X}(\mathcal M \otimes_{C^\infty_X} \mathcal N, \mathcal Z)
  \cong
  \Hom_{C^\infty_X}(\mathcal M,
  \underline\Hom_{C^\infty_X}(\mathcal N, \mathcal Z)).
  \end{equation*}
\end{enumerate}

\begin{proof}[Proof of Lemma \ref{lem:Gamma-vs-otimes}]
Consider the map
\begin{equation}
\label{eq:Hom-vs-Gamma}
\Hom_{C^\infty_X}(\mathcal M, \mathcal N)\
\to \Hom_{C^\infty(X)}(\mathcal M(X), \mathcal N(X)),
\end{equation}
which sends a sheaf homomorphism to the associated map on
global sections. We claim that, because all sheaves involved are fine,
\eqref{eq:Hom-vs-Gamma} is an isomorphism of $C^\infty(X)$-modules. To
show that \eqref{eq:Hom-vs-Gamma} is injective, suppose that
$F\colon\mathcal M \to \mathcal N$ is such that $F_X\colon \mathcal
M(X) \to \mathcal N(X)$ vanishes. We need to show that $F_U\colon
\mathcal M(U) \to \mathcal N(U)$ vanishes for any open $U \subseteq X$.
Fix $\sigma \in \mathcal M(U)$, let $U' \subseteq X$ be open with
closure contained in $U$, and let $\rho$ be a function with support in
$U$ and $\rho\vert_{U'} = 1$
Then $\rho\sigma$ defines a global section of $\mathcal M$, and
\begin{equation*}
F_U(\sigma)\vert_{U'}
= F_{U'}(\rho\sigma\vert_{U'})
= F_X(\rho\sigma)\vert_{U'} = 0.
\end{equation*}
Since, by the sheaf property of $\mathcal N$, any section over $U$ can be reconstructed from its restrictions to each open $U'$ with closure in $U$, we conclude that $F\vert_U = 0$ and therefore $F = 0$.

To show that \eqref{eq:Hom-vs-Gamma} is surjective, fix $F_X\colon \mathcal M(X) \to \mathcal N(X)$. We want to find a sheaf homomorphism $F\colon \mathcal M \to \mathcal N$ that acts as $F_X$ on global sections. Fix $\sigma \in \mathcal M(U)$ and consider, for any open $U'$ with closure in $U$, a function $\rho$ as above. Then we claim that the section $F_X(\rho\sigma)\vert_{U'} \in \mathcal N(U')$
is independent of $\rho$.  In fact, let $\rho'$ be another admissible choice, and set $\eta = \rho - \rho'$.  Then we want to show that $F_X(\eta\sigma)\vert_{U'} = 0$.  Choosing a function $\xi$ that is identically $1$ in $U'$ and has support in the set $\{x\in X \mid \eta(x) = 0 \}$, we can write
\begin{equation*}
F_X(\eta\sigma)\vert_{U'}
= \xi\vert_{U'} \cdot F_X(\eta\sigma)\vert_{U'}
= F_X(\xi\eta\sigma)\vert_{U'}
= 0,
\end{equation*}
since the section $\xi\eta\si$ vanishes globally by definition of $\xi$. Now, using the sheaf property of $\mathcal N$ to glue these coherent values $F_X(\rho\sigma)\vert_{U'} \in \mathcal N(U')$ as $U' \subseteq U$ varies over all open subsets with closure in $U$, we obtain an element $F_U(\sigma) \in \mathcal N(U)$. This recipe, applied to the open $X$, gives us the original $F_X$, and the naturality condition $F_U(\sigma)\vert_V = F_V(\sigma\vert_V)$ is satisfied for all $V \subseteq U$. This proves \eqref{eq:Hom-vs-Gamma} is surjective.

Finally, let $\mathcal Z$ be an arbitrary sheaf of $C^\infty_X$-modules.
Then the lemma follows from the Yoneda lemma and the following natural
isomorphisms:
\begin{equation*}
\def\iHom{\underline\Hom_{C^\infty_X}}
\def\sHom{\Hom_{C^\infty_X}}
\def\mHom{\Hom_{C^\infty(X)}}
\def\sotimes{\otimes_{C^\infty_X}}
\def\motimes{\otimes_{C^\infty(X)}}
\begin{aligned}
  \mHom((\mathcal M \sotimes \mathcal N)(X), \mathcal Z(X))
     & \cong \sHom(\mathcal M \sotimes \mathcal N, \mathcal Z)
  \\ & \cong \sHom(\mathcal M, \iHom(\mathcal N, \mathcal Z))
  \\ & \cong \mHom(\mathcal M(X), \sHom(\mathcal N, \mathcal Z))
  \\ & \cong \mHom(\mathcal M(X), \mHom(\mathcal N(X), \mathcal Z(X)))
  \\ & \cong \mHom(\mathcal M(X) \motimes \mathcal N(X), \mathcal Z(X)).
\end{aligned}
\end{equation*}
Above, $\underline\Hom_{C^\infty_X}$ denotes the sheaf of homomorphisms between $C^\infty_X$-modules (which is again a $C^\infty_X$-module), while $\Hom_{C^\infty_X}$ denotes its global sections, that is, the $C^\infty(X)$-module of sheaf homomorphisms.  Each of the above identifications is justified by the isomorphism \eqref{eq:Hom-vs-Gamma}, the adjunction between $\Hom_{C^\infty(X)}$ and $\otimes_{C^\infty(X)}$, or its analogue for $\underline\Hom_{C^\infty_X}$ and $\otimes_{C^\infty_X}$.
\end{proof}

\begin{cor}
Suppose that the sheaves of $C^\infty_X$-modules $\mathcal M$ and $\mathcal N$ are locally free of bounded rank.  Then so is their tensor product $\mathcal M \otimes_{C^\infty_X} \mathcal N$.  Thus, $\Lfb(X)$ is a symmetric monoidal category.  Moreover, $\Gamma_X$ is a symmetric monoidal functor.
\end{cor}

\begin{prp}
The functor $\Gamma_\loc\colon \Bun^{\text{iso}}(X) \to \Lfb^{\text{iso}}(X)$ from~(\ref{eq:bun-lfb-fpg}) is symmetric monoidal.
\end{prp}

\begin{proof}
Let $E, F \in \Bun^{\text{iso}}(X)$. The natural bilinear map $E \times F \to E \otimes F$ induces, on local sections, a map $\Gamma_\loc(E) \otimes_{C^\infty_X} \Gamma_\loc(F) \to \Gamma_\loc(E \otimes F)$. We need to show that this is an isomorphism.  Taking stalks, this gives a map
\begin{equation}
\label{eq:1}
\Gamma_\loc(E)_x \otimes_{C^\infty_{X, x}} \Gamma_\loc(F)_x
\cong
(\Gamma_\loc(E) \otimes_{C^\infty_X} \Gamma_\loc(F))_x
\to
\Gamma_\loc(E \otimes F)_x.
\end{equation}
The isomorphism on the left-hand side follows from the fact that the tensor product preserves colimts. Next, by taking stalks at $x \in X$, we get $\Gamma_\loc(E)_x \cong C^\infty_{X, x} \otimes E_x$, where $E_x$ is the fiber of $E$ at $x$ and all other subscripts $x$ denote stalks.  Using this, and the corresponding fact for $F$ and $E \otimes F$, we see that \eqref{eq:1} is an isomorphism.  A map of sheaves inducing isomorphisms on stalks is an isomorphism, as desired.
\end{proof}

\begin{cor}
\label{cor:1}
The functor $\Gamma\colon \Bun^{\text{iso}}(X) \to \Fgp^{\text{iso}}(X)$ is symmetric monoidal. In particular, given vector bundles $E$ and $F$ over $X$, there is an isomorphism
\begin{equation*}
  \Gamma(E\otimes F) \cong \Gamma(E) \otimes_{C^\infty(X)} \Gamma(F)
\end{equation*}
of $C^\infty(X)$-modules.
\end{cor}

\begin{proof}
The natural isomorphism \eqref{eq:bun-lfb-fpg} between $\Gamma$ and the monoidal functor $\Gamma_X \circ \Gamma_{\loc}$ gives the former the structure of a monoidal functor. For a more direct proof, not involving sheaf tensor products, see Nestruev~\cite[Theorem~11.39]{Nes}.
\end{proof}

Recall that a connection on a vector bundle $E \to X$ is a $\mathbb
C$-linear map
\begin{equation*}
  \nabla\colon \Gamma(E) \to \Gamma(E \otimes T^*X)
\end{equation*}
satisfying the familiar version of the Leibniz rule. By Corollary \ref{cor:1} (with $F$ being the cotangent bundle $T^*X$ of $X$), $\nabla$ determines a connection on the $C^\infty(X)$-module $\Gamma(E)$. This assignment extends to a functor
\begin{equation}\label{EQN.GammaCon}
  \Gamma^\con\colon \BunD^{\text{iso}}(X) \to \FgpD^{\text{iso}}(X),
\end{equation}
which on morphisms does the same as the functor $\Gamma$ of
\eqref{eq:bun-lfb-fpg}.

Similarly, applying Lemma \ref{lem:Gamma-vs-otimes} to the case $\mathcal N = \Omega^1_X$, we see that a connection on a $C^\infty_X$-module $\mathcal M$ determines a connection on its global sections $\Gamma_X(\mathcal M)$. Thus, we get a functor $\Gamma_X^\con\colon \LfbD(X) \to \FgpD^{\text{iso}}\big(C^\infty(X)\big)$.

\begin{lem}
\label{lem:buncontofgpcon}
Let $\mathcal M$ be a sheaf of $C^\infty_X$-modules and $M = \mathcal M(X)$ the $C^\infty(X)$-module of global sections. Then a connection $D$ on $M$ uniquely determines a connection $\mathcal D$ on $\mathcal M$ characterized by the property that $\mathcal{D}_{X}=D$.
\end{lem}

\begin{proof}
Set $\mathcal D_X\colon\mathcal M(X) \xrightarrow D \mathcal M(X) \otimes_{C^{\infty}(X)} \Omega^{1}(X) \cong (\mathcal M \otimes_{C^\infty_X} \Omega_X^1)(X)$ to be $D$ followed by the isomorphism from Lemma \ref{lem:Gamma-vs-otimes}. Our goal is to show this extends to a unique connection $\mathcal D_U\colon \mathcal M(U) \to  \mathcal M(U) \otimes_{C^{\infty}(U)} \Omega^{1}(U) \cong (\mathcal M \otimes_{C^\infty} \Omega^1)(U)$ on each open $U\subseteq X$.

Let $U \subseteq X$ be open and $\sigma \in \mathcal M(U)$.  Assume for the moment that $\sigma$ admits an extension $\sigma'\in \mathcal M(X)$ to a global section.  Then we set
$\mathcal D_U \sigma := (\mathcal D_X \sigma')\vert_U.$ Note that this is independent of the choice of $\sigma'$. To see this, pick a different extension $\sigma'' \in \mathcal M(X)$ of $\sigma$, so that $\sigma' - \sigma'' = 0$ on $U$. For each $x \in U$, we can choose a function $\rho \in C^\infty(X)$ with support in $U$ and such that $\rho = 1$ in a neighborhood of $x$, so that $\rho (\sigma' - \sigma'') = 0$ on all of $X$. Then, in a neighborhood of $x$, we have
  \begin{equation*}
    0 = \mathcal D_X \rho (\sigma' - \sigma'')
    = d\rho \otimes (\sigma' - \sigma'')
    + \rho \mathcal D_X (\sigma' - \sigma'')
    = \mathcal D_X (\sigma' - \sigma'').
  \end{equation*}
By the sheaf property (or, more specifically, the fact that sections agree if they agree locally), we see that $\mathcal D_U\sigma$ is well-defined.

Now, in general, a section $\sigma \in \mathcal E(U)$ may not admit a global extension. Nevertheless, for each $x\in U$, there exists an open neighborhood $U'\subseteq U$ of $x$ such that $\sigma|_{U'}$ has a global extension to $X$.
Thus, using global extensions of restrictions $\sigma\vert_{U'}$ for small open disks $U' \subset U$ and using the sheaf condition, we get a unique $\mathcal D_U$ with the property that $\mathcal D_U (\sigma\vert_U) = (\mathcal D_X \sigma)\vert_{U}$ for all global sections $\sigma \in \mathcal M(X)$.

It remains to check that $\mathcal D_U$ satisfies the Leibniz rule.  This holds if and only if it holds for all sections with compact support in $U$.  Now, the condition $\mathcal D_U (\sigma\vert_U) = (\mathcal D_X \sigma)\vert_{U}$, and the fact that $\mathcal D_X$ satisfies the Leibniz rule, shows that the same is true of $\mathcal D_U$.
\end{proof}

We now define a functor $\Gamma_\loc^\con\colon \BunD^{\text{iso}}(X) \to \LfbD(X)$. Given a vector bundle with connection $(E, \nabla)$, we set $\Gamma_\loc^\con(E, \nabla) := (\mathcal E, \mathcal D)$, where $\mathcal E := \Gamma_\loc(E)$ is the sheaf of sections of $E$, and $\mathcal D$ is the connection obtained via Lemma~\ref{lem:buncontofgpcon} from the connection on $\Gamma^\con(E, \nabla)$. At the level of morphism, $\Gamma_\loc^\con$ performs the same assignments as its counterpart in diagram \ref{eq:bun-lfb-fpg}. One can check that this is compatible with the connections constructed in Lemma~\ref{lem:buncontofgpcon}.

To sum up, we obtained a diagram of groupoids and functors
\begin{equation}
\label{eq:bun-lfb-fgp-conn}
\vcenter{\hbox{
\xymatrix{
\BunD^{\text{iso}}(X) \ar[d]_{\Gamma_{\mathrm{loc}}^\con} \ar[r]^{\Gamma^\con} & \FgpD^{\text{iso}}(X)\\
\LfbD(X) \ar[ur]_{\Gamma_X^\con} &
}
}},
\end{equation}
which commutes up to isomorphism.
The reason this diagram does not commute on the nose is because the definition of a connection involves the tensor product and these functors are only monoidal up to isomorphism.

\begin{prp}[Differential Serre--Swan theorem] \label{PRP.Differential.SS} Each functor in \eqref{eq:bun-lfb-fgp-conn} is an equivalence of categories.
\end{prp}

\begin{proof} The forgetful functors mapping the triangle \eqref{eq:bun-lfb-fgp-conn} to its version \eqref{eq:bun-lfb-fpg} without connection are faithful, since morphisms in the former are morphisms in the latter satisfying additional conditions (to preserve connections, in each of the three different senses). It follows that all functors in \eqref{eq:bun-lfb-fgp-conn} are faithful.

By a similar argument, $\Gamma^\con$ is full.  We now verify that $\Gamma^\con$ is essentially surjective. Take any $(M,D)\in\FgpD^{\text{iso}}(X)$. The essential surjectivity of the horizontal map in \eqref{eq:bun-lfb-fpg} gives us an isomorphism $\al:\Ga(E)\srl{\isom}\ra M$ for some $E\in \Bun^{\text{iso}}(X)$. We define a connection $\na$ on $E$ by the composite
\begin{equation*}
\na:\Ga(E)\srl{\al}\ra M\srl{D}\ra M\tsr_{\Cinf(X)}
\Om^1(X)\xrightarrow{\al^{-1}\tsr\text{id}} \Ga(E)
\tsr_{\Cinf(X)}\Om^1(X)\srl{\isom}\ra\Ga(E\tsr T^*X).
\end{equation*}
Since $D$ satisfies the Leibniz rule, so does $\nabla$. Furthermore, $\alpha$ defines an isomorphism of modules with connection $\alpha:(\Gamma(E),\nabla)\to(M,D)$ by construction of $\na$. This proves that $\Gamma^\con$ is essentially surjective.

We conclude that $\Gamma_X^\con$ is essentially surjective and full since $\Gamma^\con$ is essentially surjective and full and the former factors through the latter. We already know they are faithful, and it follows that $\Gamma^\con$ and $\Gamma_X^\con$ are equivalences. Therefore, $\Gamma_\loc^\con$ is an equivalence as well.
\end{proof}

\begin{dfn}\label{DFN.Grassmann.conn} Let $M\in\Fgp(A)$ and $\Omega_{\bullet}(A)$ a DGA on top of $A$. A \define{Grassmann connection} on a $M$ is a connection $D$ on $M$ for which there exists an $n\in\N$, an embedding $i:M\hookrightarrow A^{n}$, and a projection $p:A^{n}\to M$ such that $D:=(p\otimes\mathrm{id})\circ d\circ i$, where $d$ is the standard connection on $A^{n}$.
\end{dfn}

\begin{rmk}
A Grassmann connection is also called a \define{Berry connection} in the physics literature~\cite{Si83} and a \define{Levi-Civita connection} by Karoubi~\cite{Ka1}.
\end{rmk}

\begin{cor}[Narasimhan and Ramanan]\label{COR.all.connections.are.grassmann.conn.for.smooth.functions}
Let $X$ be a manifold and $A=\Cinf(X;\C)$. Then any connection on a finitely generated projective $A$-module is a Grassmann connection.
\end{cor}

\begin{proof} Suppose $(M,D)\in\FgpD^{\text{iso}}(\Cinf(X;\C))$. From the proof of Proposition \ref{PRP.Differential.SS} on essential surjectivity of $\Ga^\con$, there exists $(E,\na)\in \BunD^{\text{iso}}(X)$ together with an isomorphism $\alpha:\Gamma(E)\srl{\isom}\ra M$ such that $\na=(\al^{-1} \tsr \text{id})\circ D \circ \al$. By the theorem of Narasimhan and Ramanan \cites{NR1, NR2,Qu}, $\na= (p \tsr \text{id})\circ D \circ i$, where $p:X\times \C^N\ra E$ is a projection such that $\text{Im}(p)=E$ and $i:E\ra X\times \C^N$ is an embedding. Therefore $D= (\al\tsr\text{id})\circ p \circ d \circ i\circ \al^{-1}$, which is a Grassmann connection.
\end{proof}

\section{On Grothendieck fibrations and the extension of scalars}
\label{sec:fibrations}

Naturality of the noncommutative Chern character necessitates the need to introduce the extension of scalars. Since naturality is a categorical notion, it is important to describe the extension of scalars from a functorial viewpoint, and, for this reason, it is helpful to review Grothendieck fibrations and opfibrations. This justifies some of the terminology we have been using, such as \emph{cocartesian morphisms}. We also make everything explicit since taking opposite categories can be quite confusing due to the many ways in which it can occur. Our references include~\cites{St18,MoVa18,Ma19}.

\begin{dfn}
\label{defn:cartesian}
Let $\mathcal{E}$ and $\mathcal{X}$ be two categories and let $\mathcal{E}\xrightarrow{\pi}\mathcal{X}$ be a functor. A morphism $u\xrightarrow{\beta}v$ in $\mathcal{E}$ is \define{cartesian} iff for any morphism $x\xrightarrow{f}\pi(u)$ in $\mathcal{X}$ and any morphism $w\xrightarrow{\gamma}v$ in $\mathcal{E}$ such that $\pi(\beta)\circ f=\pi(\gamma)$, there exists a unique morphism $w\xrightarrow{\alpha}u$ in $\mathcal{E}$ such that
\[
\pi(\alpha)=f
\qquad\text{ and }\qquad
\beta\circ\alpha=\gamma.
\]
\end{dfn}

A visualization of the data in Definition~\ref{defn:cartesian} is often helpful:%
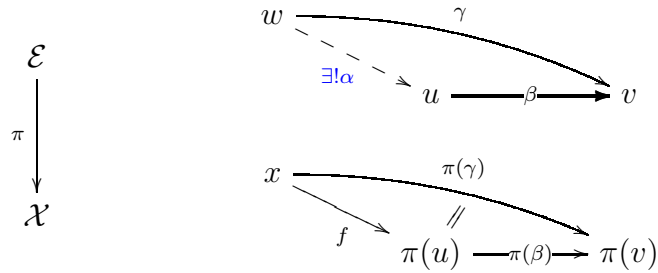
\footnote{The morphism $\beta$ has been bolded to emphasize that it is the morphism that is cartesian.}
\[
\xy0;/r.25pc/:
(0,10)*+{\mathcal{E}}="E";
(0,-10)*+{\mathcal{X}}="X";
{\ar"E";"X"_{\pi}};
\endxy
\qquad\qquad\qquad
\xy0;/r.25pc/:
(0,5)*+{u}="u";
(25,5)*+{v}="v";
(-20,15)*+{w}="w";
{\ar"u";"v"|-{\beta}};
{\ar@<0.075ex>"u";"v"|-{\phantom{\beta}}};
{\ar@<-0.075ex>"u";"v"|-{\phantom{\beta}}};
{\ar@<0.05ex>"u";"v"|-{\phantom{\beta}}};
{\ar@<-0.05ex>"u";"v"|-{\phantom{\beta}}};
{\ar@<0.025ex>"u";"v"|-{\phantom{\beta}}};
{\ar@<-0.025ex>"u";"v"|-{\phantom{\beta}}};
{\ar@/^0.75pc/"w";"v"^{\gamma}};
{\ar@{-->}@[blue]"w";"u"_{\color{blue}\exists!\alpha}};
(0,-15)*+{\pi(u)}="pu";
(25,-15)*+{\pi(v)}="pv";
(-20,-5)*+{x}="pw";
{\ar"pu";"pv"|-{\pi(\beta)}};
{\ar@/^0.75pc/"pw";"pv"^{\pi(\gamma)}};
{\ar"pw";"pu"_{f}};
{\ar@{=}"pu";(4,-9)};
\endxy
\]
Note that the conditions necessitate $\pi(w)=x$ so that one can think of the morphisms in $\mathcal{E}$ as lying above the morphisms in $\mathcal{X}$ via the functor $\pi$. In Example~\ref{exa:VectManFibration}, we will see that a cartesian morphism is a fibrewise isomorphism in the setting of vector bundles over smooth manifolds.

A closely related definition is that of a cocartesian morphism.

\begin{dfn}
Let $\mathcal{E}$ and $\mathcal{X}$ be two categories and let $\mathcal{E}\xrightarrow{\pi}\mathcal{X}$ be a functor. A morphism $u\xleftarrow{\beta}v$ in $\mathcal{E}$ is \define{cocartesian} iff for any morphism $x\xleftarrow{f}\pi(u)$ in $\mathcal{X}$ and any morphism $w\xleftarrow{\gamma}v$ in $\mathcal{E}$ such that $f\circ\pi(\beta)=\pi(\gamma)$, there exists a unique morphism $w\xleftarrow{\alpha}u$ in $\mathcal{E}$ such that
\[
\pi(\alpha)=f
\qquad\text{ and }\qquad
\alpha\circ\beta=\gamma.
\]
\end{dfn}

Cocartesian morphisms are visualized as:
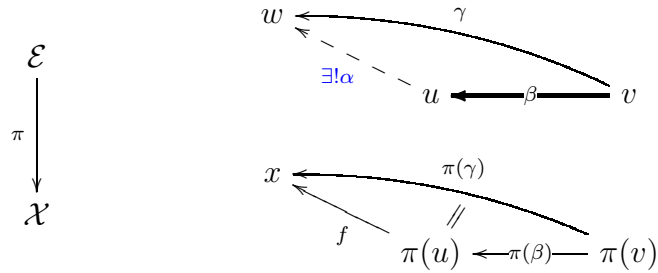
\[
\xy0;/r.25pc/:
(0,10)*+{\mathcal{E}}="E";
(0,-10)*+{\mathcal{X}}="X";
{\ar"E";"X"_{\pi}};
\endxy
\qquad\qquad\qquad
\xy0;/r.25pc/:
(0,5)*+{u}="u";
(25,5)*+{v}="v";
(-20,15)*+{w}="w";
{\ar"v";"u"|-{\beta}};
{\ar@<0.075ex>"v";"u"|-{\phantom{\beta}}};
{\ar@<-0.075ex>"v";"u"|-{\phantom{\beta}}};
{\ar@<0.05ex>"v";"u"|-{\phantom{\beta}}};
{\ar@<-0.05ex>"v";"u"|-{\phantom{\beta}}};
{\ar@<0.025ex>"v";"u"|-{\phantom{\beta}}};
{\ar@<-0.025ex>"v";"u"|-{\phantom{\beta}}};
{\ar@/_0.75pc/"v";"w"_{\gamma}};
{\ar@{-->}@[blue]"u";"w"^{\color{blue}\exists!\alpha}};
(0,-15)*+{\pi(u)}="pu";
(25,-15)*+{\pi(v)}="pv";
(-20,-5)*+{x}="pw";
{\ar"pv";"pu"|-{\pi(\beta)}};
{\ar@/_0.75pc/"pv";"pw"_{\pi(\gamma)}};
{\ar"pu";"pw"^{f}};
{\ar@{=}"pu";(4,-9)};
\endxy
\]
In other words, a cocartesian morphism in $\mathcal{E}\xrightarrow{\pi}\mathcal{X}$ is a cartesian morphism in $\mathcal{E}^{\op}\xrightarrow{\pi^{\op}}\mathcal{X}^{\op}$.

\begin{dfn}
\label{defn:fibration}
Let $\pi:\mathcal{E}\to\mathcal{X}$ be as in Definition~\ref{defn:cartesian}. Let $\mathcal{E}_{x}$ be the subcategory of $\mathcal{E}$ consisting of the objects $u$ in $\mathcal{E}$ such that $\pi(u)=x$ and $\pi(\beta)=\mathrm{id}_{x}$ for all morphisms $u\xrightarrow{\beta}v$ with $\pi(u)=x=\pi(v)$. The category $\mathcal{E}_{x}$ is called the \define{fibre} of $\pi$ over $x$ and the morphisms in $\mathcal{E}_{x}$ are called \define{vertical morphisms} of $\pi$ over $x$.
Given a morphism $x\xrightarrow{f}y$ in $\mathcal{X}$ and an object $v$ in $\mathcal{E}_{y}$, a \define{cartesian lifting of $f$ with target $v$} is a cartesian morphism $u\xrightarrow{\beta}v$ such that $\pi(\beta)=f$.
A functor $\pi:\mathcal{E}\to\mathcal{X}$ is a \define{fibration} iff for any morphism  $x\xrightarrow{f}y$ in $\mathcal{X}$ and an object $v$ in $\mathcal{E}_{y}$, a cartesian lifting exists.
Dually, given a morphism $x\xleftarrow{f}y$ in $\mathcal{X}$ and an object $v$ in $\mathcal{E}_{y}$, a \define{cocartesian lifting of $f$ with source $v$} is a cocartesian morphism $u\xleftarrow{\beta}v$ such that $\pi(\beta)=f$.
A functor $\pi:\mathcal{E}\to\mathcal{X}$ is a \define{opfibration} iff for any morphism  $x\xleftarrow{f}y$ in $\mathcal{X}$ and an object $v$ in $\mathcal{E}_{y}$, a cocartesian lifting exists.
When $\pi$ is either a fibration or an opfibration, $\mathcal{E}$ is called the \define{total category} and $\mathcal{X}$ is called the \define{base}.
Given a fibration (opfibration) $\mathcal{E}\xrightarrow{\pi}\mathcal{X}$, the collection of cartesian (cocartesian) morphisms form a fibration (opfibration) $\mathbf{cart}(\mathcal{E})\xrightarrow{\pi}\mathcal{X}$ ($\mathbf{cocart}(\mathcal{E})\xrightarrow{\pi}\mathcal{X}$).
\end{dfn}

\begin{exa}
\label{exa:VectManFibration}
Let $\Man$ denote the category of smooth manifolds and smooth maps. Let $\Bun$ denote the category of smooth vector bundles over manifolds. Namely, an object of $\Bun$ is a triple $(X,E,\pi)$, with $X$ a smooth manifold and $E\xrightarrow{\pi}X$ a smooth vector bundle. A morphism $(X,E,\pi)\to(Y,F,\rho)$ consists of a pair of smooth maps $(X\xrightarrow{f}Y,E\xrightarrow{\varphi}F)$ such that $\varphi$ is a fibrewise linear map and such that the diagram
\[
\xy0;/r.25pc/:
(-12.5,7.5)*+{E}="1";
(12.5,7.5)*+{F}="2";
(-12.5,-7.5)*+{X}="3";
(12.5,-7.5)*+{Y}="4";
{\ar"1";"2"^{\varphi}};
{\ar"1";"3"_{\pi}};
{\ar"2";"4"^{\rho}};
{\ar"3";"4"_{f}};
\endxy
\]
commutes. Then the projection map $\Bun\to\Man$, sending an object $(X,E,\pi)$ to $X$ and a morphism $(X,E,\pi)\xrightarrow{(f,\varphi)}(Y,F,\rho)$ to $f$, is a Grothendieck fibration. Indeed, to every morphism $X\xrightarrow{f}Y$ in $\Man$, and to every vector bundle $(Y,F,\rho)$ over $Y$, the \emph{pullback bundle}, defined by the categorical pullback
\[
\xy0;/r.25pc/:
(-12.5,7.5)*+{f^{*}F}="1";
(12.5,7.5)*+{F}="2";
(-12.5,-7.5)*+{X}="3";
(12.5,-7.5)*+{Y}="4";
(-8.5,5)*{\lrcorner};
{\ar"1";"2"};
{\ar"1";"3"_{f^{*}\rho}};
{\ar"2";"4"^{\rho}};
{\ar"3";"4"_{f}};
\endxy
\]
in the category $\Man$, is a cartesian morphism over $f$. This follows immediately from the universal property of the pullback---see the visualization after Definition~\ref{defn:cartesian}. This same universal property shows that all cartesian morphisms are of this form. Thus, $\mathbf{cart}(\Bun)$ can be identified with the category of vector bundles as objects and smooth maps on the base with fibrewise isomorphisms on the vector bundles as morphisms.
\end{exa}

\begin{exa}
\label{exa:ringsmodulesopfibration}
Let $\Ring$ denote the category of unital rings and unital ring homomorphisms (henceforth, all rings will be assumed unital in this appendix). Let $\Mod$ denote the category of modules. Namely, an object of $\Mod$ is a pair $(R,M)$ consisting of a ring $R$ and a right $R$-module $M$. A morphism $(S,N)\to(R,M)$ is a pair $(\psi,\Psi)$ with $S\xrightarrow{\psi}R$ a ring homomorphism and $N\xrightarrow{\Psi}M$ a $\Z$-linear map such that $\Psi(ns)=\Psi(n)\psi(s)$ for all $n\in N$ and $s\in S$. Then the forgetful functor $\Mod\to\Ring$ sending a pair $(R,M)$ to $R$ and $(\psi,\Psi)$ to $\psi$ is an opfibration. Indeed, to every ring homomorphism $S\xrightarrow{\psi}R$ and to every right $S$-module $N$, a cocartesian lift is given by the extension of scalars right $R$-module $N\otimes_{S}R:=N\otimes_{\Z}R/_{\sim}$, together with the $\Z$-linear map $N\xrightarrow{\Psi}N\otimes_{S}R$ defined by sending $n\in N$ to $n\otimes1_{R}$. This is cocartesian because to any other ring $Q$ and right $Q$-module $L$ together with morphisms as in
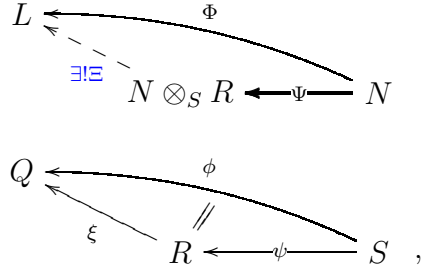
\[
\xy0;/r.25pc/:
(0,5)*+{N\otimes_{S}R}="u";
(25,5)*+{N}="v";
(-20,15)*+{L}="w";
{\ar"v";"u"|-{\Psi}};
{\ar@<0.075ex>"v";"u"|-{\phantom{\Psi}}};
{\ar@<-0.075ex>"v";"u"|-{\phantom{\Psi}}};
{\ar@<0.05ex>"v";"u"|-{\phantom{\Psi}}};
{\ar@<-0.05ex>"v";"u"|-{\phantom{\Psi}}};
{\ar@<0.025ex>"v";"u"|-{\phantom{\Psi}}};
{\ar@<-0.025ex>"v";"u"|-{\phantom{\Psi}}};
{\ar@/_0.75pc/"v";"w"_{\Phi}};
{\ar@{-->}@[blue]"u";"w"^{\color{blue}\exists!\Xi}};
(0,-15)*+{R}="pu";
(25,-15)*+{S}="pv";
(30,-16)*+{,};
(-20,-5)*+{Q}="pw";
{\ar"pv";"pu"|-{\psi}};
{\ar@/_0.75pc/"pv";"pw"_{\phi}};
{\ar"pu";"pw"^{\xi}};
{\ar@{=}"pu";(4,-9)};
\endxy
\]
the unique morphism $N\otimes_{S}R\xrightarrow{\Xi}L$ in $\Mod$ is uniquely determined by the assignment
\[
N\otimes_{\K}R\ni n\otimes r\mapsto \Phi(n)\xi(r).
\]

A quick calculation shows that this descends to a well-defined morphism $N\otimes_{S}R\xrightarrow{\Xi}L$ in $\Mod$ satisfying the required commutativity conditions. Its uniqueness follows from the fact that it is actually completely determined by its action on elements of the form $n\otimes 1_{R}$, which must be given by $\Phi(n)$ in order for the top diagram in $\Mod$ to commute.
Furthermore, all cocartesian morphisms are of this form. More precisely, if $N\xrightarrow{\Sigma}M$ is another cocartesian lift of $\psi$, then $M$ is canonically isomorphic to $N\otimes_{S}R$ in the category $\Mod$.
This follows from the fact that \emph{any} two target objects of cocartesian lifts are canonically isomorphic (this follows from the usual arguments of canonical uniqueness of objects defined via universal properties).

Thus, the category of cocartesian morphisms $\mathbf{cocart}(\Mod)$ is equivalent to the category of modules and morphisms where the morphisms always map into a module that is canonically isomorphic to an extension of scalars of the first module along the ring homomorphism in the base.
\end{exa}

Fibrations and opfibrations are occasionally constructed from functors via the \emph{Grothendieck construction}, though elements of this construction were known earlier by Yoneda and Mac~Lane (cf.~\cite[page 44]{MaMo12}).

\begin{lem}
\label{lem:grothendieckconstruction}
Let $\mathcal{X}$ be a category and let $\mathcal{X}^{\op}\xrightarrow{\mathfrak{E}}\mathbf{Cat}$ be a pseudofunctor, also called a weak 2-functor (for definitions, see~\cites{Pa19,KeSt74}). Then the following structure defines a category $\int\mathfrak{E}$.
\begin{enumerate}
\item
An object of $\int\mathfrak{E}$ is a pair $(X,E)$ with $X$ an object of $\mathcal{X}$ and $E$ an object of $\mathfrak{E}_{X}$.
\item
A morphism from $(X,E)$ to $(Y,F)$ is a pair $(f,\alpha)$, with $X\xrightarrow{f}Y$ a morphism in $\mathcal{X}$ and $E\xrightarrow{\alpha}\mathfrak{E}_{f}(F)$ a morphism in $\mathfrak{E}_{X}$.
\item
The composite of a composable pair of morphisms $(X,E)\xrightarrow{(f,\alpha)}(Y,F)\xrightarrow{(g,\beta)}(Z,G)$ is given by the composite $X\xrightarrow{g\circ f}Z$ in $\mathcal{X}$ and
\[
E\xrightarrow{\alpha}\mathfrak{E}_{f}(F)\xrightarrow{\mathfrak{E}_{f}(\beta)}\mathfrak{E}_{f}\big(\mathfrak{E}_{g}(G)\big)\xrightarrow{c_{(f,g)}}\mathfrak{E}_{g\circ f}(G)
\]
in $\mathfrak{E}_{X}$,
where $c_{(f,g)}$ is the compositor natural isomorphism from the data of a pseudofunctor.
\item
The identity morphism associated to a pair $(X,E)$ in $\int\mathfrak{E}$ is given by $E\xrightarrow{\epsilon_{X}}\mathfrak{E}_{\mathrm{id}_{X}}(E)$, where $\epsilon$ is the unitor natural isomorphism from the data of the pseudofunctor $\mathfrak{E}$.
\end{enumerate}
Furthermore, the assignment
\[
\begin{split}
\int\mathfrak{E}&\xrightarrow{\pi_{\mathfrak{E}}}\mathcal{X}\\
(X,E)&\mapsto X\\
\Big((X,E)\xrightarrow{(f,\alpha)}(Y,F)\Big)&\mapsto\Big(X\xrightarrow{f}Y\Big)
\end{split}
\]
defines a functor. This functor is a fibration in which a cartesian lift of $X\xrightarrow{f}Y$ with target $F$ in $\mathfrak{E}_{Y}$ is the pair $(f,\mathrm{id}_{\mathfrak{E}_{f}(F)})$. Finally, a morphism $(X,E)\xrightarrow{(f,\alpha)}(Y,F)$ is cartesian if and only if $E\xrightarrow{\alpha}\mathfrak{E}_{f}(F)$ is an isomorphism in $\mathfrak{E}_{X}$.
\end{lem}

\begin{dfn}
\label{defn:grothendieckconstruction}
Given a pseudofunctor $\mathcal{X}^{\op}\xrightarrow{\mathfrak{E}}\mathbf{Cat}$, the fibration $\int\mathfrak{E}\xrightarrow{\pi_{\mathfrak{E}}}\mathcal{X}$ from Lemma~\ref{lem:grothendieckconstruction} is called the \define{Grothendieck construction/fibration} associated to $\mathfrak{E}$.
\end{dfn}

\begin{exa}
\label{exa:extensionscalarsgrothendieck}
The extension of scalars briefly described in Example~\ref{exa:ringsmodulesopfibration} can be described as a pseudofunctor $\mathfrak{E}:\Ring\to\mathbf{Cat}$ sending a ring $R$ to the category of right $R$-modules $\mathfrak{E}(R):=\Mod(R)$ and sending a ring homomorphism $S\xrightarrow{\psi}R$ to a functor $\Mod(S)\xrightarrow{\mathfrak{E}_{\psi}}\Mod(R)$ defined as follows. To each $S$-module $N$, the functor assigns the $R$-module $N\otimes_{S}R$. To each morphism of $S$-modules $N\xrightarrow{\kappa}N'$, it assigns the canonical extension $N\otimes_{S}R\xrightarrow{\wtl{\kappa}}N'\otimes_{S}R$ determined uniquely by the $R$-action on elements of the form $n\otimes1_{R}$. The reason $\mathfrak{E}$ is a pseudofunctor, as opposed to a (strict) functor, is that to a pair of composable ring homomorphisms $S\xrightarrow{\psi}R\xrightarrow{\xi}Q$, there is a natural isomorphism
\[
\xy0;/r.25pc/:
(-20,-7.5)*+{\Mod(Q)}="3";
(0,7.5)*+{\Mod(R)}="2";
(20,-7.5)*+{\Mod(S)}="1";
{\ar"1";"2"_{\mathfrak{E}_{\psi}}};
{\ar"2";"3"_{\mathfrak{E}_{\xi}}};
{\ar"1";"3"^{\mathfrak{E}_{\xi\circ\psi}}};
{\ar@{=>}(0,3);(0,-6)_{c_{(\xi,\psi)}}};
\endxy
\]
that is not necessarily the identity natural transformation. Indeed, to an $S$-module $N$, it provides the isomorphism
\[
\begin{split}
(N\otimes_{S}R)\otimes_{R}Q&\to N\otimes_{S}Q\\
n\otimes r\otimes q&\mapsto n\otimes\xi(r)q
\end{split}
\]
in $\Mod(Q)$.
The inverse is given by sending $n\otimes q$ to $n\otimes1_{R}\otimes q$. Secondly, given a ring $R$, there is a natural isomorphism
\[
\xy0;/r.25pc/:
(-17.5,0)*+{\Mod(R)}="1";
(17.5,0)*+{\Mod(R)}="2";
{\ar@/^1.25pc/"1";"2"^{\mathrm{id}_{\Mod(R)}}};
{\ar@/_1.25pc/"1";"2"_{\mathfrak{E}_{\mathrm{id}_{R}}}};
{\ar@{=>}(0,3);(0,-3)^{\epsilon_{R}}};
\endxy
\]
given on a right $R$-module $M$ by the isomorphism $M\to M\otimes_{R}R$ sending $m\in M$ to $m\otimes1_{R}$.
We leave the other verifications of the definition of a pseudofunctor to the reader.

We now apply the Grothendieck construction to $\mathfrak{E}:(\Ring^{\op})^{\op}\to\mathbf{Cat}$, where the double ``op'' is taken to more directly apply Definition~\ref{defn:grothendieckconstruction}. This gives a fibration $\int\mathfrak{E}\xrightarrow{\pi_{\mathfrak{E}}}\Ring^{\op}$, where the objects of $\int\mathfrak{E}$ are given by pairs $(R,M)$ with $R$ a ring and $M$ a right $R$-module. A morphism from $(R,M)$ to $(S,N)$ in $\int\mathfrak{E}$ is a pair $(\psi^{\op},\alpha)$ with $R\xrightarrow{\psi^{\op}}S$ a morphism in $\Ring^{\op}$, which corresponds to a ring homomorphism $S\xrightarrow{\psi}R$, and a right $R$-module morphism $M\xrightarrow{\alpha}N\otimes_{S}R$.
The projection $\int\mathfrak{E}\xrightarrow{\pi_{\mathfrak{E}}}\Ring^{\op}$ takes an object $(R,M)$ to $R$ and a morphism $(\psi^{\op},\alpha)$ to $\psi^{\op}$.
\end{exa}

In what follows, we will compare the three fibrations from Example~\ref{exa:VectManFibration}, Example~\ref{exa:ringsmodulesopfibration}, and Example~\ref{exa:extensionscalarsgrothendieck}. To do this, we need to make sense of morphisms of fibrations.

\begin{dfn}
Let $\mathcal{E}\xrightarrow{\pi}\mathcal{X}$ and $\mathcal{F}\xrightarrow{\rho}\mathcal{Y}$ be two fibrations. A \define{fibred functor} from $\pi$ to $\rho$ consists of a pair of functors $\big(\mathcal{E}\xrightarrow{\Gamma}\mathcal{F},\mathcal{X}\xrightarrow{\gamma}\mathcal{Y}\big)$ such that the diagram
\[
\xy0;/r.25pc/:
(-12.5,7.5)*+{\mathcal{E}}="1";
(12.5,7.5)*+{\mathcal{F}}="2";
(-12.5,-7.5)*+{\mathcal{X}}="3";
(12.5,-7.5)*+{\mathcal{Y}}="4";
{\ar"1";"2"^{\Gamma}};
{\ar"1";"3"_{\pi}};
{\ar"2";"4"^{\rho}};
{\ar"3";"4"_{\gamma}};
\endxy
\]
commutes and such that every cartesian lifting in $\pi$ gets sent to a cartesian lifting in $\rho$.
\end{dfn}

\begin{exa}
In Example~\ref{exa:extensionscalarsgrothendieck}, we showed that the extension of scalars $\Ring\xrightarrow{\mathfrak{E}}\mathbf{Cat}$ defines a fibration $\int\mathfrak{E}\xrightarrow{\pi_{\mathfrak{E}}}\Ring^{\op}$. Let $\Man\xrightarrow{C^{\infty}}\Ring^{\op}$ be the functor that takes the ring of smooth functions. More precisely, it takes a smooth manifold $X$ to the ring of smooth functions $C^{\infty}(X)$ on $X$. It sends a smooth map $X\xrightarrow{f}Y$ to the pullback ring homomorphism $C^{\infty}(Y)\xrightarrow{f^*}C^{\infty}(X)$. Second, let $\Bun\xrightarrow{\Gamma}\int\mathfrak{E}$ be the functor that sends a smooth vector bundle $(X,E,\pi)$ to the pair $\big(C^{\infty}(X),\Gamma(\pi)\big)$, where $\Gamma(\pi)$ is the $C^{\infty}(X)$-module of smooth sections of the vector bundle $E\xrightarrow{\pi}X$. To describe what this functor assigns to a morphism $(X,E,\pi)\xrightarrow{(f,\varphi)}(Y,F,\rho)$, we first note that associated to any such morphism in $\Bun$ is a uniquely determined morphism of bundles $(X,E,\pi)\xrightarrow{(\mathrm{id}_{X},\xi)}(X,f^*F,f^*\rho)$ coming from the universal property of the pullback
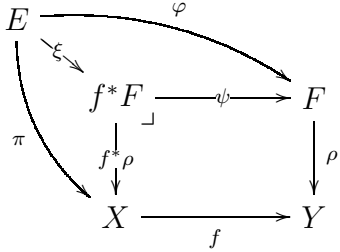
\[
\xy0;/r.25pc/:
(-25,17.5)*+{E}="0";
(-12.5,7.5)*+{f^{*}F}="1";
(12.5,7.5)*+{F}="2";
(-12.5,-7.5)*+{X}="3";
(12.5,-7.5)*+{Y}="4";
(-8.5,5)*{\lrcorner};
{\ar@/^1.0pc/"0";"2"^{\varphi}};
{\ar@/_1.0pc/"0";"3"_{\pi}};
{\ar"0";"1"|-(0.40){\xi}};
{\ar"1";"2"|-{\psi}};
{\ar"1";"3"|-{f^{*}\rho}};
{\ar"2";"4"^{\rho}};
{\ar"3";"4"_{f}};
\endxy
\]
This map induces a map of $C^{\infty}(X)$-modules $\Gamma(\pi)\xrightarrow{\Xi}\Gamma(f^*\rho)$ given by sending a section $s$ to $\xi\circ s$. This does not yet define a morphism in $\int\mathfrak{E}$. To obtain such a morphism, note that there is a natural $C^{\infty}(X)$-module map
\[
\begin{split}
\Gamma(\rho)\otimes_{C^{\infty}(Y)}C^{\infty}(X)&\xrightarrow{\Upsilon}\Gamma(f^*\rho)\\
t\otimes\eta&\mapsto\bigg(X\ni x\mapsto \psi^{-1}\Big(t\big(f(x)\big)\Big)\eta(x)\bigg),
\end{split}
\]
which is obtained by taking a section $t$ of $\rho$ to the section of $f^*\rho$ by composing the appropriate morphisms in the diagram
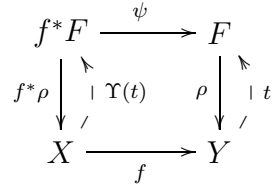
\[
\xy0;/r.25pc/:
(-10,7.5)*+{f^*F}="1";
(10,7.5)*+{F}="2";
(-10,-7.5)*+{X}="3";
(10,-7.5)*+{Y}="4";
{\ar"1";"2"^{\psi}};
{\ar"3";"4"_{f}};
{\ar"1";"3"_{f^*\rho}};
{\ar"2";"4"_{\rho}};
{\ar@/_1.0pc/@{-->}"4";"2"_{t}};
{\ar@/_1.0pc/@{-->}"3";"1"_{\Upsilon(t)}};
\endxy
\]
and using the fact that $\psi$ is a fibrewise isomorphism (and then just multiplying by the function $\eta$). This module map is an isomorphism satisfying many convenient properties, which are described in~\cite[Theorem~11.54]{Nes}.
For example, by sending the morphism $(X,E,\pi)\xrightarrow{(f,\varphi)}(Y,F,\rho)$ in $\Bun$ to the pair
\[
\Gamma(\pi)\xrightarrow{\Xi}\Gamma(f^*\rho)\xrightarrow{\Upsilon^{-1}}\Gamma(\rho)\otimes_{C^{\infty}(Y)}C^{\infty}(X)
\]
in $\int\mathfrak{E}$, the assignment $\Bun\xrightarrow{\Gamma}\int\mathfrak{E}$ defines a functor for which the diagram
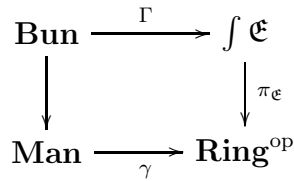
\[
\xy0;/r.25pc/:
(-12.5,7.5)*+{\Bun}="1";
(12.5,7.5)*+{\int\mathfrak{E}}="2";
(-12.5,-7.5)*+{\Man}="3";
(12.5,-7.5)*+{\Ring^{\op}}="4";
{\ar"1";"2"^{\Gamma}};
{\ar"1";"3"};
{\ar"2";"4"^{\pi_{\mathfrak{E}}}};
{\ar"3";"4"_{\gamma}};
\endxy
\]
commutes. Furthermore, a cartesian lift in $\Bun\to\Man$, which corresponds to the pullback bundle, gets sent to the extension of scalars module by~\cite[Theorem~11.54]{Nes}. Thus, $\Gamma$ defines a fibred functor.

A similar situation happens when working with (unital) algebras and right modules over them. In particular, the analogue of this fibred functor is fully faithful when one takes the subcategory of finitely generated projective modules.
\end{exa}

\begin{exa}
In Example~\ref{exa:ringsmodulesopfibration}, we showed that $\Mod\to\Ring$ is an \emph{opfibration}. Hence, $\Mod^{\op}\to\Ring^{\op}$ is a \emph{fibration}. The subcategory of cartesian morphisms in this fibration is equivalent to the subcategory of cartesian morphisms of $\int\mathfrak{E}\xrightarrow{\pi_{\mathfrak{E}}}\Ring^{\op}$ from Example~\ref{exa:extensionscalarsgrothendieck} by the comments in Example~\ref{exa:ringsmodulesopfibration}. More precisely, if $N\xrightarrow{\Psi}M$ is a cocartesian lift of $S\xrightarrow{\psi}R$, then there exists a canonical $R$-module isomorphism $M\xrightarrow{\Xi}N\otimes_{S}R$. This module isomorphism defines a cartesian morphism in $\int\mathfrak{E}$.
\end{exa}


\begin{bibdiv}
\begin{biblist}

\bib{AAS}{article}{
   author={Antonini, Paolo},
   author={Azzali, Sara},
   author={Skandalis, Georges},
   title={Flat bundles, von Neumann algebras and $K$-theory with
   $\Bbb{R}/\Bbb{Z}$-coefficients},
   journal={J. K-Theory},
   volume={13},
   date={2014},
   number={2},
   pages={275--303},
   issn={1865-2433},
   doi={10.1017/is014001024jkt253},
}



	
\bib{Ba15}{article}{
  title={Green-hyperbolic operators on globally hyperbolic spacetimes},
  author={B{\"a}r, Christian},
  journal={Commun. Math. Phys.},
  volume={333},
  number={3},
  pages={1585--1615},
  year={2015},
  publisher={Springer},
  doi={10.1007/s00220-014-2097-7},
  url={https://doi.org/10.1007/s00220-014-2097-7},
}

\bib{Ba}{article}{
   author={Basu, Devraj},
   title={$K$-theory with $R/Z$ coefficients and von Neumann algebras},
   journal={$K$-Theory},
   volume={36},
   date={2005},
   number={3-4},
   pages={327--343 (2006)},
   issn={0920-3036},
   doi={10.1007/s10977-006-7110-2},
}



\bib{BG}{article}{
   author={Bunke, Ulrich},
   author={Gepner, David},
   title={Differential function spectra, the differential Becker-Gottlieb transfer, and applications to differential algebraic K-theory},
   journal={arXiv preprint arXiv:1306.0247v2 [math.KT]},
   date={2013},
}


\bib{BNV}{article}{
   author={Bunke, Ulrich},
   author={Nikolaus, Thomas},
   author={V\"{o}lkl, Michael},
   title={Differential cohomology theories as sheaves of spectra},
   journal={J. Homotopy Relat. Struct.},
   volume={11},
   date={2016},
   number={1},
   pages={1--66},
   issn={2193-8407},
   doi={10.1007/s40062-014-0092-5},
}

\bib{BS}{article}{
   author={Bunke, Ulrich},
   author={Schick, Thomas},
   title={Smooth $K$-theory},
   language={English, with English and French summaries},
   journal={Ast\'erisque},
   number={328},
   date={2009},
   pages={45--135 (2010)},
   issn={0303-1179},
   isbn={978-2-85629-289-1},
  }


\bib{BS2}{article}{
   author={Bunke, Ulrich},
   author={Schick, Thomas},
   title={Uniqueness of smooth extensions of generalized cohomology
   theories},
   journal={J. Topol.},
   volume={3},
   date={2010},
   number={1},
   pages={110--156},
   issn={1753-8416},
   doi={10.1112/jtopol/jtq002},
}
	

\bib{Co}{article}{
   author={Connes, Alain},
   title={Noncommutative differential geometry},
   journal={Inst. Hautes \'{E}tudes Sci. Publ. Math.},
   number={62},
   date={1985},
   pages={257--360},
   issn={0073-8301},
}


\bib{CQ}{article}{
   author={Cuntz, Joachim},
   author={Quillen, Daniel},
   title={Algebra extensions and nonsingularity},
   journal={J. Amer. Math. Soc.},
   volume={8},
   date={1995},
   number={2},
   pages={251--289},
   issn={0894-0347},
   doi={10.2307/2152819},
}

\bib{F}{article}{
   author={Freed, Daniel S.},
   title={Dirac charge quantization and generalized differential cohomology},
   conference={
      title={Surveys in differential geometry},
   },
   book={
      series={Surv. Differ. Geom., VII},
      publisher={Int. Press, Somerville, MA},
   },
   date={2000},
   pages={129--194},
 }

\bib{FL}{article}{
   author={Freed, Daniel S.},
   author={Lott, John},
   title={An index theorem in differential $K$-theory},
   journal={Geom. Topol.},
   volume={14},
   date={2010},
   number={2},
   pages={903--966},
   issn={1465-3060},
   doi={10.2140/gt.2010.14.903},
}



\bib{GS1}{article}{
   author={Grady, Daniel},
   author={Sati, Hisham},
   title={Spectral sequences in smooth generalized cohomology},
   journal={Algebr. Geom. Topol.},
   volume={17},
   date={2017},
   number={4},
   pages={2357--2412},
   issn={1472-2747},
   doi={10.2140/agt.2017.17.2357},
}

\bib{GS2}{article}{
   author={Grady, Daniel},
   author={Sati, Hisham},
   title={Twisted differential generalized cohomology theories and their Atiyah--Hirzebruch spectral sequence},
   volume={19},
   ISSN={1472-2747},
   DOI={10.2140/agt.2019.19.2899},
   number={6},
   journal={Algebraic \& Geometric Topology},
   publisher={Mathematical Sciences Publishers},
   year={2019},
   pages={2899--2960},
   url={http://dx.doi.org/10.2140/agt.2019.19.2899},
}

	

\bib{HS}{article}{
   author={Hopkins, M. J.},
   author={Singer, I. M.},
   title={Quadratic functions in geometry, topology, and M-theory},
   journal={J. Differential Geom.},
   volume={70},
   date={2005},
   number={3},
   pages={329--452},
   doi={10.4310/jdg/1143642908},
}

\bib{Ka1}{article}{
   author={Karoubi, Max},
   title={Homologie cyclique et $K$-th\'{e}orie},
   language={French, with English summary},
   journal={Ast\'{e}risque},
   number={149},
   date={1987},
   pages={147},
}

\bib{Ka2}{article}{
   author={Karoubi, Max},
   title={Th\'{e}orie g\'{e}n\'{e}rale des classes caract\'{e}ristiques secondaires},
   language={French, with English summary},
   journal={$K$-Theory},
   volume={4},
   date={1990},
   number={1},
   pages={55--87},
   doi={10.1007/BF00534193},
}

\bib{KeSt74}{inproceedings}{
  author={Kelly, G. Max},
  author={Street, Ross},
  booktitle={Category seminar},
  title={Review of the elements of 2-categories},
  series={Lecture Notes in Mathematics},
  pages={75--103},
  year={1974},
  volume={420},
  publisher={Springer, Berlin, Heidelberg},
  doi={10.1007/BFb0063101},
}
   


\bib{Kl}{book}{
   author={Klonoff, Kevin Robert},
   title={An index theorem in differential K-theory},
   note={Thesis (Ph.D.)--The University of Texas at Austin},
   publisher={ProQuest LLC, Ann Arbor, MI},
   date={2008},
   pages={119},
   isbn={978-0549-70973-2},
}

\bib{La03}{book}{
  title={An introduction to noncommutative spaces and their geometries},
  author={Landi, Giovanni},
  volume={51},
  year={2003},
  publisher={Springer Science \& Business Media},
  doi={10.1007/3-540-14949-X},
}
		

 \bib{Lo}{book}{
   author={Loday, Jean-Louis},
   title={Cyclic homology},
   series={Grundlehren der Mathematischen Wissenschaften [Fundamental
   Principles of Mathematical Sciences]},
   volume={301},
   note={Appendix E by Mar\'\i a O. Ronco},
   publisher={Springer-Verlag, Berlin},
   date={1992},
   pages={xviii+454},
   isbn={3-540-53339-7},
}

\bib{Lott}{article}{
   author={Lott, John},
   title={Secondary analytic indices},
   conference={
      title={Regulators in analysis, geometry and number theory},
   },
   book={
      series={Progr. Math.},
      volume={171},
      publisher={Birkh\"{a}user Boston, Boston, MA},
   },
   date={2000},
   pages={231--293},
   doi={10.1007/978-1-4612-1314-7\_10}, 
}
		
\bib{MaMo12}{book}{
  author={Mac~Lane, Saunders},
  author={Moerdijk, Ieke},
  title={Sheaves in geometry and logic: A first introduction to topos theory},
  year={2012},
  publisher={Springer Science \& Business Media},
  doi={10.1007/978-1-4612-0927-0},
}

\bib{MS}{article}{
   author={Mathai, Varghese},
   author={Stevenson, Danny},
   title={On a generalized Connes--Hochschild--Kostant--Rosenberg theorem},
   journal={Adv. Math.},
   volume={200},
   date={2006},
   number={2},
   pages={303--335},
   issn={0001-8708},
   doi={10.1016/j.aim.2004.11.006},
}

\bib{Ma19}{article}{
   author={Mazel-Gee, Aaron},
   title={A user's guide to co/cartesian fibrations},
   journal={Graduate J. Math},
   volume={4},
   year={2019},
   pages={42--53}
}
   

\bib{MM}{article}{
   author={Minasian, Ruben},
   author={Moore, Gregory},
   title={$K$-theory and Ramond--Ramond charge},
   journal={J. High Energy Phys.},
   date={1997},
   volume={1997},
   number={11},
   pages={002},
   doi = {10.1088/1126-6708/1997/11/002},
}

\bib{MoVa18}{article}{
  title={Monoidal Grothendieck construction},
  author={Moeller, Joe},
  author={Vasilakopoulou, Christina},
  year={2020},
   JOURNAL = {Theory Appl. Categ.},
  volume={35},
  number={31},
  pages={1159--1207},
}


\bib{Mor}{article}{
   author={Morye, Archana S.},
   title={Note on the Serre--Swan theorem},
   journal={Math. Nachr.},
   volume={286},
   date={2013},
   number={2-3},
   pages={272--278},
   issn={0025-584X},
   doi={10.1002/mana.200810263},
}

\bib{Mo01}{book}{
  title={Geometry of differential forms},
  author={Morita, Shigeyuki},
  number={201},
  series={Translations of Mathematical Monographs},
  year={2001},
  publisher={American Mathematical Soc.}
}

	

\bib{NR1}{article}{
   author={Narasimhan, M. S.},
   author={Ramanan, S.},
   title={Existence of universal connections},
   journal={Amer. J. Math.},
   volume={83},
   date={1961},
   pages={563--572},
   issn={0002-9327},
   doi={10.2307/2372896},
}
\bib{NR2}{article}{
   author={Narasimhan, M. S.},
   author={Ramanan, S.},
   title={Existence of universal connections. II},
   journal={Amer. J. Math.},
   volume={85},
   date={1963},
   pages={223--231},
   issn={0002-9327},
   doi={10.2307/2373211},
 }

 \bib{Nes}{book}{
   author={Nestruev, Jet},
   title={Smooth manifolds and observables},
   series={Graduate Texts in Mathematics},
   volume={220},
   publisher={Springer-Verlag, New York},
   date={2003},
   pages={xiv+222},
   isbn={0-387-95543-7},
}


\bib{Pa2}{article}{
   author={Byungdo Park},
   title={A smooth variant of Hopkins--Singer differential $K$-theory},
   journal={New York J. Math.},
   volume={23},
   date={2017},
   number={23},
   pages={655--670},
}


\bib{Pa19}{article}{
  title={Two-dimensional algebra in lattice gauge theory},
  author={Parzygnat, Arthur J.},
  journal={J. Math. Phys.},
  volume={60},
  number={4},
  pages={043506},
  year={2019},
  doi={10.1063/1.5078532},
}







\bib{Qu}{article}{
   author={Quillen, Daniel},
   title={Superconnection character forms and the Cayley transform},
   journal={Topology},
   volume={27},
   date={1988},
   number={2},
   pages={211--238},
   issn={0040-9383},
   doi={10.1016/0040-9383(88)90040-7},
}
	

\bib{Ro}{book}{
   author={Rosenberg, Jonathan},
   title={Algebraic $K$-theory and its applications},
   series={Graduate Texts in Mathematics},
   volume={147},
   publisher={Springer-Verlag, New York},
   date={1994},
   pages={x+392},
   isbn={0-387-94248-3},
}


\bib{Sch}{article}{
  author={Schreiber, Urs},
  title={Differential cohomology in a cohesive infinity-topos},
  journal={arXiv preprint arXiv:1310.7930v1 [math-ph]},
  date={2013},
}

\bib{Se}{article}{
   author={Serre, Jean-Pierre},
   title={Faisceaux alg\'{e}briques coh\'{e}rents},
   journal={Ann. Math.},
   volume={61},
   number={2},
   YEAR = {1955},
   PAGES = {197--278},
   ISSN = {0003-486X},
   DOI = {10.2307/1969915},
}

\bib{Si83}{article}{
  title = {Holonomy, the Quantum Adiabatic Theorem, and {B}erry's Phase},
  author = {Simon, Barry},
  journal = {Phys. Rev. Lett.},
  volume = {51},
  issue = {24},
  pages = {2167--2170},
  year = {1983},
  publisher = {American Physical Society},
  doi = {10.1103/PhysRevLett.51.2167},
}


\bib{SS}{article}{
   author={Simons, James},
   author={Sullivan, Dennis},
   title={Structured vector bundles define differential $K$-theory},
   conference={
      title={Quanta of maths},
   },
   book={
      series={Clay Math. Proc.},
      volume={11},
      publisher={Amer. Math. Soc., Providence, RI},
   },
   date={2010},
   pages={579--599},
}

\bib{St18}{article}{
      author = {Streicher, Thomas},
        title = {Fibred Categories a la Jean Benabou},
     journal={arXiv preprint arXiv:1801.02927v11 [math.CT]},
         year = {2018},
}

\bib{Sw}{article}{
    AUTHOR = {Swan, Richard G.},
     TITLE = {Vector bundles and projective modules},
   JOURNAL = {Trans. Amer. Math. Soc.},
    VOLUME = {105},
      YEAR = {1962},
     PAGES = {264--277},
      ISSN = {0002-9947},
       DOI = {10.2307/1993627},
}
	

\bib{Sz}{article}{
   author={Szabo, Richard J.},
   title={D-branes and bivariant K-theory},
   conference={
      title={Noncommutative geometry and physics. 3},
   },
   book={
      series={Keio COE Lect. Ser. Math. Sci.},
      volume={1},
      publisher={World Sci. Publ., Hackensack, NJ},
   },
   date={2013},
   pages={131--175},
   doi={10.1142/9789814425018\_0005}, 
}
	
\bib{TWZ1}{article}{
  author={Tradler, Thomas},
  author={Wilson, Scott O.},
  author={Zeinalian, Mahmoud},
  title={An elementary differential extension of odd K-theory},
  journal={J. K-Theory},
  volume={12},
  date={2013},
  number={2},
  pages={331--361},
  issn={1865-2433},
  doi={10.1017/is013002018jkt218},
  }
  


  \bib{TWZ2}{article}{
  author={Tradler, Thomas},
  author={Wilson, Scott O.},
  author={Zeinalian, Mahmoud},
  title={Differential K-theory as equivalence classes of maps to Grassmannians and unitary groups},
  journal={New York J. Math.},
  volume={22},
  date={2016},
  number={2},
  pages={527-581},
  issn={1076-9803},
  }
   
   
\bib{Tr16}{book}{
  title={Topological Vector Spaces, Distributions and Kernels},
  author={Treves, Francois},
  series={Pure and Applied Mathematics},
  volume={25},
  year={2016},
  publisher={Elsevier}
}


\bib{Wi}{article}{
   author={Witten, Edward},
   title={D-branes and $K$-theory},
   journal={J. High Energy Phys.},
   date={1999},
   volume={1998},
   number={12},
   pages={019},
   doi={10.1088/1126-6708/1998/12/019},
}
\end{biblist}
\end{bibdiv}
\end{document}